\documentclass[final
]{scrartcl}
\pdfoutput=1
\usepackage[utf8]{inputenc}
\usepackage{hyperref}
\usepackage[english]{babel}
\usepackage{enumerate}
\usepackage{aliascnt}
\usepackage{amssymb}
\usepackage{amsmath}
\usepackage{amsthm}
\usepackage{verbatim}
\usepackage[normalem]{ulem}
\usepackage{mathtools}
\usepackage{esint} % Für Durchschnittsintegral

\usepackage[noabbrev]{cleveref}

\usepackage{tikz}
\usepackage{tikzscale}
\usepackage{graphicx}
\usepackage[labelformat=simple]{subcaption}
\usepackage{mathscinet} % Fuer Russische Buchstaben

\usepackage[nottoc]{tocbibind} % Section number von Bibliography verhindern
\usepackage[title]{appendix}

\usepackage[notcite,notref]{showkeys}
\usepackage{color}

\numberwithin{equation}{section} % Gleichungen nach Section nummerieren

\usepackage{dsfont} %f\"ur Indikatorfunktion \mathds{1}
\usepackage{hyperref} %Links, autoref

\usepackage{autonum} % Load last!

\allowdisplaybreaks[4] %0-4

\theoremstyle{plain}

\newtheorem{prop}{Proposition}[section]

\newaliascnt{lem}{prop} 

\newtheorem{lem}[lem]{Lemma}%[section]

\aliascntresetthe{lem} 
\Crefname{lem}{Lemma}{Lemmas}

\newaliascnt{thm}{prop} 
\newtheorem{thm}[thm]{Theorem}%[section]

\aliascntresetthe{thm} 

\newaliascnt{cor}{prop} 
\newtheorem{cor}[cor]{Corollary}%[section]

\aliascntresetthe{cor}

\newaliascnt{hyp}{prop} 
\newtheorem{hyp}[lem]{Hypothesis}

\aliascntresetthe{hyp} 

\theoremstyle{definition}

\newaliascnt{defi}{prop} 
\newtheorem{defi}[defi]{Definition}%[section

\aliascntresetthe{defi} 
\Crefname{defi}{Definition}{Definitions}

\newaliascnt{problem}{prop} 
\newtheorem{problem}[problem]{Problem}%[section]

\aliascntresetthe{problem} 

\newaliascnt{example}{prop} 
\newtheorem{example}[example]{Example}%[section]

\aliascntresetthe{example} 

\theoremstyle{remark}

\newaliascnt{remark}{prop} 
\newtheorem{remark}[remark]{Remark}%[section]

\aliascntresetthe{remark} 

\def\equationautorefname~#1\null{%
	(#1)\null
}

\DeclareMathOperator{\spt}{spt} 
\DeclareMathOperator{\diam}{diam} 
\DeclareMathOperator{\lin}{span} 
\newcommand{\R}{\ensuremath{\mathbb{R}}}
\newcommand{\N}{\ensuremath{\mathbb{N}}}
\newcommand{\Z}{\ensuremath{\mathbb{Z}}}
\renewcommand{\S}{\ensuremath{\mathbb{S}}}

\newcommand*\diff{\mathop{}\!\mathrm{d}}

\newcommand{\defeq}{\vcentcolon=}
\newcommand{\eqdef}{=\vcentcolon}
\newcommand{\CalE}{\ensuremath{\mathcal{E}}}

\newcommand{\CalF}{\ensuremath{\mathcal{F}}}
\newcommand{\CalL}{\ensuremath{\mathcal{L}}}
\newcommand{\CalA}{\ensuremath{\mathcal{A}}}

\newcommand{\CalW}{\ensuremath{\mathcal{W}}}

\newcommand{\CalH}{\ensuremath{\mathcal{H}}}

\newcommand{\abs}[1]{\ensuremath{\lvert #1\rvert}}
\newcommand{\Abs}[1]{\ensuremath{\left\lvert #1\right\rvert}}

\newcommand{\norm}[2]{\ensuremath{\Vert #1 \Vert_{#2}}}

\DeclareMathOperator{\Tan}{Tan}

\DeclareMathOperator{\divergence}{div}

\DeclareMathOperator{\CalV}{\ensuremath{\mathcal{V}}}

\makeatletter
\renewcommand\p@subfigure{\thefigure.}
\makeatother

\title{Li--Yau inequalities for the Helfrich functional and applications}
	\author{Fabian Rupp\thanks{Faculty of Mathematics, University of Vienna, Oskar-Morgenstern-Platz 1, 1090 Vienna, Austria. \texttt{fabian.rupp@univie.ac.at}} \and Christian Scharrer\thanks{Institute for Applied Mathematics, University of Bonn, Endenicher Allee 60, 53115, Bonn, Germany. \texttt{scharrer@iam.uni-bonn.de}}}

\begin{document}
	\maketitle
	\begin{abstract}
		\noindent\textbf{Abstract:}
	We prove a general Li--Yau inequality for the Helfrich functional where the spontaneous curvature enters with a singular volume type integral. In the physically relevant cases, this term can be converted into an explicit energy threshold that guarantees embeddedness. We then apply our result to the spherical case of the variational Canham--Helfrich model. If the infimum energy is not too large, we show existence of smoothly embedded minimizers. Previously, existence of minimizers was only known in the classes of immersed bubble trees or curvature varifolds.
	\end{abstract}
	
	

	\section{Introduction}

\emph{Lipid bilayers} make up the cellular membranes of most organisms. These extremely thin structures commonly form vesicles, so mathematically they are naturally modeled as two-dimensional structures, i.e.\ closed surfaces. The \emph{Canham--Helfrich model} \cite{Canham,Helfrich} characterizes the equilibrium shapes of lipid bilayers as (constrained) minimizers of a curvature dependent bending energy. For an oriented surface $\Sigma$ and an immersion $f\colon \Sigma\to\R^3$, the \emph{Helfrich energy} is defined by
\begin{align}\label{eq:defHelfrich}
	\CalH_{c_0}(f) \defeq\frac{1}{4}\int_{\Sigma}\abs{H-c_0n}^2\diff \mu.
\end{align}
Here, $H=H_f$ is the mean curvature vector of the immersion, i.e.\ the
trace of the second fundamental form,
$n=n_f$ is the unit normal induced by the orientation of $\Sigma$ (see \eqref{eq:def normal} below) and $\mu=\mu_f$ denotes the Riemannian measure associated to the pullback metric $g=g_f=f^{*}\langle\cdot,\cdot\rangle$ on $\Sigma$, where $\langle\cdot,\cdot\rangle$ denotes the Euclidean inner product in $\R^3$. The constant $c_0\in \R$ is called \emph{spontaneous curvature}. 
Since $H$ is normal to the surface, the Helfrich energy can also be written as
\begin{align}\label{eq:defHelfrich2}
	\CalH_{c_0}(f) =\frac{1}{4}\int_{\Sigma}(H_{\mathrm{sc}}-c_0)^2\diff \mu,
\end{align}
where $H_{\mathrm{sc}}\defeq \langle H, n\rangle$ is the scalar mean curvature with respect to $n$. We choose the inner unit normal such that the standard embedding of the round sphere of radius $r>0$ has positive scalar mean curvature $H_{\mathrm{sc}} = \frac{2}{r} >0$. In particular, if $c_0>0$, the Helfrich energy is zero for a round sphere of appropriately chosen radius.
 Reversing the orientation on $\Sigma$ corresponds precisely to replacing $n$ by $-n$. Thus, we have 
\begin{align}\label{eq:Helfrich orient transform}
	\CalH_{c_0}(f) = \CalH_{-c_0}(\hat{f}),
\end{align}
where $\hat{\Sigma}$ is the surface $\Sigma$ with reversed orientation and $\hat{f}\colon \hat{\Sigma}\to\R^3, \hat{f}(p)=f(p)$. Clearly, the Helfrich functional is not scale-invariant. However, we observe the following scaling property involving both arguments:
\begin{equation}
	\CalH_{c_0}(f) = \CalH_{\frac{c_0}{r}}(rf)\qquad\text{for $r>0$.}
\end{equation} 
In particular, we see that
\begin{equation} \label{eq:intro:convergence_Helfrich_to_Willmore}
	\lim_{r\to0+}\CalH_{c_0}(rf) = \lim_{r\to0+}\CalH_{rc_0}(f) = \CalH_{0}(f) \eqdef \CalW(f).
\end{equation}
The right hand side is well known as the \emph{Willmore energy}. In contrast to the Helfrich functional, the Willmore functional $\CalW$ is scale-invariant and does not depend on the unit normal field $n$ or the orientation of the underlying surface $\Sigma$.
One may also consider the $L^2$-CMC-deficit
\begin{equation}\label{eq:intro:CMC deficit}
	\bar\CalH(f)\defeq \inf_{c_0\in\R}\CalH_{c_0}(f) = \frac{1}{4}\int_\Sigma(H_{\mathrm{sc}} - \bar H_{\mathrm{sc}})^2\diff\mu,
\end{equation}
where $\bar H_{\mathrm{sc}} \defeq \fint_\Sigma H_{\mathrm{sc}}\diff\mu$ is the average scalar mean curvature. Also the functional $\bar \CalH$ is scale-invariant and does not depend on the orientation of $\Sigma$. For more details and corresponding results, see \Cref{sec:scale invariant} and \Cref{subsec:pos tot mean}. 

We are primarily interested in the following minimization problem suggested by Canham~\cite{Canham} and Helfrich~\cite{Helfrich} in order to study the shape of red blood cells. Our main contribution is stated in \Cref{thm:regularity_Helfrich_problem} below (see also \Cref{thm:regularity}).
\begin{problem}\label{problem:CH}
	Let $c_0\in\R$ and $A_0,V_0>0$ be given constants. Let the unit sphere $\S^2\subset\R^3$ be oriented by the inner unit normal. Minimize the functional $\CalH_{c_0}$ in the class of smooth embeddings $f\colon\S^2\to\R^3$ subject to the constraints
	\begin{equation}
		\CalA(f)\defeq\int_{\S^2} 1\diff\mu = A_0,\qquad \CalV(f)\defeq-\frac{1}{3}\int_{\S^2}\langle f,n\rangle \diff\mu = V_0.
	\end{equation}
\end{problem}
We consider the following example of \Cref{problem:CH} where the infimum cannot be attained by a smooth embedding, cf.\ \cite{MondinoScharrer1}.
\begin{example}\label{example:bubble-tree}
	Let $\iota_{\S^2}\colon \S^2\to\R^3$ be the inclusion of the unit sphere.
	Let $c_0 \defeq 2$, $A_0 \defeq 2\CalA(\iota_{\S^2})$, and $V_0\defeq2\CalV(\iota_{\S^2})$. There exists a sequence of smooth embeddings $f_k\colon \S^2\to \R^3$ satisfying $\CalA(f_k) = A_0$ and $\CalV(f_k) = V_0$ which converges in the varifold topology to the set $T\subset\R^3$ given by two translations of the unit sphere that meet in exactly one point (see Figure~\ref{fig:touching_spheres}) such that
	\begin{equation}\label{eq:intro:energy_convergence}
		\lim_{k\to\infty}\CalH_{c_0}(f_k) = 2\CalH_{c_0}(\iota_{\S^2}) = 0.
	\end{equation}
	To see this, let $\Sigma_{\ell,r}$ be the spherical $C^{1,1}$-regular surface that results by gluing two spherical caps at the ends of a cylinder of length $\ell\geq0$ with radius $r>0$. Denote with $\Sigma_{0,\ell,r}$ the disjoint union of $\Sigma_{\ell,1}$ with $\Sigma_{0,r}$ (a sphere with radius $r$) and with $\Sigma_{a,\ell,r}$ the spherical $C^{1,1}$-regular surface that results by connecting $\Sigma_{\ell,1}$ with $\Sigma_{0,r}$ through a catenoidal bridge of small neck size $a>0$, cf.\ Figure \ref{fig:dumbbell}. The sequence $(\Sigma_{k^{-1},0,1})_{k\in\N}$ satisfies \eqref{eq:intro:energy_convergence}. However, the gluing only gives $C^{1,1}$-regularity and the conditions on area and volume are not met. We will first adjust the isoperimetric ratio defined by $\mathcal I \defeq \mathcal A^3/\mathcal V^2$. A short computation reveals that $I_{a,\ell,r}\defeq\mathcal I(\Sigma_{a,\ell,r})$ satisfies
	\begin{equation}\label{eq:intro:iso_ratios}
		I_{0,\ell,1} > I_{0,0,1} \eqdef I_0 = A_0^3/V_0^{2} > I_{0,0,r}\qquad \text{for $\ell>0$ and $0<r<1$}.
	\end{equation}
	Moreover, $I_{a,\ell,r}$ depends continuously on $a\geq0$, $\ell\geq0$ and $r>0$. Hence, if $k\in\N$ and $I_{k^{-1},0,1} > I_0$, we can first choose $1-k^{-1}<r_k<1$ such that still $I_{k^{-1},0,r_k} > I_0$ and then $0<a_k<k^{-1}$ with $I_{a_k,\ell_k,r_k} = I_0$ where $\ell_k = 0$. If on the other hand $k\in\N$ and $I_{k^{-1},0,1} < I_0$, we can first choose $0<\ell_k<k^{-1}$ such that still $I_{k^{-1},\ell_k,1} < I_0$ and then $0<a_k < k^{-1}$ with $I_{a_k,\ell_k,r_k} = I_0$ where $r_k = 1$. Now, we let $S_k = \Sigma_{a_k,\ell_k,r_k}$ and choose $\lambda_k>0$ such that $\CalA(\lambda_k S_k) = A_0$ and $\CalV(\lambda_k S_k) = V_0$. Then, since $a_k\to 0$, $\ell_k\to 0$, and $r_k\to 1$ as $k\to \infty$, there holds $\CalA(S_k)\to A_0$, so $\CalV(S_k)\to V_0$ and $\lambda_k\to 1$. It follows that also the sequence $(\lambda_kS_k)_{k\in\N}$ satisfies \eqref{eq:intro:energy_convergence}.
    
\begin{figure}[h]
	\centering
	\caption{Visualization of the construction.    \label{fig:dumbbell}}
  	\vspace{0.5cm}	
	\includegraphics[width=0.4\textwidth]{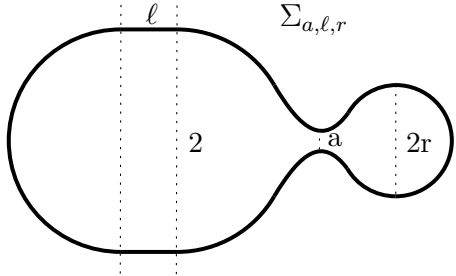}
\end{figure} 

	In order to have $C^\infty$-regularity, one can apply mollifications near the patching regions and subsequently compensate possible changes of area and volume using \cite[Lemma~2.1]{MondinoScharrer2} supported on the catenoid away from the mollifiers. 
		
	Consequently, the infimum of \Cref{problem:CH} is attained by the set $T$. Clearly, $T$ cannot be written as the image of a smooth immersion of $\S^2$. If on the other hand $f:\S^2\to \R^3$ is a smooth immersion with $\CalH_{c_0}(f) = 0$ then, by a result of Hopf~\cite[Chapter VI, Theorem 2.1]{Hopf}, the image of~$f$ must be a round sphere. In particular, if $f\colon \S^2 \to \R^3$ is a smooth embedding with $\CalH_{c_0}(f) = 0$, then $\CalA(f) \neq A_0$ and $\CalV(f)\neq V_0$. 
\end{example}  
In the terminology of \cite{MondinoRiviereACV,MondinoScharrer1}, the set $T$ in \Cref{example:bubble-tree} is the \emph{bubble tree} consisting of two unit spheres. Bubbling phenomena have also been observed in nature and are known as \emph{budding transition}~\cite{SeifertBerndlLipowsky}. Thus, the space of bubble trees appears to be a natural class in which to minimize the Helfrich functional. Indeed, in \cite[Theorem~1.7]{MondinoScharrer1}, the existence of minimizers for the Helfrich functional in the class of immersed bubble trees was proven. Each of the bubbles is given by a map $\S^2\to\R^3$ which outside of finitely many so called \emph{branch points} is a smooth immersion. For similar results, see \cite{ChoksiVeneroni,EichmannAGAG}. However, not all minimizers of \Cref{problem:CH} are necessarily bubble trees, consider for instance the case $c_0=2$, $A_0=\CalA(\iota_{\S^2})$, $V_0 = \CalV(\iota_{\S^2})$. One may conjecture that bubbling can only occur if the parameters $A_0$ and $V_0$ are within a certain range depending on $c_0$. Apart from the geometric relevance to obtain such qualitative results for the minimizers of \Cref{problem:CH}, 
it is of great interest to confirm mathematically that the Canham--Helfrich model is suitable for the study of red blood cells which are actually embedded --- rather than a bubble tree. As a first step in this direction it was proven in \cite{MondinoScharrer1} that there exists a constant $\varepsilon = \varepsilon(A_0,V_0)>0$ such that the minimizers are given by smooth embeddings provided $|c_0|<\varepsilon$. However, apart from the fact that $\varepsilon(A_0,V_0)$ is implicitly small, one would rather want to have a criterion of the following type: \emph{For all $c_0\in\R$, the \Cref{problem:CH} has a solution provided $A_0$ and $V_0$ are in a certain explicit range depending on $c_0$}. 

The proof of embeddedness of minimizers in~\cite{MondinoScharrer1} is based on the fact that for $|c_0|$ small, the Helfrich functional is close to the Willmore functional, see \eqref{eq:intro:convergence_Helfrich_to_Willmore}, 
and minimizers for $c_0=0$ are given by smooth embeddings, see~\cite{Schygulla}. A crucial tool to prove smoothness and embeddedness of the minimizers in~\cite{Schygulla} (i.e.\ solutions of \Cref{problem:CH} for $c_0=0$) is the following inequality of Li and Yau \cite[Theorem 6]{LiYau}. If $\Sigma$ is \emph{closed} (i.e.\ compact and without boundary), for any $x_0\in \R^3$ we have 
\begin{align}\label{eq:LY Willmore}
	\CalH^{0}(f^{-1}\{x_0\}) \leq \frac{1}{4\pi}\CalW(f),
\end{align}
where $\CalH^{0}$ denotes the counting measure. In particular, if $\CalW(f)<8\pi$, then $f$ must be an embedding. This observation 
also played an essential role in the study of the Willmore energy and related topics, 
cf.\ for instance \cite{SimonWillmore,KSRemovability,Schygulla,KuwertLiCAG,RiviereCrelle,KellerMondinoRiviere,RuppIso}.

In view of the fact that branch points have multiplicity at least $2$, such a tool could be the key to exclude bubbling in the Canham--Helfrich model. Apart from comparing the Helfrich energy with the Willmore energy for small $|c_0|$ via~\eqref{eq:intro:convergence_Helfrich_to_Willmore}, one might also try to make use of the Li--Yau inequality~\eqref{eq:LY Willmore} by estimating the Willmore energy from above in terms of the Helfrich energy (see \eqref{eq:bound_Willmore_by_Helfrich}):
\begin{equation}\label{eq:intro:bound_Willmore_by_Helfrich}
	\CalW(f) \leq 2\CalH_{c_0}(f) + \frac{1}{2}c_0^2\CalA(f).
\end{equation}
Again, if the right hand side is strictly less than $8\pi$, then the Li--Yau inequality~\eqref{eq:LY Willmore} implies that $f$ is an embedding. However, as one of our results reveals (see \Cref{thm:Willmore-inequality}), in the case of red blood cells where $c_0<0$ (see \cite{DeulingHelfrich}), there holds $\CalH_{c_0}(f)>4\pi$. 
In particular, the right hand side of \eqref{eq:intro:bound_Willmore_by_Helfrich} is already strictly larger than $8\pi$ and one cannot apply the Li-Yau inequality~\eqref{eq:LY Willmore} to deduce embeddedness of $f$.

Another naive attempt to apply \eqref{eq:LY Willmore} would be to show that $\CalW \leq \CalH_{c_0}$ for $c_0<0$. However, this is impossible by the following simple scaling argument. Let $f\colon \mathbb{S}^2\to\R^3$ be an immersion such that $\int_{\mathbb{S}^2}H_{\mathrm{sc}}\diff\mu<0$ (such an $f$ exists by \cite[Theorem 1.2]{DHMT16}). We find that 
\begin{align}\label{eq:Helfrich_below_Willmore}
\CalH_{c_0}(rf)-\CalW(rf) = -\frac{rc_0}{2}\int_{\mathbb{S}^2} H_{\mathrm{sc}}\diff \mu + \frac{r^2 c_0^2}{4}~\CalA(f)
\end{align}
which becomes negative if $r>0$ is sufficiently small as $c_0<0$.

\subsection{Main results}

Instead of applying~\eqref{eq:LY Willmore} by comparing the Helfrich energy with the Willmore energy, the aim of this article is to prove and apply a Li--Yau type inequality directly for the Helfrich functional. In the smooth setting, our multiplicity inequality reads as follows.

\begin{lem}\label{thm:LY smooth}
	Let $f\colon \Sigma\to\R^3$ be a smooth proper immersion of an oriented surface $\Sigma$ without boundary. Let $c_0\in\R$, $x_0\in \R^3$, 
	and suppose that the \emph{concentrated volume of $f$ at $x_0$} defined by
	\begin{align}\label{eq:def conc vol}
		\CalV_{c}(f, x_0) \defeq - \int_{\Sigma}\frac{\langle f-x_0, n\rangle}{\abs{f-x_0}^2}\diff\mu
	\end{align}
	exists. Then
	\begin{align}\label{eq:LY Immersion}
		\CalH^0(f^{-1}\{x_0\}) \leq \limsup_{\rho\to\infty} \frac{\mu(f^{-1}(B_\rho(x_0)))}{\pi\rho^2}+ \frac{1}{4\pi} \CalH_{c_0}(f) + \frac{c_0}{2\pi}\CalV_c(f,x_0).
	\end{align}
\end{lem}

In order to apply \Cref{eq:LY Immersion}, it is of crucial interest to determine the sign of the concentrated volume.
Despite singular, the integrand in \eqref{eq:def conc vol} is subcritical and locally integrable, see \Cref{lem:subcrit_integral} and \Cref{rem:subcrit_integral} below. Moreover, the integrand is nonpositive if $f[\Sigma]$ parametrizes the boundary of an open set in $\R^3$ which is star-shaped with respect to $x_0$ and $n$ is the inner unit normal, cf.\ \cite[9.4.2]{Evans}. However, such an immersion $f$ must be embedded a priori.

It turns out that the sign of the concentrated volume can be determined if we can find a suitable notion of inner unit normal, resulting in an appropriate divergence theorem.

\begin{defi}\label{defi:Alexandrov_immersions}
	We call a smooth immersion $f\colon \Sigma\to\R^3$ of a closed surface $\Sigma$ an \emph{Alexandrov immersion}, cf.\ \cite{Alexandrov}, if there exist a smooth compact 3-manifold $M$ with boundary $\partial M=\Sigma$, a smooth inner unit normal field $\nu$ to $\Sigma$ and a smooth immersion $F\colon M\to\R^3$ such that $F\vert_{\Sigma}=f$. The surface~$\Sigma$ is then necessarily orientable.
	Moreover, we choose the orientation on $\Sigma$ such that the induced normal field along $f$ (see \eqref{eq:def normal} below) satisfies $n=\diff F(\nu)$.
\end{defi}

Our orientation on $\Sigma$ does not coincide with the usual Stokes orientation. The reason for this is that we want to work with the inner unit normal such that the standard embedding of a round sphere has positive scalar mean curvature.

In the setting of \Cref{defi:Alexandrov_immersions}, the Li--Yau inequality \eqref{eq:LY Immersion} can be put into the following more convenient form.

\begin{thm}\label{thm:LY Alexandrov}
	Let $\Sigma$ be a closed surface and let $f\colon \Sigma \to \R^3$ be an Alexandrov immersion with  $f=F\vert_{\Sigma}$, $F\colon M\to\R^3$ as in \Cref{defi:Alexandrov_immersions}. Then for all $x_0\in \R^3$ we have
	\begin{align}
		\CalH^{0}(f^{-1}\{x_0\})\leq \frac{1}{4\pi} \CalH_{c_0}(f) + \frac{c_0}{2\pi}\int_{F[M]} \frac{\CalH^{0}(F^{-1}\{x\})}{\abs{x-x_0}^{2}}\diff \CalL^3(x).
	\end{align} 
	In particular, in case $c_0\leq 0$ we infer
	\begin{align}\label{eq:LYHelfrich}
		\CalH^{0}(f^{-1}\{x_0\})\leq \frac{1}{4\pi}\CalH_{c_0}(f).
	\end{align}
\end{thm} 

Due to round spheres, the above extension of \eqref{eq:LY Willmore} can only hold if $c_0\leq 0$ and $n$ is the inner unit normal. Of course, in view of \eqref{eq:Helfrich orient transform} we could simply reverse the orientation on $\Sigma$, but this will generically make it impossible to find an Alexandrov immersion where $M$ in \Cref{defi:Alexandrov_immersions} is compact. 

As a key application of our Li--Yau inequalities, we prove the following contribution to \Cref{problem:CH} based on the previous result in~\cite{MondinoScharrer1}.

\begin{thm}\label{thm:regularity_Helfrich_problem}
	Let $c_0\in \R$ and suppose $A_0,V_0>0$ satisfy the isoperimetric inequality $36\pi V_0^2\leq A_0^3$. 
	Set
	\begin{align}\label{eq:def eta}
	\eta(c_0, A_0, V_0)\defeq \inf\{ \CalH_{c_0}(f)\,|\,f\in C^{\infty}(\S^2;\R^3)\text{\,embedding,\,}\CalA(f)=A_0,\CalV(f)=V_0\}.
	\end{align}	
	There exists $\Gamma(c_0, A_0, V_0) \in \R$ such that if
	\begin{align}
		\eta(c_0, A_0, V_0) < 8\pi + \Gamma(c_0, A_0, V_0)
	\end{align}
	then the infimum in \eqref{eq:def eta} is attained. Moreover, there holds
	\begin{align}
		\Gamma(c_0, A_0, V_0)\geq \begin{cases}
			4\pi\left(\sqrt{1 + \frac{|c_0|V_0}{2\cdot 9^2 (A_0 + \frac{2}{3}|c_0|V_0)}} - 1\right) & \text{ if }c_0< 0, \\
			-6c_0(4\pi^2V_0)^\frac{1}{3}& \text{ if }c_0\geq 0.
		\end{cases}
	\end{align}			
	In particular $\Gamma(c_0, A_0, V_0)>0$ for $c_0<0$ and for any $c_0\leq 0$ there exist $A_0, V_0>0$ with $\eta(c_0, A_0, V_0)<8\pi$.
\end{thm}
	The explicit estimates of the constant $\Gamma(c_0, A_0, V_0)$ are due to further geometric applications of our Li--Yau inequality, see \Cref{thm:upper_diameter_bound} and \Cref{rem:regularity}.
As a consequence of \Cref{thm:regularity_Helfrich_problem}, the only missing step to exclude bubbling in \Cref{problem:CH} are estimates from above for $\eta(c_0, A_0, V_0)$. 
This is an interesting problem to be addressed in future research. One idea is to construct example surfaces that can be used to derive these bounds numerically.
 
\subsection{A suitable setup for the Li--Yau inequalities}

We now discuss the different notions of (generalized) surfaces that we want to prove and apply Li--Yau inequalities for. 
There are four key points to be considered. 
\begin{enumerate}[(i)]
	\item \label{item:unit_normal} In order to even define the Helfrich energy, the surface needs to have a unit normal vector field. In the smooth case, this naturally means that the surface is orientable.
	\item \label{item:double_points} One of the main applications of the classical Li--Yau inequality for the Willmore functional is to deduce embeddedness of immersions whose energy lies strictly below~$8\pi$. Therefore, the Li--Yau inequality should hold for surfaces that are not already embedded a priori, i.e.\ we want to allow for multiplicity points.
	\item \label{item:sign_volume} In order to actually apply the Li--Yau inequality for the Helfrich energy (see \Cref{thm:LY smooth}), it is necessary to determine the sign of the concentrated volume~\eqref{eq:def conc vol}. A sufficient tool to do so would be a divergence theorem.
	\item \label{item:compactness} Another important application of the classical Li--Yau inequality is to infer regularity and embeddedness of minimizers. It is therefore of interest to prove Li--Yau inequalities for weak surfaces that have good compactness properties.
\end{enumerate}

\paragraph{Oriented varifolds} The most general notion of surface that comprises all shapes shown in Figures~\ref{fig:touching_spheres}-\ref{fig:torus_and_sphere} and that naturally satisfies Items~\eqref{item:unit_normal}, \eqref{item:double_points}, and \eqref{item:compactness} are \emph{oriented varifolds,} cf.\ \cite{Hutchinson}. They generalize the idea of immersed submanifolds and allow for a generalized concept of mean curvature. Since they also possess strong compactness properties,
they have already been applied in several variational settings for the Canham--Helfrich model, see \cite{EichmannCalcVar,EichmannAGAG,BrazdaLussardiStefanelli}. 
Our most general version of the Li--Yau inequality for oriented varifolds, \Cref{thm:LY varifold}, is also applicable if the first variation has a nontrivial singular part $\beta$ (see 
\Cref{hyp:generalised_mean_curvature_bdr}). The reason for this generality is that we would like the Li--Yau inequality to be applicable for surfaces like the one shown in Figure~\ref{fig:lens}, see also \Cref{ex:lens}. Moreover, the Li--Yau inequality can then also be applied in the context of boundary problems, see \cite{EichmannCalcVar}. These are naturally formulated using \emph{curvature varifolds with boundary,} cf.\ \cite{Mantegazza}.

\paragraph{Alexandrov immersions}
In nature, one expects the principle of noninterpenetration of matter to hold true. As for vesicles that means there is a clearly defined inside. Nevertheless, membranes can be squeezed together as in Figures~\ref{fig:sausage}, \ref{fig:two_tears}, and~\ref{fig:pillow}. In order to satisfy a divergence theorem, a surface should possess a well defined inside.
In the smooth case, the so called \emph{Alexandrov immersions} (see \cite{Alexandrov} and \Cref{defi:Alexandrov_immersions} above) do satisfy a divergence theorem, see \Cref{lem:Alexandrov_immersions} below. They allow for multiplicity points as shown in Figures~\ref{fig:sausage} and~\ref{fig:Alexandrov}. Moreover, since the underlying $3$-manifold of an Alexandrov immersion does not have to be connected, they also allow for multiplicity points that arise from two touching surfaces as shown in Figure~\ref{fig:touching_spheres} or even two intersecting surfaces as shown in Figure~\ref{fig:intersecting_spheres}. However, the rotationally symmetric surface in Figure~\ref{fig:two_tears} is not an Alexandrov immersion. The Li--Yau inequality for Alexandrov immersions is stated in \Cref{thm:LY Alexandrov}.

\paragraph{Sets of finite perimeter}
A nonsmooth notion of surfaces that satisfy a divergence theorem are the boundaries of \emph{sets of finite perimeter,} cf.\ \cite[Chapter 5]{EvansGariepy}. As opposed to Alexandrov immersions, they allow for multiplicity points as shown in Figure~\ref{fig:two_tears} but they do not allow for the multiplicity points in Figure~\ref{fig:Alexandrov}. Sets of finite perimeter do have good compactness properties. Moreover, they comprise nonsmooth objects as shown in Figure~\ref{fig:lens} and discussed in \Cref{ex:lens}. In \Cref{sec:finiteperimeter}, we introduce a weak notion of Alexandrov immersions, the \emph{varifolds with enclosed volume}. Their underlying $3$-dimensional structure is a sequence of decreasing sets of finite perimeter.  
They allow for multiplicity points as in Figures~\ref{fig:touching_spheres}--\ref{fig:pillow} and Figure~\ref{fig:intersecting_spheres} and still satisfy a divergence theorem. The corresponding Li--Yau inequality is stated in \Cref{cor:LY finite perimeter}.

\paragraph{Currents}
Another important class of surfaces that naturally satisfies Items~\eqref{item:unit_normal} and~\eqref{item:compactness} above are \emph{currents} (see \cite[Chapter~4]{Federer}). A downside of this concept is that the current associated with the immersion of Figure~\ref{fig:pillow} corresponds to the surface shown in Figure~\ref{fig:current}. More precisely, a current induced by an immersion with a given unit normal field looses information about multiplicity points that arise by overlapping where the sum of the unit normal vectors vanishes, cf.\ also \eqref{eq:comparing_normal_vectors}. As a consequence, the varifold corresponding to the surface in Figure~\ref{fig:current} has a nontrivial singular part while the varifold corresponding to the immersion of Figure~\ref{fig:pillow} has no singular part. 

\paragraph{Lipschitz quasi-embeddings} In \Cref{sec:regularity}, we introduce a concept to model cellular membranes as weak immersions that can only self intersect tangentially. This is a new concept inspired by the previously developed \emph{Lipschitz immersions} of Rivi{\`e}re~\cite{RiviereCrelle}. The resulting class of surfaces is termed \emph{Lipschitz quasi-embeddings} and satisfies Items~\eqref{item:unit_normal}--\eqref{item:compactness} above.
They describe cellular shapes and comprise the surfaces in Figures~\ref{fig:sausage}, \ref{fig:two_tears} and \ref{fig:pillow}, but do not allow for interpenetration as in \ref{fig:Alexandrov}. It turns out that this class is well-suited for the variational discussion of the spherical Canham--Helfrich model which is why we rely on it for the proof of \Cref{thm:regularity_Helfrich_problem}.
\bigskip

The only kind of surface where the sign of the concentrated volume cannot be determined in general are those surfaces where the unit normal vector field changes between inner and outer, see Figure~\ref{fig:torus_and_sphere} and \Cref{ex:volume}. These are surfaces where interpenetration necessarily happens.

\begin{figure}[h]
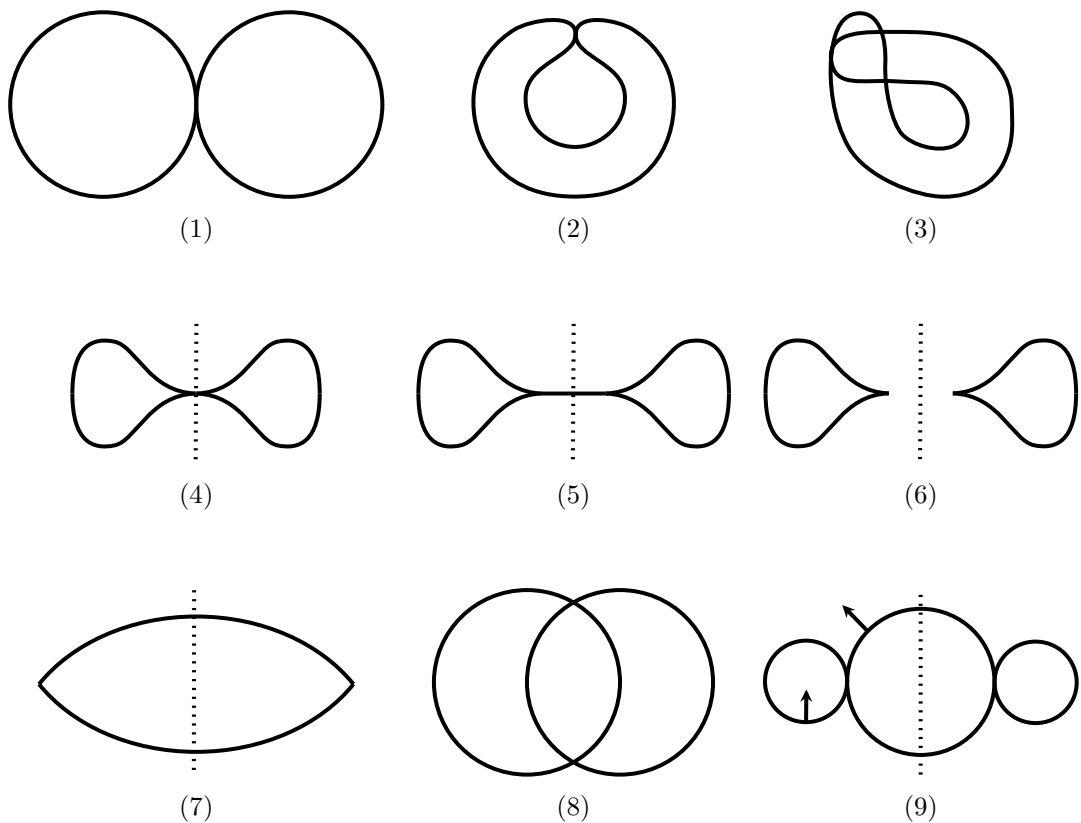

	\centering
	\caption{Profiles of surfaces with different types of multiplicity points. Dotted lines indicate rotationally symmetric surfaces.\label{fig:intersections}}
	\vspace{0.5cm}	
	\addtocounter{figure}{-1}
	\begin{tabular}{ccc}
		\subfloat[]{\centering
			\label{fig:touching_spheres}\includegraphics[height=2.5cm]{Pictures/touching_spheres.tikz}}
		&
		\subfloat[]{\centering
			\label{fig:sausage}\includegraphics[height=2.5cm]{Pictures/sausage.tikz}}
		&
		\subfloat[]{\centering\label{fig:Alexandrov}\includegraphics[height=2.5cm]{Pictures/Alexandrov.tikz}} 
		\vspace{1cm}
		\\
		\subfloat[]{\centering
			\label{fig:two_tears}\includegraphics[height=1.9cm]{Pictures/two_tears.tikz}}
		&
		\subfloat[]{\centering
			\label{fig:pillow}\includegraphics[height=1.9cm]{Pictures/pillow.tikz}}
		&
		\subfloat[]{\centering
			\label{fig:current}\includegraphics[height=1.9cm]{Pictures/current.tikz}}
		\vspace{1cm}
		\\
		\subfloat[]{\centering
			\label{fig:lens}\includegraphics[height=2.5cm]{Pictures/lens.tikz}}
		&
		\subfloat[]{\centering
			\label{fig:intersecting_spheres}\includegraphics[height=2.5cm]{Pictures/intersecting_spheres.tikz}}
		&
		\subfloat[]{\centering
			\label{fig:torus_and_sphere}\includegraphics[height=2.5cm]{Pictures/torus_and_sphere.tikz}}
	\end{tabular}
\end{figure}

\subsection{Structure of this article}
After a brief discussion of the geometric and measure theoretic background in \Cref{sec:preliminaries}, we examine the concentrated volume and its properties in \Cref{sec:volume}. This includes Hölder continuity in $x_0\in \R^3$ and continuity with respect to varifold convergence. In \Cref{sec:LY inequalities}, we then derive a monotonicity formula for the Helfrich functional, from which we deduce our most general Li--Yau inequality for varifolds, \Cref{thm:LY varifold}. After that, we review the notion of sets of finite perimeter and introduce the concept of varifolds with enclosed volume. The Li--Yau inequalities in the smooth setting, \Cref{thm:LY smooth} and \Cref{thm:LY Alexandrov} are then a direct application, see \Cref{sec:smooth}.
Finally, in \Cref{sec:applications}, we derive some geometric estimates and discuss implications of our results. This includes a nonexistence result for the penalized version of \Cref{problem:CH} (\Cref{subsec:diameter}), diameter bounds (\Cref{subsec:diameter}), the existence and regularity result for the Canham--Helfrich model, \Cref{thm:regularity_Helfrich_problem} (\Cref{sec:regularity}), and a criterion for positive total mean curvature (\Cref{subsec:pos tot mean}).

	\section{Preliminaries}\label{sec:preliminaries}

In this section, we will review some of the concepts and tools used throughout this article.

\subsection{Notation and definitions}
Let $\mu$ be a Radon measure over $\R^3$ and define the closed balls
\begin{equation}
	B_\rho(x) \defeq \{y\in \R^3\mid |x-y|\leq \rho\}
\end{equation}  
for all $x\in \R^3$ and $\rho>0$. For each nonnegative integer $m$ and $x\in \R^3$, the $m$-dimensional \emph{lower} and \emph{upper density} of $\mu$ at $x$ are defined by
\begin{equation}
	\theta^{m}_*(\mu,x) \defeq \liminf_{\rho \to 0+}\frac{\mu(B_\rho(x))}{\omega_m\rho^m},\qquad \theta^{*m}(\mu,x) \defeq \limsup_{\rho \to 0+}\frac{\mu(B_\rho(x))}{\omega_m\rho^m},
\end{equation}	 
where $\omega_m = \mathcal L^m(B_1(0))$ and $\CalL^m$ is the $m$-dimensional Lebesgue measure over $\R^m$. The $m$-dimensional \emph{density} of $\mu$ at $x$ is defined by
\begin{equation}
	\theta^{m}(\mu,x) \defeq \lim_{\rho \to 0+}\frac{\mu(B_\rho(x))}{\omega_m\rho^m},
\end{equation}
provided the limit exists. 
We define the \emph{support} of $\mu$ by 
\begin{equation}
	\spt \mu \defeq \R^3 \setminus\{x \in \R^3\mid \text{$\exists \rho>0$ such that $\mu(B_\rho(x)) = 0$}\}.
\end{equation}
The $m$-dimensional Hausdorff measure in Euclidean space is denoted with $\CalH^m$. We say that an integral exists if and only if it exists in the Lebesgue sense (i.e.\ its integrand is \emph{summable} in the terminology of \cite[2.4.2]{Federer}).

\subsection{Oriented $2$-varifolds} \label{sec:oriented_varifolds}
Let $\mathbb G^{\mathrm o}(3,2)$ be the set of \emph{oriented 2-dimensional subspaces of $\R^3$}. In view of \cite[3.2.28(2)]{Federer}, we identify $\mathbb G^{\mathrm o}(3,2)$ with 
\begin{equation}
	\{\xi \in \textstyle\bigwedge_2 \R^3 \mid \text{$\xi$ is simple, $|\xi|=1$}\}
\end{equation}
which is a smooth submanifold of the 2-nd exterior power $\bigwedge_2 \R^3$. In particular, $\mathbb G^{\mathrm o}(3,2)$ is a locally compact Hausdorff space. 

Following Hutchinson~\cite{Hutchinson}, %(see also Young~\cite{Young}), 
we say that~$V$ is an \emph{oriented $2$-varifold} on $\R^3$, if and only if $V$ is a Radon measure over $\mathbb G^\mathrm{o}_2(\R^3)\defeq \R^3 \times \mathbb G^\mathrm{o}(3,2)$. The \emph{weight measure} $\mu_V$ on $\R^3$ is defined by 
\begin{equation}
	\mu_V(A) \defeq V\bigl(\{(x,\xi) \in \mathbb G_2^\mathrm{o}(\R^3) \mid x \in A\}\bigr) \qquad \text{whenever $A\subset \R^3$}.
\end{equation}%
It is the push forward of~$V$ under the projection $\R^3 \times \mathbb G^\mathrm{o}(3,2)\to \R^3$. 
The set of oriented $2$-varifolds in $\R^3$ is denoted by $\mathbb V_2^\mathrm{o}(\R^3)$.

For each $\xi \in \mathbb G^{\mathrm o}(3,2)$ we define the unoriented $2$-dimensional subspace $T(\xi)$ of $\R^3$ by
\begin{equation}
	T(\xi) \defeq \{v \in \R^3 \mid v \wedge \xi = 0\}.
\end{equation}
Since $\xi$ is simple, there exist $v_1,v_2\in\R^3$ such that $\xi = v_1\wedge v_2$. Moreover, $|\xi|=1$ implies that $v_1\wedge v_2 = e_1 \wedge e_2$ for $e_1\defeq v_1/|v_1|$ and $e_2\defeq\tilde v_2/|\tilde v_2|$ where $\tilde v_2\defeq v_2 - \langle e_1,v_2\rangle e_1$. In other words, each $\xi \in \mathbb G^\mathrm{o}(3,2)$ corresponds to an oriented orthonormal basis $(e_1,e_2)$ with $\xi = e_1\wedge e_2$ and $T(\xi) = \lin\{e_1,e_2\}$. In particular, each oriented $2$-varifold $V\in\mathbb V_2^\mathrm{o}(\R^3)$ induces a general (unoriented) $2$-varifold in the sense of \cite[Definition 3.1]{Allard}, given by the push forward of $V$ under the map $q(x,\xi) \defeq (x,T(\xi))$. Notice that the two weight measures of $V$ and~$q_\#V$ %(see the definition in~\cite[3.1]{Allard}) 
coincide. 

For all compactly supported vector fields $X \in C^1_c(\R^3;\R^3)$ and $2$-dimensional subspaces $T$ of $\R^3$ with orthonormal basis $\{e_1,e_2\}$, we define
\begin{equation}
	\divergence_T X(x) \defeq \sum_{j=1}^2 \langle e_j, \mathrm D X(x)e_j \rangle. 
\end{equation}
The \emph{first variation} of an oriented $2$-varifold $V$ in $\R^3$ is defined by
\begin{equation} \label{eq:first_variation} 
	\delta V \colon C^1_c(\R^3;\R^3) \to \R, \quad \delta V(X) \defeq \int_{\mathbb G^\mathrm{o}_2(\R^3)} \divergence_{T(\xi)} X(x) \diff V(x,\xi).
\end{equation}
Notice that $\delta V$ coincides with the first variation of the unoriented $2$-varifold $q_\#V$ as defined in~\cite[Definition 4.2]{Allard}. In other words, $\delta V$ does not depend on the orientation.

%We define the \emph{total variation} of $\delta V$ by
%\begin{equation}
%	\|\delta V\|(U) \defeq \sup\{\delta V(X)\mid X\in C_c^1(\R^3;\R^3),\,\spt X\subset U,\,|X|\leq1\}
%\end{equation} 
%for all open sets $U\subset \R^3$ and
%\begin{equation}
%	\|\delta V\|(A) \defeq \inf\{\|\delta V\|(U)\mid \text{$U$ is open in $\R^3$, $A\subset U$}\}
%\end{equation}
%whenever $A\subset \R^3$. Notice that $\|\delta V\|$ defines a Borel regular measure over $\R^3$. Thus, $\|\delta V\|$ is a Radon measure if and only if it is finite on compact sets, in which case $V$ is said to have \emph{locally bounded first variation}. We define the \emph{singular part} of the first variation 
%\begin{equation}\label{eq:singular_part}
%	\beta_V \defeq \|\delta V\| - \|\delta V\|_{\mu_V}
%\end{equation}
%where $\|\delta V\|_{\mu_V}$ denotes the absolutely continuous part of $\|\delta V\|$ with respect to $\mu_V$ (see \cite[2.9.1]{Federer}).

For $k=0, \dots, 3$ we may identify $k$-vectors in $\R^3$ with $(3-k)$-vectors by means of the Hodge star operator 
\begin{equation}
	\star \colon \textstyle\bigwedge_k \R^3 \to \textstyle\bigwedge_{3-k} \R^3.
\end{equation}
If $v_1, v_2\in \R^3\simeq \textstyle\bigwedge_1\R^3$, we have $\star(v_1 \wedge v_2) = v_1 \times v_2$, where $\times$ denotes the usual cross product on~$\R^3$. In particular, for all $\xi\in\mathbb G^\mathrm{o}(3,2)$ there holds $|\star\xi| = 1$. Moreover, we have $\star\star v=v$ for all $v\in \R^3 \simeq \textstyle\bigwedge_1\R^3$.

	If $V$ has locally bounded first variation, that is the total variation of $\delta V$ is a Radon measure, then $\delta V$ can be represented by integration as follows, where the singular part will be denoted by $\beta_V$ (i.e.\ $\beta_V=\norm{\delta V}{\mathrm{sing}}$, cf.\ \cite[4.3]{Allard}).
\begin{hyp} \label{hyp:generalised_mean_curvature_bdr}
	Let $V\in\mathbb V_2^\mathrm{o}(\R^3)$, $\eta \in L^{\infty}(\beta_V;\S^2)$, and $H \in L^1_\mathrm{loc}(\mu_V;\R^3)$. Suppose
	\begin{equation} \label{eq:generalised_mean_curvature_bdr}
		\delta V(X) = -\int_{\R^3} \langle X, H\rangle \diff \mu_V + \int_{\R^3} \langle X, \eta\rangle\diff\beta_V 
	\end{equation}
	and
	\begin{equation} \label{eq:perpendicular_bdr}
		H(x)\wedge \star \xi = 0 \qquad \text{for $V$-almost all $(x,\xi)$}.
	\end{equation}
\end{hyp}
The map $H$ is often referred to as \emph{generalized mean curvature} and $\spt\beta_V$ can be seen as \emph{generalized boundary}. %(see for instance \cite[§39]{SimonGMT}). 
Indeed, one can understand $\beta_V$ as the boundary measure. However, two boundary parts can fall together as in Figure~\ref{fig:lens}. %In this case, $\spt\beta_V$ is the singular part of the boundary of a set of finite perimeter, see \Cref{ex:lens}. In the same \Cref{ex:lens} one can also see that $\beta_V$ does not necessarily have integer multiplicity even if $\mu_V$ has integer multiplicity. 
Typically, one can determine $H$ and $\beta_V$ using Remark~4.4 and 4.7 in~\cite{Allard}. If $V$ is rectifiable (i.e.\ $q_\#V$ is rectifiable in the sense of~\cite[3.5]{Allard}) then the condition in~\eqref{eq:perpendicular_bdr} means that the generalized mean curvature is perpendicular. It is satisfied provided $V$ is an integral varifold, see \cite[Section~5.8]{Brakke}. In the absence of the singular part, \Cref{hyp:generalised_mean_curvature_bdr} simplifies as follows.

\begin{hyp} \label{hyp:generalised_mean_curvature}
	Let $V\in\mathbb V_2^\mathrm{o}(\R^3)$ and $H \in L^1_\mathrm{loc}(\mu_V;\R^3)$. Suppose 
	\begin{equation}\label{eq:generalised_mean_curvature}
		\delta V(X) = -\int_{\R^3} \langle X, H\rangle \diff \mu_V
	\end{equation}
	and
	\begin{equation} \label{eq:perpendicular}
		H(x)\wedge \star \xi = 0 \qquad \text{for $V$-almost all $(x,\xi)$}.
	\end{equation}
\end{hyp}

Let $c_0$ be a real number, and assume $V\in\mathbb V_2^\mathrm{o}(\R^3)$ and $H \in L^1_\mathrm{loc}(\mu_V;\R^3)$ satisfy \Cref{hyp:generalised_mean_curvature_bdr} for some $\eta \in L^\infty(\beta_V;\S^2)$. Then we define the \emph{Helfrich energy} 
\begin{equation}
	\CalH_{c_0}(V) \defeq \frac{1}{4}\int |H(x) - c_0(\star\xi)|^2\diff V(x,\xi) = \frac{1}{4}\int (\langle H(x), \star\xi \rangle - c_0)^2 \diff V(x,\xi).
\end{equation}
Notice that the Helfrich energy does not depend on the singular part of the first variation. This is analogous to the definition in \cite[Section~2]{EichmannCalcVar}. For $c_0 = 0$, we obtain the \emph{Willmore functional} $\CalW \defeq \CalH_{0}$.

\begin{remark}\label{rem:H L2 loc}
%	\begin{enumerate}[(i)]
%		\item 
		%Since $V$ is a Radon measure, it is not difficult to see that the integrand of the Helfrich energy $(x, \xi)\mapsto H(x)-c_0(\star\xi)\in L^2_\mathrm{loc}(\mathbb{G}^{\mathrm{o}}_2(\R^3);\R^3)$ if and only if $H\in L^2_\mathrm{loc}(\mu_V;\R^3)$. In parcticular, $\CalH_{c_0}(V)<\infty$ implies $H\in L^2_\mathrm{loc}(\mu_V;\R^3)$.
		Since $\mu_V$ and $V$ are Radon measures, we have $H\in L^2_\mathrm{loc}(\mu_V;\R^3)$ if and only if the function $(x,\xi)\mapsto H(x) - c_0(\star\xi)$ is a member of $L^2_\mathrm{loc}(V;\R^3)$. Indeed, given any Borel set $B$ in $\R^3$, the Cauchy--Schwarz inequality implies
		\begin{equation}\label{eq:bound_Helfrich_by_Willmore}
			\int_{B\times \mathbb G^\mathrm{o}(3,2)} |H(x) - c_0(\star\xi)|^2\diff V(x,\xi) \leq 2\int_B |H|^2\diff\mu_V +2 c_0^2 \mu_V(B).
		\end{equation}
		On the other hand,
		\begin{align}
			\int_B |H|^2\diff\mu_V & %= \int_{B\times \mathbb G^\mathrm{o}(3,2)} 2|H(x) - c_0(\star\xi)|^2 - |H(x) - 2c_0(\star\xi)|^2 +2c_0^2\diff V(x,\xi)\\ \label{eq:bound_Willmore_by_Helfrich}
			%&
			\leq \int_{B\times \mathbb G^\mathrm{o}(3,2)} 2|H(x) - c_0(\star\xi)|^2\diff V(x,\xi) + 2c_0^2\mu_V(B).\label{eq:bound_Willmore_by_Helfrich}
		\end{align}
		In particular, $\CalH_{c_0}(V)<\infty$ implies $H\in L^2_\mathrm{loc}(\mu_V;\R^3)$.
		
%		\item \label{item:existence_density} Under the assumptions of \Cref{lem:monotonicity}, the existence, finiteness and upper semicontinuity for $x_0\in \R^3\setminus \spt \beta_V$ of the density $\theta^2(\mu_V,x_0)$
%		follows from the fact that $H\in L^2_\mathrm{loc}(\mu_V;\R^3)$ by the varifold Li--Yau inequality for the Willmore energy, cf.\ \cite[Appendix A]{KSRemovability}, \cite[Theorem 3.6]{Scharrer2}.
%		%\item The terms involving the singular part $\beta_V$ do not depend on $c_0$.
%	\end{enumerate}
\end{remark}

\begin{comment}
Given finite dimensional vector spaces $V,W$ and a linear map $f\colon V\to W$, we denote with $\textstyle\bigwedge_2f$ the linear map $\textstyle\bigwedge_2 V\to \textstyle\bigwedge_2 W$ uniquely determined by 
\begin{equation}
	(\textstyle\bigwedge_2f)(v_1\wedge v_2) = f(v_1) \wedge f(v_2) \qquad \text{whenever $v_1,v_2 \in V$}.
\end{equation}
Recall that if the $V$ is $2$-dimensional and $W=V$, then $(\textstyle\bigwedge_2f)\xi = \det(f)\xi$ for $\xi \in\textstyle\bigwedge_2 V$ (see \cite[1.3.4]{Federer}).

Following \cite[4.1.31]{Federer}, we term a $2$-dimensional Riemannian manifold $(M,g)$ to be \emph{orientable} if and only if there exists a continuous $2$-vector field $\zeta$ on $M$ (i.e. $\zeta(p) \in \textstyle\bigwedge_2 T_pM$ for all $p\in M$) such that $|\zeta|_g = 1$. Such a function $\zeta$ is called \emph{$2$-vector field orienting $M$}.
\end{comment}

%Following \cite{Janich}, we term a $2$-dimensional manifold $\Sigma$ to be \emph{orientable}, if there exists a family $\{\mathcal O_p\}_{p\in\Sigma}$ of orientations of the tangent spaces such that for each $p\in M$, there exists an \emph{orientation-preserving chart}, i.e. a chart $x\colon U\subset M\to\mathbb R^2$ around $p$ such that for every $q\in U$, the differential $\mathrm d x_q \colon T_q\Sigma\to\mathbb R^2$ takes the orientation $\mathcal O_q$ to the usual orientation $\R^2$. If in addition $f\colon \Sigma\to \R^2$ is an immersion, then we define the induced continuous normal field $n$ along $f$ by

\subsection{Oriented varifolds induced by immersions}

A particular class of oriented varifolds will be given by immersions of oriented surfaces. Following \cite{Sakai}, we term a surface $\Sigma$ to be \emph{orientable}, if there exists an atlas $A = \{(U_\alpha,x_\alpha)\}_{\alpha\in I}$ such that the Jacobians $\det\mathrm D(x_{\alpha_1}\circ x_{\alpha_2}^{-1})$ of all coordinate transformations are positive. The members of $A$ are called \emph{positive} charts. If $f\colon \Sigma\to \R^3$ is a smooth immersion, then we define the induced smooth normal field $n$ along $f$ (the \emph{Gauss map}) by
\begin{equation}\label{eq:def normal}
	n\colon \Sigma\to\S^2, \quad n \defeq \frac{\partial_{x^1}f \times \partial_{x^2}f}{|\partial_{x^1}f \times \partial_{x^2}f|},
\end{equation}
whenever $x$ is a positive chart. %an orientation preserving chart. 
Notice that since the Hodge star operator is an isometry, $\star n = \partial_{x^1}f\wedge\partial_{x^2}f/|\partial_{x^1}f\wedge\partial_{x^2}f|$ takes values in $\mathbb G^\mathrm{o}(3,2)$. Moreover, in the context of an immersion $f$, we will always denote by $\mu=\mu_f$ the Riemannian measure induced by the pullback metric $g=g_f\defeq f^*\langle\cdot,\cdot\rangle$, and we define by
\begin{align}\label{eq:area_and_volume}
	\CalA(f) \defeq \int_{\Sigma}1\diff\mu,\quad \CalV(f) \defeq -\frac{1}{3}\int_{\Sigma}\langle f, n\rangle\diff\mu\label{eq:defvol}
\end{align}
the area and the (algebraic) volume of $f$, provided the respective integral exists. If $f$ is an embedding and $n$ is the inner unit normal, $\CalV(f)$ yields the enclosed volume as a consequence of the divergence theorem, see \cite[Appendix A]{RuppVolumePreserving} for a more detailed discussion.

In the sequel, the immersion under consideration will usually be clear from the context, so we will drop the dependence on $f$ of the associated geometric quantities.

\begin{example}[Oriented varifold associated with immersed surface]\label{ex:varifold_of_immersion}
	Let $f\colon \Sigma\to \R^3$ be a smooth proper immersion of an oriented surface $\Sigma$ without boundary. %Denote with $n$ the induced continuous normal vector field along $f$ (according to~\eqref{eq:def normal}) and with $\mu$ the Riemannian measure on $\Sigma$ given by the pullback metric. 
	We define the oriented $2$-varifold $V \in \mathbb V_2^\mathrm{o}(\R^3)$ associated with $(\Sigma,f)$ by 
	\begin{equation}
		V(A) \defeq \mu \bigl(\{p \in \Sigma \mid (f(p), \star n(p)) \in A \}\bigr) \qquad \text{whenever $A\subset \mathbb G^\mathrm{o}_2(\R^3)$,}
	\end{equation}%
	i.e.\ $V$ is the push forward of $\mu$ under the map $\Sigma\to \R^3 \times \mathbb G^\mathrm{o}(3,2), p\mapsto (f(p), \star n(p))$.
%	In other words, $V$ is the push forward of $\mu$ under the map 
%	\begin{equation}
%		\Sigma\to \R^3 \times \mathbb G^\mathrm{o}(3,2), \quad p\mapsto (f(p), \star n(p)).
%	\end{equation}
	Since this map is continuous and proper, $V$ is indeed a Radon measure (see \cite[2.2.17]{Federer}). Notice that $T(\star n(p)) = \mathrm d f_p[T_p\Sigma]$ for $p\in\Sigma$. In view of \cite[Lemma~2.3]{Scharrer2}, there holds
	\begin{gather}
		\mu_V(B) = (f_\#\mu)(B) = \int_B\CalH^0(f^{-1}\{x\})\diff\CalH^2(x) \qquad\text{for all Borel sets $B$ in $\R^3$},\\
		\theta^2(\mu_V,x) = \CalH^0(f^{-1}\{x\}) \qquad \text{for all $x\in\R^3$}.
	\end{gather}	
	Moreover, by \cite[2.4.18]{Federer} and the area formula (cf. \cite[Lemma~2.3]{Scharrer2}), we have
	\begin{align} 
		\int_{\mathbb G_2^\mathrm{o}(\R^3)} k(x,\xi)\diff V(x,\xi) & = \int_\Sigma k(f(p),\star n(p))\diff\mu(p) \\ \label{eq:integral_push_forward}
		&= \int_{\R^3}\sum_{p\in f^{-1}\{x\}} k(x, \star n(p))\diff \CalH^2(x)
	\end{align}
	whenever $k\colon\mathbb G_2^\mathrm{o}(\R^3)\to\R$ is a nonnegative Borel function.
	%is $V$-integrable. 
	
	Let $H_f\colon\Sigma\to \R^3$ be the classical mean curvature (vector) of $f$, i.e.\ the trace of the second fundamental form, %(see \cite[Section~3]{HoffmanSpruck}) 
	and define 
	\begin{equation}\label{eq:gmc_immersions}
		H(x) \defeq 
		\begin{cases}
			\frac{1}{\theta^2(\mu_V,x)}\sum_{p\in f^{-1}\{x\}}H_f(p) & \text{if $\theta^2(\mu_V,x) >0$} \\
			0&\text{if $\theta^2(\mu_V,x) =0$}.
		\end{cases}
	\end{equation}
	Then, $H\in L^\infty_{\mathrm{loc}}(\mu_V;\R^3)$ and in view of \cite[Example~2.4]{Scharrer2}, $H(x)\wedge \star\xi = 0$ for $V$-almost all $(x,\xi)$, and
	\begin{equation}
		\delta V(X) = -\int_{\R^3} \langle X,H\rangle \diff\mu_V.
	\end{equation}
	Thus, $V,H$ satisfy \Cref{hyp:generalised_mean_curvature}. 
	
	In the sequel, we will always use the above notation to distinguish $H_f$ as the classical mean curvature when $f$ is an immersion and $H$ defined by \eqref{eq:gmc_immersions} as the generalized mean curvature of the associated varifold. By \cite[Theorem~4]{SchatzleJDG}, there holds
	\begin{align}\label{eq:varifold_H_vs_H_f}
		H_f(p) = H(f(p))\quad \text{for $\mu$-almost all $p\in \Sigma$}.
	\end{align}
	\begin{comment}
	Indeed, using \eqref{eq:first_variation}, \eqref{eq:integral_push_forward} and the divergence theorem, a short computation reveals
	\begin{align}
		\delta V(X) &= \int_{\sigma} \left(\divergence_{T(\star n)} X\right)\circ f\diff\mu = \int_{\Sigma} g^{ij}\langle \partial_{x_i}(X\circ f), \partial_{x_j}f\rangle\diff \mu 
		%\\
		%&= \int_{\Sigma} \divergence_g (X\circ f)^{\top}\diff\mu -\int_{\Sigma}\langle X\circ f, H_f\rangle\diff\mu 
		= -\int_{\Sigma}\langle X\circ f, H_f\rangle\diff\mu
	\end{align}
	for any $X\in C^1_c(\R^3;\R^3)$.
	%, 	where $(X\circ f)^{\top}$ denotes the tangential projection of $X\circ f$ along $f$. 
	Thus, by \eqref{eq:generalised_mean_curvature} and \eqref{eq:integral_push_forward} we find
	\begin{align}
	- \int_{\Sigma}\langle X\circ f, H\circ f\rangle\diff \mu =\delta V(X) =	-\int_{\Sigma}\langle X\circ f, H_f\rangle\diff\mu\quad \text{for all }X\in C^1_c(\R^3;\R^3).
	\end{align} 
	If $f$ is an embedding, this clearly implies $H\circ f=H_f$ $\mu$-almost everywhere. In the general case, $f$ is locally an embedding and \eqref{eq:varifold_H_vs_H_f} follows from a partition of unity argument.}	
%	
%	using local (possibly overlapping) graph representations of~$f[\Sigma]$ and the fact that for any differentiable map $u\colon \R^2\to\R$ there holds $\mathrm Du(x)=0$ for $\CalL^2$-almost all $x\in u^{-1}\{0\}$ as a consequence of \cite[2.9.11]{Federer} and approximate differentiation. 
%	
	%(see the definition \cite[3.1.2]{Federer}). 
	\end{comment}
	Thus, by~\eqref{eq:integral_push_forward} we observe
	\begin{equation}
		\CalH_{c_0}(V) = \frac{1}{4}\int_{\Sigma} |H_f - c_0n|^2\diff\mu. \label{eq:H_energy_varifold_vs_immersion}
	\end{equation}
\end{example}
	
	\section{On the concentrated volume}\label{sec:volume}

In this section, we discuss the concentrated volume  \eqref{eq:def conc vol} in the context of varifolds.

\begin{defi}\label{defi:volume}
	Suppose $V\in\mathbb V_2^\mathrm{o}(\R^3)$ and $x_0\in\R^3$. Then we define the \emph{concentrated volume of $V$ at $x_0$} by
	\begin{equation}
		\CalV_c(V,x_0) \defeq - \int_{\mathbb G_2^\mathrm{o}(\R^3)}\frac{\langle x-x_0,\star\xi\rangle}{|x-x_0|^2}\diff V(x,\xi)
	\end{equation}	
	and the \emph{algebraic volume at $x_0$}
	\begin{equation}
		\CalV(V,x_0) \defeq - \frac{1}{3}\int_{\mathbb G_2^\mathrm{o}(\R^3)}\langle x-x_0,\star\xi\rangle\diff V(x,\xi)
	\end{equation}
	provided the respective integral exists.
\end{defi}
If the varifold $V$ is associated with an immersion $f\colon \Sigma\to\R^3$, then we also write $\CalV_c(f,x_0)$ instead of $\CalV_c(V,x_0)$. By \eqref{eq:integral_push_forward}, this is consistent with \eqref{eq:def conc vol}.
If $\Sigma$ is closed,
%If $f\colon \Sigma \to\R^3$ is an immersion of a closed surface $\Sigma$ and $V$ is the associated varifold as in \Cref{ex:varifold_of_immersion},
then we have $\CalV(V,x_0)=\CalV(f)$ for all $x_0\in \R^3$ after integration by parts. 
	
%\begin{defi}\label{def:volume}
%	Let $V\in\mathbb V_2^\mathrm{o}(\R^3)$. The \emph{algebraic volume of $V$ at $x_0\in\R^3$} is defined by
%	\begin{equation}
%		\CalV(V,x_0)\defeq -\frac{1}{3}\int_{\mathbb{G}^\mathrm{o}_2(\R^3)}\langle x-x_0,\star\xi\rangle\diff V(x,\xi),
%	\end{equation}
%	provided the integral exists.
%\end{defi}

In general, the algebraic volume of an oriented varifold depends on the point $x_0$. Indeed, one may consider the varifold associated to the $2$-dimensional unit sphere in $\R^3$ where the upper hemisphere is oppositely oriented to the lower hemisphere. Moreover, the algebraic volume at $x_0$ exists if and only if the concentrated volume at $x_0$ exists, see \Cref{prop:cvol existence}.
	
\begin{lem}\label{lem:subcrit_integral}
	Suppose $m,\rho_0,D>0$, $\mu$ is a Radon measure over $\R^3$, $x_0\in\R^3$, and
	\begin{equation}
		\mu(B_\sigma(x_0)) \leq  D\sigma^m
	\end{equation}
	for all $0<\sigma<\rho_0$. Then, for all $1\leq p <m$, there exists $C(p,m,D) <\infty$ such that
	\begin{equation}
		\int_{B_\sigma(x_0)}\frac{1}{|x-x_0|^p}\diff\mu(x) \leq C(p,m,D) \sigma^{m-p} 
	\end{equation}
	for all $0<\sigma<\rho_0$. Moreover, $C(1,2,D) = 2D$, and if $\mu(\R^3)<\infty$ then
	\begin{equation}
		\int_{\R^3}\frac{1}{|x-x_0|^p}\diff\mu(x) \leq C(p,m,D) \rho_0^{m-p} + \frac{\mu(\R^3)}{\rho_0^p} < \infty.
	\end{equation}		
\end{lem}

\begin{remark}\label{rem:subcrit_integral}
	\begin{enumerate}[(i)]
		\item \label{item:subcrit_loc} Assume $V\in\mathbb V_2^\mathrm{o}(\R^3)$, $H\in L^{2}_{\mathrm{loc}}(\mu_V;\R^3)$ satisfy \Cref{hyp:generalised_mean_curvature}. By \cite[Theorem 3.6]{Scharrer2}, we find that the density $\theta^2(\mu_V,x_0)$ exists and is finite for all $x_0\in \R^3$. Hence there exist $\rho_0=\rho_0(V,x_0)>0$ and $D=D(V,x_0)<\infty$ such that
		\begin{align}\label{eq:Willmore_density_small}
			\mu_V(B_\sigma(x_0))\leq D\sigma ^2 \quad \text{ for all }0<\sigma<\rho_0, x_0\in \R^3.
		\end{align}
		This immediately implies $\mu_V(\{x_0\})=0$ for all $x_0\in \R^3$. By \Cref{rem:H L2 loc}, the condition $H\in L^2_{\mathrm{loc}}(\mu_V;\R^3)$ is in particular satisfied if $\CalH_{c_0}(V)<\infty$.
		
		\item \label{item:subcrit} If $V\in\mathbb V_2^\mathrm{o}(\R^3)$, $\mu_V(\R^3)<\infty$, and $H\in L^2(\mu_V;\R^3)$ satisfies \Cref{hyp:generalised_mean_curvature}, then the hypothesis of \Cref{lem:subcrit_integral} is satisfied for $m=2$, $\rho_0 =\infty$, and all $x_0 \in \R^3$ with $D = C\CalW(V)$ for some universal constant $0<C<\infty$. Indeed, by \cite[Appendix~(A.16)]{KSRemovability} there holds
		\begin{equation}\label{eq:Willmore_density}
			\mu_V(B_\sigma(x_0)) \leq  C\CalW(V) \sigma^2 \quad \text{ for all }\sigma >0, x_0\in\R^3.
		\end{equation}
	\end{enumerate}

\end{remark}

\begin{proof}[Proof of \Cref{lem:subcrit_integral}]
	Using Fubini's theorem, we compute (cf. \cite[Theorem~1.15]{Mattila})
	\begin{align}
		\int_{B_\sigma(x_0)}\frac{1}{|x-x_0|^p}\diff\mu(x) &= \int_0^\infty \mu(B_\sigma(x_0)\cap B_{t^{-1/p}}(x_0))\diff t \\
		&=\int_0^{\sigma^{-p}}\mu(B_\sigma(x_0))\diff t + \int_{\sigma^{-p}}^\infty \mu(B_{t^{-1/p}}(x_0))\diff t \\
		&\leq D\left(\int_0^{\sigma^{-p}}\sigma^m\diff t + \int_{\sigma^{-p}}^\infty t^{-m/p}\diff t\right) \\ 
		&=D\left(\sigma^{m-p} + \frac{p\sigma^{m-p}}{m-p}\right) = C(p,m,D)\sigma^{m-p}.
	\end{align}	
	The last statement follows by splitting the integral into $\R^3 = B_{\rho_0}(x_0) \cup (\R^3 \setminus B_{\rho_0}(x_0))$. 
\end{proof}

\begin{prop}\label{prop:cvol existence}
	Suppose $V\in \mathbb{V}^{\mathrm{o}}_2(\R^3)$ and $H\in L^2_{\mathrm{loc}}(\mu_V;\R^3)$ satisfy \Cref{hyp:generalised_mean_curvature} and assume that $\CalV(V,x_0)$ exists for some $x_0\in\R^3$. Then also $\CalV_c(V,x_0)$ exists.
\end{prop}

\begin{proof}
	Splitting the integral, for $\rho_0>0$ we find
	\begin{align}
		&\int_{\mathbb{G}^\mathrm{o}_2(\R^3)}\frac{\abs{\langle x-x_0, \star\xi\rangle}}{\abs{x-x_0}^2}\diff V(x,\xi) 
		\\&\qquad \leq \int_{B_{\rho_0}(x_0)}\frac{1}{\abs{x-x_0}}\diff\mu_V(x) + \frac{1}{\rho_0^2}\int_{\mathbb{G}^{\mathrm{o}}_2(\R^3)}\abs{\langle x-x_0, \star\xi\rangle}\diff V(x,\xi).
	\end{align}
	By \Cref{lem:subcrit_integral} and \Cref{rem:subcrit_integral}\eqref{item:subcrit_loc}, on the right hand side the first integral is finite for $\rho_0>0$ small, whereas the second integral is finite since $\CalV(V,x_0)$ exists.
\end{proof}

We recall the concept of convergence of oriented varifolds.

\begin{defi}\label{defi:varifold_convergence}
	Suppose $V_k$ is a sequence in $\mathbb V_2^\mathrm{o}(\R^3)$. Then we say that \emph{$V_k$ converges to $V$ in $\mathbb V_2^\mathrm{o}(\R^3)$} and write
	\begin{equation}
		V_k\to V\qquad\text{in $\mathbb V_2^\mathrm{o}(\R^3)$ as $k\to\infty$}
	\end{equation}
	if and only if $V\in\mathbb V_2^\mathrm{o}(\R^3)$ and
	\begin{equation}
		\int_{\mathbb G_2^\mathrm{o}(\R^3)}\varphi(x,\xi)\diff V_k(x,\xi)\to \int_{\mathbb G_2^\mathrm{o}(\R^3)}\varphi(x,\xi)\diff V(x,\xi) \qquad\text{as $k\to\infty$}
	\end{equation}
	for all continuous functions $\varphi\colon \mathbb G_2^\mathrm{o}(\R^3)\to\R$ with compact support.
\end{defi}

\begin{lem}\label{lem:varifold_convergence}
	Suppose $V_k$ is a sequence in $\mathbb V_2^\mathrm{o}(\R^3)$, $V \in \mathbb V_2^\mathrm{o}(\R^3)$, $H_k\in L^2(\mu_{V_k};\R^3)$ and $H\in L^2(\mu_V;\R^3)$ satisfy \Cref{hyp:generalised_mean_curvature}, 
	\begin{equation}\label{eq:area_and_energy_bound}
		\sup_{k\in\N} \left(\mu_{V_k}(\R^3)+\CalW(V_k)\right) <\infty 
	\end{equation}%
	and 
	\begin{equation}\label{eq:varifold_convergence}
		V_k\to V\qquad\text{in $\mathbb V_2^\mathrm{o}(\R^3)$ as $k\to\infty$}.
	\end{equation}
	Then for all $x_0\in\R^3$, the concentrated volume converges: $\lim_{k\to\infty}\CalV_c(V_k,x_0)=\CalV_c(V,x_0)$.
\end{lem}

\begin{proof}
	Let $x_0\in\R^3$, $0<\sigma <\rho<\infty$, and pick a continuous function $\chi\colon\R^3\to\R$ with compact support in $\R^3\setminus\{x_0\}$ such that $0\leq \chi\leq 1$ and $\chi(x) = 1$ for $\sigma\leq |x - x_0|\leq \rho$. Define the function
	\begin{equation}
		\varphi\colon \R^3\times\mathbb G^\mathrm{o}(3,2)\to\R, \qquad \varphi(x,\xi) \defeq \chi(x) \frac{\langle x-x_0,\star\xi\rangle}{|x-x_0|^2}.
	\end{equation}
	Then $\varphi$ has compact support, $\varphi$ is continuous, and thus, by~\eqref{eq:varifold_convergence},
	\begin{equation}\label{eq:varifold_convergence1}
		\int_{\mathbb G_2^\mathrm{o}(\R^3)}\varphi\diff V_k \to \int_{\mathbb G_2^\mathrm{o}(\R^3)}\varphi\diff V \qquad \text{as $k\to\infty$}.
	\end{equation}
	Let
	\begin{equation}
		A\defeq \sup_{k\in\N}\bigl(\mu_{V_k}(\R^3)+\mu_V(\R^3)\bigr),\qquad D\defeq \sup_{k\in\N}\bigl(\CalW(V_k) + \CalW(V)\bigr).
	\end{equation}
	Then, by \eqref{eq:area_and_energy_bound}, \Cref{lem:subcrit_integral}, and \Cref{rem:subcrit_integral}\eqref{item:subcrit}, we have 
	\begin{equation}
		\sup_{k\in\N}\left(\int_{B_\sigma(x_0)}\frac{1}{|x-x_0|}\diff\mu_{V_k}(x)+ \int_{B_\sigma(x_0)}\frac{1}{|x-x_0|}\diff\mu_V(x)\right)\leq C(D)\sigma 
	\end{equation}
	and
	\begin{equation}
		\sup_{k\in\N}\left(\int_{\R^3\setminus B_\rho(x_0)}\frac{1}{|x-x_0|}\diff\mu_{V_k}(x)+\int_{\R^3\setminus B_\rho(x_0)}\frac{1}{|x-x_0|}\diff\mu_V(x)\right)\leq \frac{C(A)}{\rho}. 
	\end{equation}
	Since $|\varphi(x,\xi)|\leq 1/|x-x_0|$ for all $(x,\xi)\in \mathbb{G}_2^\mathrm{o}(\R^3)$, it follows 
	\begin{equation}
		|\CalV_c(V,x_0) - \CalV_c(V_k,x_0)|\leq \left|\int_{\mathbb G_2^\mathrm{o}(\R^3)}\varphi\diff V_k - \int_{\mathbb G_2^\mathrm{o}(\R^3)}\varphi\diff V\right| +C(D)\sigma + \frac{C(A)}{\rho}.
	\end{equation}%
	Now, the conclusion follows from the convergence in~\eqref{eq:varifold_convergence1}.
\end{proof}	
	
\begin{lem}\label{lem:continuity}
	Suppose $V\in\mathbb V_2^\mathrm{o}(\R^3)$, $\spt\mu_V$ is compact, and $H\in L^2(\mu_V;\R^3)$ satisfies \Cref{hyp:generalised_mean_curvature}.
	Then the concentrated volume $\CalV_c(V,\cdot)$ is H\"older continuous with exponent $\alpha$ for any $0<\alpha<1$ and constant $C=C(\alpha,V)$ depending monotonically nondecreasing on $\mu_V(\R^3)$ and $\CalW(V)$. 
\end{lem}

\begin{proof}
	Let $0<\alpha<1$, $x_0,x_1\in\R^3$ with $0<|x_0-x_1|\leq 1$, and abbreviate $\sigma\defeq |x_0-x_1|$, $A\defeq \mu_V(\R^3)$ and $D\defeq \CalW(V)$. By \Cref{lem:subcrit_integral} and \Cref{rem:subcrit_integral}\eqref{item:subcrit} there holds
	\begin{equation}
		\int_{B_{2\sigma}(x_0)}\frac{1}{|x-x_0|}\diff \mu_V(x) \leq 4D|x_0-x_1|\leq C(D)|x_0-x_1|^\alpha 
	\end{equation}
	and, since $B_{2\sigma}(x_0)\subset B_{3\sigma}(x_1)$,
	\begin{equation}
		\int_{B_{2\sigma}(x_0)}\frac{1}{|x-x_1|}\diff \mu_V(x) \leq \int_{B_{3\sigma}(x_1)}\frac{1}{|x-x_1|}\diff \mu_V(x) \leq 6D|x_0-x_1|\leq C(D)|x_0-x_1|^\alpha.
	\end{equation}
	Thus, we have
	\begin{align}
		&|\CalV_c(V,x_0) - \CalV_c(V,x_1)|\\ \label{eq:cont1} 
		&\quad \leq C(D)|x_0-x_1|^\alpha + \int_{\pi^{-1}[\R^3\setminus B_{2\sigma}(x_0)]}\left|\frac{\langle x-x_0,\star\xi\rangle}{|x-x_0|^2} - \frac{\langle x-x_1,\star\xi\rangle}{|x-x_1|^2}\right|\diff V(x,\xi). 
	\end{align}
	where $\pi:\R^3\times \mathbb G^\mathrm{o}(3,2)\to \R^3$ is the projection. For all $x\in \R^3\setminus B_{2\sigma}(x_0)$ we have 
	\begin{equation}
		2|x_0-x_1| = 2\sigma \leq |x-x_0|, \qquad |x-x_0|\leq 2|x-x_0| - 2|x_0-x_1| \leq 2|x-x_1|
	\end{equation}%
	and thus $|x-x_1|\geq|x-x_0|/2$. Since also $|x_0-x_1|\leq |x-x_0|$, we infer
	\begin{align}
		&\left|\frac{\langle x-x_0,\star\xi\rangle}{|x-x_0|^2} - \frac{\langle x-x_1,\star\xi\rangle}{|x-x_1|^2}\right| \\
		&\quad \leq \left|\frac{\langle x-x_0,\star\xi\rangle}{|x-x_0|^2} - \frac{\langle x-x_1,\star\xi\rangle}{|x-x_0|^2}\right|+ \left|\frac{\langle x-x_1,\star\xi\rangle}{|x-x_0|^2} - \frac{\langle x-x_1,\star\xi\rangle}{|x-x_0||x-x_1|}\right| \\
		&\quad \quad + \left|\frac{\langle x-x_1,\star\xi\rangle}{|x-x_0||x-x_1|} - \frac{\langle x-x_1,\star\xi\rangle}{|x-x_1|^2}\right|\\
		&\quad\leq 2\frac{|x_0 - x_1|}{|x-x_0|^2} + \frac{|x_0 - x_1|}{|x-x_0||x-x_1|} \leq 4\frac{|x_0-x_1|^\alpha}{|x-x_0|^{1+\alpha}}.  
	\end{align}
	Integrating this inequality and applying \Cref{lem:subcrit_integral} for $p=1+\alpha$, $m=2$ and $\rho_0 = 1$, \eqref{eq:cont1} becomes
	\begin{align}
		|\CalV_c(V,x_0) - \CalV_c(V,x_1)|&\leq C(D)|x_0-x_1|^\alpha + 4|x_0-x_1|^\alpha \int_{\R^3}\frac{1}{|x-x_0|^{1+\alpha}}\diff\mu_V(x) \\
		&\leq \bigl[C(D) + 4C(\alpha,D) + 4A\bigr]|x_0-x_1|^\alpha = C(\alpha,A,D)|x_0-x_1|^\alpha.
	\end{align}
	For $|x_0-x_1|\geq 1$ we apply \Cref{lem:subcrit_integral} to see
	\begin{equation}
		|\CalV_c(V,x_0) - \CalV_c(V,x_1)| \leq 2(2D + A) \leq C(A,D) |x_0-x_1|^\alpha 
	\end{equation}
	which concludes the proof.
\end{proof}

\begin{example} \label{ex:volume}
	Consider $S\defeq\partial B_1(0)\subset \R^3$, the round sphere with radius one centered at the origin. Moreover, for $r>0$ let $T_r\subset \R^3$ be the torus which is obtained by revolving a circle with radius $r$ and center $(1+r,0)$ (in the $xz$-plane) around the $z$-axis. Note that if we revolve the corresponding disk instead of the circle, we obtain a full torus $T_r^{\mathrm{full}}$ with $\partial T_r^{\mathrm{full}} = T_r$. We now define a smooth unit normal $n$ on $S\cup T_r$ by taking $n$ to be the outer unit normal on $S$ and the inner unit normal on $T_r$, cf.\ Figure~\ref{fig:torus_and_sphere}. It is not difficult to see that $S\cup T_r$ is the image of a $C^{1,1}$-immersion $f\colon \mathbb{S}^2\to\R^3$.

	By standard formulas in geometry, the algebraic volume can be computed as
		\begin{align}\label{eq:example torus vol}
			\CalV(f) = -\CalL^{3}(B_1(0)) + \CalL^{3}(T_r^{\mathrm{full}}) = -\frac{4\pi}{3}+2\pi^2 r^2(1+r).
		\end{align}
	Let $x_0 = 0$ be the origin.
		Using that $\abs{x-x_0}=1$ for $x\in S$ and applying the divergence theorem to $T_r^{\mathrm{full}}$, the concentrated volume is given by
		\begin{align}
			\CalV_c(f,x_0) =-\frac{4\pi}{3} + \int_{T_r^{\mathrm{full}}} \frac{1}{\abs{x-x_0}^2}\diff \CalL^3(x).
		\end{align}
		Clearly $\abs{x-x_0}>1$ for $\CalL^3$-almost every $x\in T_r^{\mathrm{full}}$, so that
		\begin{align}\label{eq:example torus conc vol}
			\CalV_c(f,x_0)< -\frac{4\pi}{3} + \int_{T_r^{\mathrm{full}}}\diff \CalL^3(x) = \CalV(f).
		\end{align}
		By means of \eqref{eq:example torus vol}, we thus find $r>0$ such that $\CalV(f)=0$ but $\CalV_c(f,x_0)<0$. Slightly increasing the radius, we have $\CalV_c(f,x_0)<0<\CalV(f)$ by a continuity argument. Lastly, we replace a small disk on $S$ by a thin dent, such that the new surface $\tilde{S}$ satisfies $x_0\in \tilde{S}$ and such that $\tilde{S}\cup T_r$ is still the image of an immersion $\tilde{f}\colon \mathbb{S}^2\to\R^3$. Making this dent sufficiently thin and smoothing, we can achieve that $\CalV_c(\tilde{f},x_0)<0<\CalV(\tilde{f})$ is still satisfied and $\tilde{f}$ is smooth. Therefore, positive algebraic volume does not imply positive concentrated volume, even at points in the support.
\end{example}

\section{The Li--Yau inequality in the varifold setting} \label{sec:LY inequalities}

\subsection{A monotonicity formula}
Our first essential observation is the following lemma, which can be seen as an extension of the monotonicity formula due to Simon \cite[(1.2)]{SimonWillmore}.
We follow the varifold approach in \cite[Appendix A]{KSRemovability} relying on
%, see also \cite[Appendix A]{KSRemovability} in the context of varifolds. As in \cite[Appendix A]{KSRemovability}, we will work with 
the first variation identity and examine the additional terms originating from the spontaneous curvature.

\begin{lem} \label{lem:monotonicity}
	Suppose $V \in \mathbb V_2^\mathrm{o}(\R^3)$, $\eta\in L^\infty(\beta_V;\S^2)$, and $H \in L^2_{\mathrm{loc}}(\mu_V;\R^3)$ satisfy \Cref{hyp:generalised_mean_curvature_bdr}. Let $x_0\in\R^3$ and abbreviate $B_r\defeq B_r(x_0)$, $A_r\defeq B_r\times \mathbb{G}^\mathrm{o}(3,2)$ for all $r>0$. Then, for $c_0\in\R$ and $0<\sigma<\rho<\infty$, there holds
	\begin{align}
		&\frac{\mu_V(B_\sigma)}{\sigma^2} + \int_{A_{\rho}\setminus A_{\sigma}} \left(\frac{1}{4}(\langle H(x),\star\xi\rangle - c_0) + \frac{\langle x-x_0, \star\xi\rangle}{|x-x_0|^2}\right)^2\diff V(x,\xi)\\
		&\quad = \frac{1}{16}\int_{A_{\rho}\setminus A_{\sigma}}|H(x) - c_0(\star\xi)|^2\diff V(x,\xi) - \frac{c_0}{2}\int_{A_{\rho}\setminus A_{\sigma}}\frac{\langle x-x_0,\star\xi\rangle}{|x-x_0|^2}\diff V(x,\xi) + \frac{\mu_V(B_\rho)}{\rho^2}\\
		&\quad\quad -		
		\frac{1}{2\sigma^2}\int_{A_\sigma}\langle x-x_0,H(x)-c_0(\star\xi)\rangle \diff V(x,\xi) - \frac{c_0}{2\sigma^{2}}\int_{A_\sigma}\langle x-x_0, \star\xi\rangle\diff V(x,\xi) \\
		&\quad\quad + \frac{1}{2\rho^2}\int_{A_\rho}\langle x-x_0,H(x)-c_0(\star\xi)\rangle \diff V(x,\xi) + \frac{c_0}{2\rho^{2}}\int_{A_\rho}\langle x-x_0, \star\xi\rangle\diff V(x,\xi) \\
		&\quad\quad+\frac{1}{2\sigma^2}\int_{B_\sigma}\langle x-x_0,\eta(x)\rangle\diff\beta_V(x) + \frac{1}{2}\int_{B_\rho\setminus B_\sigma}\frac{\langle x-x_0,\eta(x)\rangle}{|x-x_0|^2}\diff\beta_V(x) \\
		&\quad \quad- \frac{1}{2\rho^2}\int_{B_\rho}\langle x-x_0,\eta(x)\rangle\diff\beta_V(x).
		 \label{eq:monoton function}
	\end{align}
%	where $A_{\sigma} \defeq B_\sigma\times \mathbb{G}^\mathrm{o}(3,2)$ is the preimage of $B_\sigma$ under the projection $\R^3\times \mathbb G^\mathrm{o}(3,2) \to \R^3$.
\end{lem}

\begin{proof}[Proof of \Cref{lem:monotonicity}]
	Following the computations in \cite[p.~284]{SimonWillmore}, we consider the smooth vector field $X(x)\defeq x-x_0$ for $x\in \R^3$ and the Lipschitz function 
	\begin{equation}\label{eq:varphi}
		\varphi\colon \R\to\R,\qquad \varphi(t)\defeq (\max\{t,\sigma\}^{-2} - \rho^{-2})_+.
	\end{equation}
	Choose a sequence $\varphi_k$ in $C^\infty_c(-\infty,\rho+1)$ such that $\sup_{k\in\N}\|\varphi_k\|_{C^1(\R)}<\infty$, 
	\begin{comment}
	vector field $(|X|_\sigma^{-2} - \rho^{-2})_+X$ where $|X|_\sigma \defeq \max\{|X|, \sigma\}$. The Lipschitz function $(|X|_\sigma - \rho^{-2})_+$ can be approximated by a sequence $\varphi_k\in C^\infty_c(\R^3,\R)$ such that for all $\xi \in \mathbb G^\mathrm{o}(3,2)$ and $x\in\R^3$ there holds 
	\begin{equation}\label{eq:monotonicity_cases}
		\lim_{k\to\infty}\varphi_k(x) = (|X(x)|_\sigma^{-2} - \rho^{-2})_+ = 
		\begin{cases}
			\frac{1}{\sigma^2} - \frac{1}{\rho^2} & \text{for $x\in B_\sigma$}\\
			\frac{1}{|X(x)|^2} - \frac{1}{\rho^2} & \text{for $x\in B_\rho\setminus B_\sigma$}\\
			0 & \text{for $x\in\R^3\setminus B_\rho$}
		\end{cases}
	\end{equation}
	as well as
	\end{comment}
	\begin{align}
		&\varphi_k \to \varphi && \text{locally uniformly as $k\to\infty$},\\
		&\varphi_k'(t) \to \varphi'(t) &&\text{as $k\to\infty$ for all $t\in\R\setminus\{\sigma,\rho\}$}
	\end{align}%
	and such that for all $k\in\N$, there holds $\varphi'_k(\sigma) = 0$ and $\varphi'_k(\rho) =- 2\rho^{-3}$. Abbreviating $\Phi_k\defeq \varphi_k\circ |X|$ it follows
	\begin{equation}
		\lim_{k\to\infty}\divergence_{T(\xi)}(\Phi_kX)(x) =
		\begin{cases}
			2(\frac{1}{\sigma^2} - \frac{1}{\rho^2}) & \text{for $(x,\xi)\in A_\sigma$}\\
			\frac{2\langle X(x),\star\xi\rangle^2}{|X(x)|^4} - \frac{2}{\rho^2} & \text{for $(x,\xi)\in A_\rho\setminus A_\sigma$}\\
			0 & \text{for $(x,\xi)\in\mathbb G_2^\mathrm{o}(\R^3)\setminus A_\rho$}.
		\end{cases}
	\end{equation} 
	Denoting $|X|_\sigma \defeq \max\{|X|,\sigma\}$, testing the first variation identity (see \eqref{eq:generalised_mean_curvature_bdr}, and~\eqref{eq:first_variation}) with the vector fields $\Phi_kX$ and passing to the limit as $k\to\infty$, we obtain
	\begin{align}
		&\frac{2\mu_V(B_\sigma)}{\sigma^2} + \int_{A_{\rho}\setminus A_\sigma}\frac{2\langle X(x),\star\xi\rangle^2}{|X(x)|^4}\diff V(x,\xi) \\
		&\quad = \frac{2\mu_V(B_\rho)}{\rho^2} - \int_{B_\rho} (|X|_\sigma^{-2} - \rho^{-2})\langle X, H\rangle\diff \mu_V + \int_{B_\rho} (|X|_\sigma^{-2} - \rho^{-2})\langle X, \eta\rangle\diff \beta_V.
	\end{align} 
	By~\eqref{eq:perpendicular_bdr} and since $|\star\xi|=1$ for $\xi\in\mathbb G^\mathrm{o}(3,2)$, we have the pointwise identity
	\begin{equation}
		\left|\frac{1}{4}(H - c_0(\star\xi)) + \frac{\langle X,\star\xi\rangle(\star\xi)}{|X|^2}\right|^2 = \frac{1}{16}|H-c_0(\star\xi)|^2 + \frac{\langle H - c_0(\star\xi),X\rangle}{2|X|^2} + \frac{\langle X,\star\xi\rangle^2}{|X|^4}
	\end{equation}
	and consequently
	\begin{align}
		& \frac{\mu_V(B_\sigma)}{\sigma^2} + \int_{A_{\rho}\setminus A_\sigma} \left(\frac{1}{4}(\langle H(x),\star\xi\rangle - c_0) + \frac{\langle X(x), \star\xi\rangle}{|X(x)|^2}\right)^2\diff V(x,\xi) \\
		&\quad = \frac{1}{16}\int_{A_{\rho}\setminus A_\sigma}|H(x)-c_0(\star\xi)|^2\diff V(x,\xi) + \frac{\mu_V(B_\rho)}{\rho^2} \\
		&\quad \quad + \frac{1}{2}\int_{A_{\rho}\setminus A_\sigma} \frac{\langle H(x) - c_0(\star\xi),X(x)\rangle}{|X(x)|^2}\diff V(x,\xi) - \frac{1}{2}\int_{B_\rho} (|X|_\sigma^{-2} - \rho^{-2})\langle X, H\rangle\diff \mu_V  \\ \label{eq:monotonicity_ln1}
		&\quad \quad + \frac{1}{2} \int_{B_\rho} (|X|_\sigma^{-2} - \rho^{-2})\langle X, \eta\rangle\diff \beta_V.
	\end{align}
%	It follows
	Moreover, we have
	\begin{align}
		& \frac{1}{2}\int_{A_{\rho}\setminus A_\sigma} \frac{\langle H(x) - c_0(\star\xi),X(x)\rangle}{|X(x)|^2}\diff V(x,\xi) - \frac{1}{2}\int_{B_\rho} (|X|_\sigma^{-2} - \rho^{-2})\langle X, H\rangle\diff \mu_V \\
		&\quad =  - \frac{c_0}{2}\int_{A_{\rho}\setminus A_\sigma} \frac{\langle X(x), \star\xi\rangle}{|X(x)|^2}\diff V(x,\xi) - \frac{1}{2\sigma^2}\int_{B_\sigma} \langle X, H\rangle\diff \mu_V +  \frac{1}{2\rho^2}\int_{B_\rho} \langle X, H\rangle\diff \mu_V \\
		&\quad =  - \frac{c_0}{2}\int_{A_{\rho}\setminus A_\sigma} \frac{\langle X(x), \star\xi \rangle}{|X(x)|^2}\diff V(x,\xi) \\
		&\quad \quad - \frac{1}{2\sigma^2}\int_{A_\sigma} \langle X(x), H(x) - c_0(\star\xi)\rangle\diff V(x,\xi) - \frac{c_0}{2\sigma^2}\int_{A_\sigma} \langle X(x), \star\xi\rangle\diff V(x,\xi)\\ \label{eq:monotonicity_ln2}
		&\quad \quad + \frac{1}{2\rho^2}\int_{A_\rho} \langle X(x), H(x) - c_0(\star \xi)\rangle\diff V(x,\xi)+ \frac{c_0}{2\rho^2}\int_{A_\rho} \langle X(x), \star \xi \rangle\diff V(x,\xi)
	\end{align}
	as well as 
	\begin{align}
		&\frac{1}{2} \int_{B_\rho} (|X|_\sigma^{-2} - \rho^{-2})\langle X, \eta\rangle\diff \beta_V \\ \label{eq:monotonicity_ln3}
		&\quad = \frac{1}{2\sigma^2}\int_{B_\sigma}\langle X,\eta\rangle\diff\beta_V + \frac{1}{2}\int_{B_\rho\setminus B_\sigma}\frac{\langle X,\eta\rangle}{|X|^2}\diff\beta_V - \frac{1}{2\rho^2}\int_{B_\rho}\langle X,\eta\rangle\diff\beta_V.
	\end{align}
	Now, using $X(x) = x-x_0$ and putting \eqref{eq:monotonicity_ln2}, \eqref{eq:monotonicity_ln3} into \eqref{eq:monotonicity_ln1}, the conclusion follows.
\end{proof}

\subsection{The general varifold case}

We now use the monotonicity formula \eqref{eq:monoton function} to prove our most general Li--Yau inequality.
%\todo{In \Cref{thm:LY varifold}, the assumption $H\in L^2_\mathrm{loc}$ is redundant.}

\begin{thm}\label{thm:LY varifold}
	Suppose $V\in\mathbb{V}_2^{\mathrm{o}}(\R^3)$,  $\eta\in L^\infty(\beta_V;\S^2)$ and $H\in L^1_\mathrm{loc}(\mu_V;\R^3)$ satisfy \Cref{hyp:generalised_mean_curvature_bdr}. Let $c_0\in \R$ and suppose that
	\begin{align}\label{eq:Helfrich finite}
		\CalH_{c_0}(V)<\infty
	\end{align}
	and
	\begin{align}\label{eq:dens at infty finite}
		\theta^{*2}(\mu_V,\infty)\defeq\limsup_{\rho\to\infty}\frac{\mu_V(B_\rho(0))}{\pi\rho^2}<\infty.
	\end{align}
	Then, for all $x_0\in \R^3\setminus \spt \beta_V$ we have
	\begin{align}
			\theta^2(\mu_V, x_0)&\leq \theta^{*2}(\mu_V,\infty) + \frac{1}{4\pi}\CalH_{c_0}(V) \\
			&\quad  + \limsup_{\rho\to\infty}\frac{c_0}{2\pi}\left(\int_{B_\rho(x_0)\times \mathbb{G}^{\mathrm{o}}_2(\R^3)}\left(\rho^{-2}-\abs{x-x_0}^{-2}\right) \langle x-x_0, \star\xi\rangle \diff V(x,\xi)\right)\\
			&\quad +  \limsup_{\rho\to\infty}\frac{1}{2\pi}\left(\int_{B_\rho(x_0)}(\abs{x-x_0}^{-2}-\rho^{-2})\langle x-x_0, \eta(x)\rangle\diff \beta_V(x)\right).\label{eq:LY Helfrich Varifold}
	\end{align}
\end{thm}

\begin{remark}\label{rem:LY varifold}
	\begin{enumerate}[(i)]
		\item We do not assume $\mu_V(\R^3)<\infty$ in \Cref{thm:LY varifold}. Indeed, let $r=1/c_0$ for $c_0>0$, let $f\colon \R\times \mathbb{S}^1\to\R^3$, $f(t, \varphi)=(r\cos\varphi, r\sin\varphi,t)$ be the cylinder with radius $r$, let $V$ be the associated varifold, cf.\ \Cref{ex:varifold_of_immersion}, and let $x_0=(r,0,0)\in \spt\mu_V$. 
		It is not difficult to see that $\beta=0$, $\CalH_{c_0}(V)=0$ and $\mu_V(B_\rho(x_0))=O(\rho)$ as $\rho\to\infty$, so that $\theta^{*2}(\mu_V,\infty)=0$ whereas $\mu_V(\R^3)=\infty$. 
		Moreover,
		the third term on the right hand side of  \eqref{eq:LY Helfrich Varifold} is
		\begin{align}
			-\frac{c_0}{2\pi} \int_{\R\times \S^1}\frac{\langle f-x_0, n\rangle}{\abs{f-x_0}^2}\diff \mu &= \frac{c_0}{2\pi} \int_0^{2\pi}\int_{\R} \frac{a(\varphi)}{2a(\varphi) +t^2}\diff t\diff \varphi,
		\end{align}
		where $a(\varphi)\defeq r^2(1-\cos\varphi)\geq 0$. Hence, the inner integral can be evaluated using the $\arctan$, yielding
		\begin{align}
			-\frac{c_0}{2\pi} \int_{\R\times \S^1}\frac{\langle f-x_0, n\rangle}{\abs{f-x_0}^2}\diff \mu &= \frac{c_0}{2\pi} \int_0^{2\pi} \pi\sqrt{\frac{a(\varphi)}{2}} \diff\varphi = \frac{c_0}{2\pi}\cdot 4\pi r.
		\end{align}
		In the last step, we used  $1-\cos\varphi =  2\sin^2(\frac{\varphi}{2})$ and the symmetry of the sine function.
%
%			Moreover, the third term on the right hand side of  \eqref{eq:LY Helfrich Varifold} can be computed to be $4\pi r$.
		\item We can reverse the orientation of the varifold $V$ by considering $\hat{V}$, the push forwad under the map $(x, \xi)\mapsto (x,-\xi)$, which is continuous and proper so $\hat{V}\in \mathbb{V}^{\mathrm{o}}_2(\R^3)$ by \cite[2.2.17]{Federer}. In view of \eqref{eq:Helfrich orient transform} it is not suprising that
		\begin{align}
			\CalH_{c_0}(V) = \CalH_{-c_0}(\hat{V}).
		\end{align}
		Similarly, the other term in \eqref{eq:LY Helfrich Varifold} involving $c_0$ remains unchanged if we replace $V$ by $\hat{V}$ and $c_0$ by $-c_0$. The singular part does not change under reversing the orientation.
		\item Equality  holds for $c_0=0$ if $V$ corresponds to the unit sphere and $x_0$ is any point on the unit sphere. Equality also holds for $c_0=0$ if $V$ corresponds to the unit disk and $x_0$ is the center, %of the unit disk, 
		and if $V$ corresponds to a plane and $x_0$ is any point on the plane.
		\item If the singular part $\beta_V$ is regular enough, for instance if $\spt\beta_V$ is given by a smooth embedding $\gamma\colon\S^1\to\R^3$ and $\eta\circ\gamma$ is a unit normal field along $\gamma$, then the statement remains valid even for $x_0\in\spt\beta_V$. Indeed, for $x$ close to $x_0$, the vectors $x-x_0$ and $\eta(x)$ are nearly orthogonal. Thus, since $\theta^1(\beta_V,x_0)=1$, a short argument using the Taylor expansion of $\gamma$ implies
		%\begin{equation}
		%	\lim_{\sigma\to0+} \frac{1}{2\sigma^2}\int_{B_\sigma}\langle x-x_0,\eta(x)\rangle\diff\beta_V(x) = 0.
		%\end{equation}  
		%Similarly,
		\begin{align}
			x\mapsto \abs{x-x_0}^{-2}\langle x-x_0, \eta(x)\rangle\in L^1_{\mathrm{loc}}(\beta_V).
		\end{align}	
	\end{enumerate}
\end{remark}

\begin{proof}[Proof of \Cref{thm:LY varifold}]
	For $\rho>0$ let $B_\rho$ and $A_\rho$ be as in \Cref{lem:monotonicity}. 
	By \Cref{rem:subcrit_integral}\eqref{item:subcrit_loc},
%	By \Cref{rem:H L2 loc}, the density $\theta^2(\mu_V, x_0)$ exists and is finite, and hence 
	there exist $D<\infty$ and $\rho_0>0$ such that
	\begin{align}
		\mu(B_\rho)\leq D\rho^2 \quad \text{for all }0<\rho<\rho_0.\label{eq:density quotient bound}
	\end{align}
	Consequently, \Cref{lem:subcrit_integral} yields
	\begin{align}\label{eq:subcrit integral}
		\int_{B_{\rho}}\frac{1}{\abs{x-x_0}}\diff\mu_V(x)\leq C \rho \quad \text{for all }0<\rho< \rho_0,
	\end{align}
	and thus $x\mapsto \abs{x-x_0}^{-1}\in L^1_{\mathrm{loc}}(\mu_V)$. 
	Moreover we have $\mathrm{dist}(x_0,\spt\beta_V)>0$, and consequently
	\begin{align}
		x\mapsto \abs{x-x_0}^{-2}\langle x-x_0, \eta(x)\rangle\in L^1_{\mathrm{loc}}(\beta_V).\label{eq:subcrit integral boundary}
	\end{align}	

	Using \eqref{eq:Helfrich finite}, \eqref{eq:subcrit integral} and \eqref{eq:subcrit integral boundary}, we find that the function $\gamma\colon (0,\infty)\to\R$ with
	\begin{align}
		\gamma(\rho)&\defeq \frac{\mu_V(B_\rho)}{\rho^2} + \frac{1}{16} \int_{A_\rho}\abs{H(x)-c_0(\star\xi)}^2\diff V(x,\xi) - \frac{c_0}{2}\int_{A_\rho}\frac{\langle x-x_0, \star\xi\rangle}{\abs{x-x_0}^2}\diff V(x,\xi)\\
		&\quad  + \frac{1}{2\rho^2}\int_{A_\rho}\langle x-x_0, H(x)-c_0(\star\xi)\rangle\diff V(x,\xi) + \frac{c_0}{2\rho^2}\int_{A_\rho}\langle x-x_0, \star\xi\rangle\diff V(x,\xi)\\
		&\quad + \frac{1}{2} \int_{B_\rho}(\abs{x-x_0}^{-2}-\rho^{-2})\langle x-x_0, \eta(x)\rangle\diff\beta_V(x)
		\label{eq:def gamma}
	\end{align}
	is well defined and, by \Cref{lem:monotonicity}, it is monotonically nondecreasing. 
	
	We now examine the limit $\lim_{\sigma\to 0+}\gamma(\sigma)$.
	By \eqref{eq:Helfrich finite}, the second term in $\gamma(\sigma)$ goes to zero as $\sigma\to 0+$ and so does the third term by \eqref{eq:subcrit integral}. For the fourth term, we use the Cauchy--Schwarz inequality to estimate
	\begin{align}
		&\Abs{\sigma^{-2}\int_{A_\sigma}\langle x-x_0, H(x)-c_0(\star\xi)\rangle\diff V(x,\xi)} \\ \label{eq:gamma 4 to 0}
		&\qquad \leq \left(\sigma^{-2}\mu_V(B_\sigma)\right)^{\frac{1}{2}}\left(\int_{A_\sigma}\abs{H(x)-c_0(\star\xi)}^2\diff V(x,\xi)\right)^{\frac{1}{2}}, 
	\end{align}
	where the right hand side goes to zero by
	\eqref{eq:Helfrich finite}, \eqref{eq:density quotient bound} and since $\mu_V(\{x_0\})=0$ by \Cref{rem:subcrit_integral}\eqref{item:subcrit_loc}. The fifth term in $\gamma(\sigma)$ also goes to zero as $\sigma\to 0+$, since
	\begin{align}\label{eq:gamma 5 to 0}
		\sigma^{-2}\Abs{\int_{A_\sigma} \langle x-x_0, \star\xi\rangle\diff V(x,\xi)}\leq  \sigma^{-1}\mu_V(B_\sigma) \leq D\sigma,
	\end{align}
	using \eqref{eq:density quotient bound}. Since $x_0\not\in \spt\beta_V$, we have $\beta_V(B_\sigma)=0$ for $\sigma>0$ sufficiently small. Consequently, using $\omega_2=\pi$, we find
	$\lim_{\sigma\to 0+}\gamma(\sigma)= \pi \theta^{2}(\mu_V, x_0)$.
	
	Now, we discuss the limit $\lim_{\rho\to\infty}\gamma(\rho)$. It is not too difficult to see that
	\begin{align}\label{eq:density comparison}
		\limsup_{\rho\to\infty}\frac{\mu_V(B_\rho)}{\pi\rho^2} = \limsup_{\rho\to\infty}\frac{\mu_V(B_\rho(0))}{\pi\rho^2} = \theta^{*2}(\mu_V, \infty).
	\end{align}
	For the fourth term in \eqref{eq:def gamma}, for any $0<\sigma<\rho$, we estimate by Cauchy--Schwarz 
	\begin{align}
		&\Abs{\rho^{-2}\int_{A_\rho}\langle x-x_0, H(x)-c_0(\star\xi)\rangle\diff V(x,\xi)}\\
		&\qquad \leq \left(\rho^{-2}\mu_V(B_\rho)\right)^{\frac{1}{2}}\left(\int_{\mathbb{G}_2^{\mathrm{o}}(\R^3)\setminus A_\sigma} \abs{H(x)-c_0(\star\xi)}^2\diff V(x,\xi)\right)^{\frac{1}{2}} \\
		&\qquad \quad + \rho^{-1}\int_{A_\sigma}\abs{H(x)-c_0(\star\xi)}\diff V(x,\xi).
	\end{align}
	Sending first $\rho\to\infty$ and then $\sigma\to\infty$, this goes to zero by \eqref{eq:Helfrich finite}, \eqref{eq:dens at infty finite} and \eqref{eq:density comparison}. 
	The claim then follows from the monotonicity of $\gamma$.
\end{proof}

%	Consequently, by monotonicity we obtain
%	\begin{align}
	%		\theta^2(\mu_V, x_0) &= \frac{1}{\pi} \lim_{\sigma\to 0+}\gamma(\sigma)\leq  \frac{1}{\pi}\lim_{\rho\to\infty}\gamma(\rho)\\
	%		&\leq \theta^{*2}(\mu_V,\infty) + \frac{1}{4\pi}\CalH_{c_0}(V) \\
	%		&\quad + \limsup_{\rho\to\infty}\frac{c_0}{2\pi}\left(\int_{A_\rho}\left(\rho^{-2}-\abs{x-x_0}^{-2}\right) \langle x-x_0, \star\xi\rangle \diff V(x,\xi)\right)\\
	%		&\quad +  \limsup_{\rho\to\infty}\frac{1}{2\pi}\left(\int_{B_\rho}(\abs{x-x_0}^{-2}-\rho^{-2})\langle x-x_0, \eta(x)\rangle \diff\beta_V(x)\right). &&\qedhere
	%	\end{align}

%\begin{remark}
%	By an adaptation of the arguments in \cite[Appenix A]{KSRemovability}, the existence and upper semicontinuity of the density can also be deduced from the monotonicity of $\gamma$ in \eqref{eq:def gamma}.
%\end{remark}

If the singular part $\beta_V$ vanishes and $\CalV_c(V,x_0)$ exists, using \Cref{lem:subcrit_integral} we obtain the following
\begin{cor}\label{cor:LY varif area dec}
	Suppose $V\in \mathbb{V}^{\mathrm{o}}_2(\R^3)$ and $H\in L^2_\mathrm{loc}(\mu_V;\R^3)$ satisfy \Cref{hyp:generalised_mean_curvature}. Let $c_0\in \R$, $x_0\in \R^3$ and suppose that $\CalV_c(V,x_0)$ exists.
	Then we have
	\begin{align}
			\theta^{2}(\mu_V, x_0)\leq \theta^{* 2}(\mu_V,\infty)+ \frac{1}{4\pi}\CalH_{c_0}(V) + \frac{c_0}{2\pi} \CalV_c(V,x_0).\label{eq:LY varif area dec}
		\end{align}
\end{cor}

\begin{proof}
	Without loss of generality, we may assume $\CalH_{c_0}(V)<\infty$, $\theta^{* 2}(\mu_V,\infty)<\infty$.
	By \Cref{thm:LY varifold}, we only need to discuss the third term on the right hand side of \eqref{eq:LY Helfrich Varifold}. To that end, for $0<\sigma<\rho$ we estimate
	\begin{align}
		&\frac{1}{\rho^2}\int_{A_\rho}\abs{\langle x-x_0, \star\xi\rangle}\diff V(x,\xi) \\
		&\qquad \leq \int_{\mathbb{G}^{\mathrm{o}}_2(\R^3)\setminus A_\sigma} \frac{\abs{\langle x-x_0, \star\xi\rangle}}{\abs{x-x_0}^2}\diff V(x,\xi) + \frac{1}{\rho^2} \int_{A_\sigma}\abs{\langle x-x_0, \star\xi\rangle}\diff V(x,\xi).
	\end{align}
	Sending first $\rho\to\infty$ and then $\sigma\to \infty$ this goes to zero since $\CalV_c(V,x_0)$ exists by assumption. The result follows.
\end{proof}

\subsection{Varifolds with enclosed volume}\label{sec:finiteperimeter}

In this section we introduce a class of oriented varifolds that satisfy a divergence theorem, see \Cref{hyp:EmuH}. These varifolds comprise the surfaces shown in Figures~\ref{fig:touching_spheres}--\ref{fig:pillow} and \ref{fig:intersecting_spheres}. We then show that their concentrated volume is positive, see \Cref{lem:cvol-Alexandrov}. We start with a short review of sets of locally finite perimeter, cf.~\cite[Chapter~5]{EvansGariepy}, \cite[Section~4.5]{Federer}. 

Let $E\subset \R^3$. We define the \emph{measure theoretic boundary of} $E$ by
\begin{equation}
\partial_*E = \{x\in\R^3 \mid \theta^{*3}(\CalL^3\llcorner E,x)>0,\,\theta^{*3}(\CalL^3\llcorner (\R^3\setminus E),x)>0\}.
\end{equation} 
Moreover, we denote with $n_E\colon\R^3\to\R^3$ the \emph{measure theoretic inner unit normal of~$E$} (see the definition \cite[4.5.5]{Federer}). 
In view of \emph{Federer's criterion}~\cite[4.5.11,\,2.10.6]{Federer}, we say that $E$ has \emph{locally finite perimeter}, if and only if $E$ is an $\CalL^3$-measurable set, and $\CalH^2(K\cap\partial_*E)<\infty$ for all compact sets $K\subset\R^3$.

Let $E\subset\R^3$ be a set of locally finite perimeter and $B = \{x\in\R^3\mid \abs{n_E(x)}=1\}$. We collect the following properties (see \cite[4.5.6]{Federer}).  
\begin{itemize}
	\item The sets $B$ and $\partial_*E$ are $\CalH^2$-almost equal.
	\item $\CalH^2\llcorner\partial_*E$ is a Radon measure over $\R^3$ and $n_E$ is $\CalH^2\llcorner\partial_*E$-measurable.
%	\item The approximate tangent plane %(\cite[3.2.16]{Federer}) 
%	is given by $\Tan^2(\CalH^2\llcorner\partial_*E,x) = T(\star n_E(x))$ for $\CalH^2\llcorner\partial_*E$-almost all $x$.
	\item The divergence theorem reads as 
	\begin{equation}\label{eq:div-thm-finite-perimeter}
	-\int_{\partial_*E}\langle X,n_E\rangle\diff\CalH^2 = \int_E\divergence X\diff\CalL^3
	\end{equation} 
	for all Lipschitz maps $X\colon\R^3\to\R^3$ with compact support.
\end{itemize}

In view of Riesz's representation theorem, we define the oriented varifold $V\in\mathbb V_2^\mathrm{o}(\R^3)$ associated with~$\partial_*E$ 
%as the push forward of the Radon measure $\CalH^2\llcorner\partial_*E$ under the map
%\begin{equation}\label{eq:perimeter_varifold}
%	\R^3\to\R^3\times\mathbb G^\mathrm{o}(3,2),\quad x\mapsto(x,\star n_E(x)).
%\end{equation}
by
\begin{equation}\label{eq:perimeter_varifold}
	V(\varphi) \defeq \int_{\R^3}\varphi(x,\star n_E(x))\,\mathrm d(\mathcal H^2\llcorner\partial_*E)(x) 
\end{equation}
for all real valued continuous functions $\varphi$ on $\mathbb G^\mathrm{o}_2(\R^3)$ with compact support.
%
%In view of \cite[Lemma~2.6]{MenneScharrer2}, $V$ is indeed a Radon measure over $\mathbb G_2^\mathrm{o}(\R^3)$. 
There holds $\mu_V = \CalH^2\llcorner\partial_*E$ and \eqref{eq:div-thm-finite-perimeter} reads
%and (by \cite[2.4.18]{Federer})
%\begin{equation}
%	\int_{\mathbb G_2^\mathrm{o}(\R^3)}k(x,\xi)\diff V(x,\xi) = \int_{\partial_*E}k(x,\star n_E(x))\diff \CalH^2(x)
%\end{equation}
%whenever $k\colon\mathbb G_2^\mathrm{o}(\R^3)\to\R$ is $V$-integrable. In particular,
\begin{equation}\label{eq:Gauss-Green}
	\int_{\mathbb G_2^\mathrm{o}(\R^3)}\langle X(x),\star\xi\rangle\diff V(x,\xi) = - \int_{E}\divergence X\diff \CalL^3
\end{equation}
for all Lipschitz maps $X\colon\R^3\to\R^3$ with compact support.

This divergence theorem is the main motivation for considering a particular class of varifolds in the sequel.

\begin{hyp}\label{hyp:EmuH}
	Suppose $V\in\mathbb V_2^\mathrm{o}(\R^3)$ and $H\in L_\mathrm{loc}^1(\mu_V;\R^3)$ satisfy \Cref{hyp:generalised_mean_curvature}, $E\subset \R^3$ is an $\CalL^3$-measurable set, $\Theta\in L^1_{\mathrm{loc}}(\CalL^3\llcorner E;\N)$, 
	\begin{equation}\label{eq:diam-cond-enclosing_varifolds}
		\diam \spt (\CalL^3\llcorner E) \leq \diam \spt \mu_V,
	\end{equation}
	and
	\begin{equation} \label{eq:generalized_div_thm}
		-\int_{\mathbb G_2^\mathrm{o}(\R^3)}\langle X(x),\star\xi\rangle\diff V(x,\xi) =  \int_{E}(\divergence X) \Theta\diff \CalL^3
	\end{equation}
	for all Lipschitz maps $X\colon\R^3\to\R^3$ with compact support. In this case, we term $V$ \emph{a varifold with enclosed volume}. 
\end{hyp}

\begin{remark}
	\begin{enumerate}[(i)]
		\item If $V\in \mathbb{V}^{\mathrm{o}}_2(\R^3)$ is integral with compact support and such that the associated $2$-current has zero boundary, by \cite[4.5.17]{Federer} we find that \eqref{eq:generalized_div_thm} is satisfied for some measurable $E\subset \R^3$ and $\Theta\in L^1_{\mathrm{loc}}(\CalL^3\llcorner E;\Z)$, see also \cite[Section 3]{EichmannAGAG}. In \Cref{hyp:EmuH} we additionally require $\Theta>0$ a.e.\ on $E$, the diameter bound \eqref{eq:diam-cond-enclosing_varifolds} and that $V$ satisfies \Cref{hyp:generalised_mean_curvature}.
%			In \Cref{hyp:EmuH}, we additionally require $\Theta>0$ and \eqref{eq:dim-cond-enclosing_varifolds}
		\item The sign in Equation~\eqref{eq:generalized_div_thm} is adapted to our convention that the unit normal points to the interior.
		\item Since the divergence theorem \eqref{eq:div-thm-finite-perimeter} remains true if we replace $E$ with $\R^3\setminus E$ and $n_E$ with $-n_E$, condition \eqref{eq:diam-cond-enclosing_varifolds} ensures that we pick the correct orientation.
		\item The function $\Theta$ has \emph{locally bounded variation} (see the definition \cite[Section~5.1]{EvansGariepy}) and the coarea formula \cite[4.5.9(13)]{Federer} implies that 
		\begin{equation}
			E_k\defeq \{x\in\R^3\mid \Theta(x) \geq k\} \qquad\text{for $k\in\N$}
		\end{equation}  
		defines a sequence of decreasing sets of locally finite perimeter.
		\item If $\Theta\equiv1$, then the varifold associated with $\partial_*E$ does not necessarily coincide with~$V$, compare Figures~\ref{fig:pillow} and \ref{fig:current}.
		\item If $V$ is associated with the reduced boundary of a set $E$ of locally finite perimeter, then $q_\#V$ is an integral varifold (in the sense of \cite[3.5]{Allard}). Hence, if additionally $V$ has generalized mean curvature $H$ and vanishing singular part $\beta_V = 0$, then there holds $H(x)\wedge \star \xi = 0$ for $V$-almost all $(x,\xi)$ by \cite[Section 5.8]{Brakke}, $V,H$ satisfy \Cref{hyp:generalised_mean_curvature} and thus $V,H,E$ and $\Theta\equiv1$ satisfy \Cref{hyp:EmuH}.
	\end{enumerate}
\end{remark}

As the following example shows, not all varifolds associated with sets of finite perimeter satisfy \Cref{hyp:EmuH}.

\begin{example}\label{ex:lens}
	Let $C_\alpha$ be the closed spherical cap of the unit sphere with opening angle $\alpha = \pi/3$ (the hemisphere has opening angle $\pi/2$) whose boundary circle lies in the plane $\{z=0\}$. Let $S = C_\alpha \cup (-C_\alpha)$, i.e.\ $S$ is the gluing of the spherical cap $C_\alpha$ with its reflection at the plane $\{z=0\}$. The surface $S$ looks like a lens, see Figure~\ref{fig:lens}. Its singular part is the circle $\Gamma_a$ of radius $a=\sqrt{3}/2$ centered at the origin and lying in the plane $\{z=0\}$. 
	Since $\CalH^2(S)<4\pi<\infty$, one can use Federer's criterion to show that $S$ is the boundary of a set $E$ of finite perimeter. 
	However, the varifold $V$ associated with $S=\partial_*E$ (cf.\ \eqref{eq:perimeter_varifold}) does not satisfy \Cref{hyp:EmuH}. In fact, $V$ does not satisfy \Cref{hyp:generalised_mean_curvature}, but the more general \Cref{hyp:generalised_mean_curvature_bdr}. 
	%To be more precise, let $V$ be the oriented $2$-varifold associated to $\partial_*E$ as defined in~\eqref{eq:perimeter_varifold}. Then, 
	Indeed, in view of \cite[4.4,\,4.7]{Allard}, there holds $\mu_V = \CalH^2\llcorner S$, $\beta_V = \sqrt{3}\CalH^1\llcorner \Gamma_a$, and
	\begin{equation}
	\delta V(X) = -\int_{S}\langle X,H\rangle\diff\CalH^2 + \sqrt{3}\int_{\Gamma_a}\frac{\langle X(x),x\rangle}{|x|} \diff \CalH^1(x)
	\end{equation}	
	where $H$ is the mean curvature of the spherical caps $\pm C_\alpha$. Notice that $\theta^1(\beta_V,x)=\sqrt{3}$ for all $x\in\Gamma_a$. In other words, $\beta_V$ does not have integer multiplicity even though $\theta^2(\mu_V,x) = 1$ for all $x\in S$. Notice also that $V$ satisfies the hypothesis of \Cref{thm:LY varifold} for all $c_0\in \R$. 
	\begin{comment}
	Taking $x_0 = (0,0,1/2)$ to be the north pole of the upper spherical cap and $c_0 = 0$, we have
	\begin{equation}
	\theta^2(\mu_V,x_0) = 1, \qquad \theta^{*2}(\mu_V,x_0) = 0, \qquad \frac{1}{4\pi}\CalH_{c_0}(V) = \frac{1}{4\pi}\CalW(V) = \frac{1}{2}
	\end{equation}
	and
	\begin{align}
	&\frac{1}{2\pi} \limsup_{\rho\to\infty}\left(\int_{B_\rho}(\abs{x-x_0}^{-2}-\rho^{-2})\langle x-x_0, \eta(x)\rangle\diff \beta_V(x)\right) \\
	&\quad = \frac{1}{2\pi} \int_{\Gamma_a}\frac{\langle x-x_0, \frac{x}{|x|}\rangle}{|x-x_0|^2}\diff \beta_V(x) = \frac{\sqrt{3}}{2\pi}\int_{\Gamma_a}\frac{\sqrt{3}}{2}\diff \CalH^1 = \frac{3\sqrt{3}}{4}.
	\end{align}
	\end{comment}
\end{example}

The set $E$ in \Cref{hyp:EmuH} corresponds to an enclosed volume in the following sense, where the algebraic volume does not depend on the point $x_0\in \R^3$.

\begin{prop}\label{prop:algebraic_volume}
	Suppose $V, H, E, \Theta$ satisfy \Cref{hyp:EmuH} with $\spt\mu_V$ compact. Then 
	\begin{align}
		\CalV(V, x_0)= \int_{E}\Theta\diff\CalL^3 \eqdef\CalV(V)\qquad \text{ for all }x_0\in \R^3.
	\end{align}
\end{prop}
\begin{proof}
	Since $\spt\mu_V$ is compact, so is $\spt(\CalL^3\llcorner E)$ by \Cref{hyp:EmuH}. We may thus apply \eqref{eq:generalized_div_thm} with $X(x)=x-x_0$, suitably cutoff away from $\spt\mu_V$ and $\spt(\CalL^3\llcorner E)$.
\end{proof}

\begin{comment}
\begin{example}\label{ex:currents}
	Suppose $\R^3\eqdef E_0\supset E_1\supset E_2\supset\ldots$ is a sequence of sets of locally finite perimeter, $V_k\in\mathbb V_2^\mathrm{o}(\R^3)$ are the oriented varifolds associated with $\partial_* E_k$, and suppose $V\defeq\sum_{k=1}^\infty V_k$ is again a Radon measure (i.e. $V\in\mathbb V_2^\mathrm{o}(\R^3)$). Define the $\CalL^3$-measurable function 
	\begin{equation}
		\Theta\colon\R^3\to \mathbb N,\qquad \Theta(x) = k \text{ whenever $x\in E_k\setminus E_{k+1}$}.
	\end{equation}
	Then, there holds $E_k = \{x\in\R^3\mid \Theta(x)\geq k\}$ and a simple cut-off procedure implies
	\begin{equation}
		-\int_{\mathbb G_2^\mathrm{o}(\R^3)}\langle X(x),\star\xi\rangle\diff V(x,\xi) = \int_{\R^3}(\divergence X)\Theta\diff \CalL^3
	\end{equation}
	for all $X\in C_{c}^1(\R^3;\R^3)$. Moreover, by \cite[3.5,~Theorem~(2)]{Allard} and \cite[4.5.6]{Federer},
	\begin{equation}
		\theta^2(\mu_V,x) = \sum_{k=1}^\infty\theta^2(\mu_{V_k},x)=\CalH^0(\{k\in\N\mid x\in \partial_*E_k\})
	\end{equation}
	for $\CalH^2$-almost all $x\in\R^3$.
\end{example}

\begin{remark}
	In view of \cite[4.5.9(13), 4.5.17]{Federer}, one may also think of $V$ as the oriented varifold associated with the boundary of a normal $3$-current in $\R^3$ with nonnegative integer valued density function.
\end{remark}
\end{comment}
	
%For varifolds with enclosed volume, the third term on the right hand side of \eqref{eq:LY Helfrich Varifold} can be expressed as a limit of a volume integral. 
Under suitable assumptions, the concentrated volume can be computed by \eqref{eq:generalized_div_thm}, too.

\begin{lem}\label{lem:cvol-Alexandrov}
	Suppose $V,H,E,\Theta$ satisfy \Cref{hyp:EmuH}. Let $x_0\in\R^3$ and assume
	\begin{equation}\label{eq:hyp:cvol-Alexandrov}
		\lim_{\rho\to\infty}\frac{1}{\rho^2}\int_{E\cap B_\rho(x_0)}\Theta\diff\CalL^3 =0.
	\end{equation}	
	Then we have
	\begin{equation}\label{eq:cvol-AlexandrovTheta}
		\CalV_c(V,x_0) = \int_E\frac{\Theta(x)}{|x-x_0|^2}\diff\CalL^3(x),
	\end{equation}
	provided both sides exist.
\end{lem}

\begin{remark}\label{rem:cvol-Alexandrov}	
%	As a consequence of the Lebesgue differentiation theorem, the condition~\eqref{eq:hyp:cvol-Alexandrov} is satisfied for $m=3$ and $\CalL^3$-almost all $x_0\in E$ (cf. \cite[2.9.8]{Federer}). However, not all $\Theta\in L^1_{\mathrm{loc}}(\CalL^3\llcorner E,\N)$ and $x_0\in\R^3$ satisfy~\eqref{eq:hyp:cvol-Alexandrov}. This can be seen by taking $\Theta(x)$ to be the integer part of $|x-x_0|^{-2}$.
	\begin{enumerate}[(i)]
		\item \label{item:exist_vol_exist_cvol} By \Cref{prop:cvol existence}, if $V\in \mathbb{V}^{\mathrm{o}}_2(\R^3)$ and $H\in L^2_{\mathrm{loc}}(\mu_V;\R^3)$ satisfy \Cref{hyp:generalised_mean_curvature} and if $\CalV(V,x_0)$ exists, then also $\CalV_c(V,x_0)$ exists.
		
		\item\label{item:m_density_Theta} Suppose $\int_{E}\Theta\diff\CalL^3<\infty$.
		By \Cref{lem:subcrit_integral} applied to the measure $\Theta\CalL^3\llcorner E$, the right hand side of \eqref{eq:cvol-AlexandrovTheta} exists if for some $m>2$ we have
		\begin{align}\label{eq:hyp:cvol-Alexandrov3}
			\limsup_{\sigma\to 0+} \frac{1}{\sigma^m}\int_{E\cap B_\sigma(x_0)}\Theta\diff\CalL^3 <\infty.
		\end{align}
		As a consequence of the Lebesgue differentiation theorem, this is true for $m=3$ and $\CalL^3$-almost all $x_0\in E$ (cf.\ \cite[2.9.8]{Federer}). However, not all $\Theta\in L^1_{\mathrm{loc}}(\CalL^3\llcorner E;\N)$ and $x_0\in\R^3$ satisfy~\eqref{eq:hyp:cvol-Alexandrov3}. This can be seen by taking $\Theta(x)\defeq \lceil \abs{x-x_0}^{-2}\rceil$.
		%$ to be the integer part of $|x-x_0|^{-2}$.
		 Nevertheless, \eqref{eq:hyp:cvol-Alexandrov3} is clearly satisfied if $\Theta\in L^{\infty}(\CalL^3\llcorner E;\N)$.
	\end{enumerate}
\end{remark}

\begin{proof}[Proof of \Cref{lem:cvol-Alexandrov}]
%	By \eqref{eq:bound_Willmore_by_Helfrich}, \Cref{lem:subcrit_integral} and \Cref{rem:subcrit_integral}, we find that
%	\begin{align}\label{eq:vol alexandrov 1}
%		\int_{\R^3}\frac{1}{\abs{x-x_0}}\diff\mu_V(x) <\infty.
%	\end{align}
	Since $\CalV_c(V,x_0)$ exists, we have
	\begin{align}\label{eq:vol alexandrov 1}
		\int_{\mathbb{G}^{\mathrm{o}}_2(\R^3)}\frac{\abs{\langle x-x_0, \star\xi\rangle}}{\abs{x-x_0}^2}\diff V(x,\xi)<\infty.
	\end{align}
	Now, let $0<\sigma<\rho$ and let $B_\rho, A_\rho$ be as in \Cref{lem:monotonicity}. Moreover, let $\varphi$ be as in \eqref{eq:varphi}, $X(x)\defeq x-x_0$, $\Phi(x)\defeq\varphi(\abs{x-x_0})$ for $x\in \R^3$. For $\CalL^3$-almost every $x\in \R^3$ we find
	\begin{align}
		\divergence \big(\Phi X\big)(x)
		& = \left\lbrace\begin{array}{ll}
			3(\sigma^{-2}-\rho^{-2}) & \text{ for } x\in B_\sigma \\
			|X(x)|^{-2}- 3\rho^{-2}& \text{ for } x\in B_\rho\setminus B_\sigma\\
			0 & \text{ for } x\in \R^3\setminus B_\rho.
		\end{array}\right.
	\end{align}
	Thus \eqref{eq:generalized_div_thm} implies
	\begin{align}
		& -\frac{1}{\sigma^2}\int_{A_\sigma}\langle X(x), \star\xi\rangle\diff V(x,\xi)+ \frac{1}{\rho^2}\int_{A_\rho}\langle X(x),\star\xi\rangle\diff V(x,\xi) - \int_{A_\rho\setminus A_\sigma}\frac{\langle X(x), \star\xi\rangle}{\abs{X(x)}^2}\diff V(x,\xi) \\
		&\quad= \frac{3}{\sigma^2}\int_{E\cap B_\sigma}\Theta\diff\CalL^3- \frac{3}{\rho^2}\int_{E\cap B_\rho}\Theta\diff\CalL^3 + \int_{E\cap B_\rho\setminus B_\sigma} \frac{\Theta(x)}{\abs{X(x)}^2}\diff\CalL^3(x).\label{eq:perimeter proof 1}
	\end{align}
	We analyze each term in \eqref{eq:perimeter proof 1} separately.
	First, as $\sigma\to 0+$, the first term on the left vanishes, since \eqref{eq:vol alexandrov 1} yields
	\begin{align}
		\frac{1}{\sigma^2}\int_{A_\sigma}\abs{\langle X(x), \star\xi\rangle}\diff V(x,\xi)\leq \int_{A_\sigma}\frac{\abs{\langle X(x), \star\xi\rangle}}{\abs{X(x)}^2}\diff V(x,\xi)\to 0.
	\end{align}
	Here we used that $\mu_V(\{x_0\})=0$ by \Cref{rem:subcrit_integral}\eqref{item:subcrit_loc}.
	The first term on the right hand side of \eqref{eq:perimeter proof 1} goes to zero as $\sigma\to 0$, since the right hand side of \eqref{eq:cvol-AlexandrovTheta} exists. For the second term on the left, taking $0<r<\rho$ and splitting the integral we obtain 
	\begin{align}
		&\frac{1}{\rho^2}\int_{A_\rho}\abs{\langle X(x), \star\xi\rangle}\diff V(x,\xi) \\
		&\qquad \leq \int_{\mathbb{G}^{\mathrm{o}}_2(\R^3)\setminus A_r} \frac{\abs{\langle X(x), \star\xi\rangle}}{\abs{X(x)}^2}\diff V(x,\xi) + \frac{1}{\rho^2} \int_{A_r}\abs{\langle X(x), \star\xi\rangle}\diff V(x,\xi),
	\end{align}
	which goes to zero by \eqref{eq:vol alexandrov 1}, if we send first $\rho\to\infty$ and then $r\to\infty$.
	Taking $\rho\to\infty$ the second term on the right of \eqref{eq:perimeter proof 1} vanishes by \eqref{eq:hyp:cvol-Alexandrov}. Thus, if we let first $\sigma\to 0$ and then $\rho \to\infty$ in \eqref{eq:perimeter proof 1} and use that both sides of \eqref{eq:cvol-AlexandrovTheta} exist, the claim follows.
\end{proof}

By the preceding discussion, the statement of \Cref{cor:LY varif area dec} can be simplified if $V$ is a varifold with enclosed volume. For simplicity, we only consider the case where $\spt\mu_V$ is compact.

\begin{cor}\label{cor:LY finite perimeter}
	Suppose $V,H,E, \Theta$ satisfy \Cref{hyp:EmuH} with $\spt\mu_V$ compact. Then %, $H\in L^{2}(\mu_V;\R^3)$ and $\Theta\in L^{\infty}(\CalL^3\llcorner E;\N)$. 
	%Then for all $x_0\in \R^3$ we have
	\begin{align}\label{eq:LY varifold encl vol}
		\theta^2(\mu_V,x_0)\leq \frac{1}{4\pi} \CalH_{c_0}(V) + \frac{c_0}{2\pi}\int_{E}\frac{\Theta(x)}{\abs{x-x_0}^2}\diff\CalL^3(x)
	\end{align}
	for all $x_0\in \R^3$, provided the second term on the right hand side exists.
\end{cor}
\begin{proof}
	By \Cref{hyp:EmuH} we find that $\spt(\CalL^3\llcorner E)$ is compact, so that using \Cref{rem:H L2 loc} and \Cref{rem:cvol-Alexandrov}\eqref{item:exist_vol_exist_cvol} we find that the assumptions of \Cref{cor:LY varif area dec} and \Cref{lem:cvol-Alexandrov} are satisfied. The result then directly follows using \eqref{eq:cvol-AlexandrovTheta}.
\end{proof}

	\section{The smooth setting}\label{sec:smooth}
In this section, we will transfer the general varifold Li--Yau inequalities to the setting of smoothly immersed surfaces.

\subsection{Proofs of the Li--Yau inequalities}

\Cref{thm:LY smooth} is an easy consequence of the varifold result.

\begin{proof}[Proof of \Cref{thm:LY smooth}]
	The claim follows directly from \Cref{cor:LY varif area dec} if we consider the varifold associated to the immersion $f$, cf.\ \Cref{ex:varifold_of_immersion}.
\end{proof}

We now show that any Alexandrov immersion induces a varifold with enclosed volume.

\begin{lem} \label{lem:Alexandrov_immersions}
	Let $\Sigma$ be a closed surface and let $f\colon \Sigma \to\R^3$ be an Alexandrov immersion with $\Sigma=\partial M$, $f=F\vert_{\Sigma}$ and $F\colon M\to\R^3$ as in \Cref{defi:Alexandrov_immersions}. Let 
	$V$ be the oriented $2$-varifold on $\R^3$ associated to $(\Sigma,f)$ as in \Cref{ex:varifold_of_immersion}. Then, there holds
	\begin{equation}
		- \int_{\mathbb G_2^\mathrm{o}(\R^3)} \langle X(x), \star\xi \rangle \diff V(x,\xi) = \int_{F[M]} (\divergence X)(x) \CalH^{0}(F^{-1}\{x\}) \diff \CalL^3(x)  
	\end{equation}
	for all Lipschitz $X\colon \R^3\to\R^3$ with compact support. In particular, with $E\defeq F[M]$, $\Theta\defeq \CalH^{0}(F^{-1}\{\cdot\})$ we see that $V, H, E, \Theta$ satisfy \Cref{hyp:EmuH}.
\end{lem}	

\begin{proof} By an approximation argument, it suffices to consider $X\in C_c^{1}(\R^3;\R^3)$.
	Denote with $\Omega$ the Riemannian measure on $M$ induced by the pullback metric $g_F\defeq F^*\langle\cdot,\cdot\rangle$, let $\mu$ be the induced measure on~$\Sigma$, and let $\nu$ be the inner unit normal on $\Sigma$. Given any vector field $X \in C^1(\R^3;\R^3)$, we define the vector field $X^*$ on $M$ by $X^*(p) = (\mathrm d F_p)^{-1}(X(F(p)))$. By~\eqref{eq:integral_push_forward} and since $n=\diff F(\nu)$, we compute 
	\begin{align}
		-\int_{\mathbb G^\mathrm{o}_2(\R^3)} \langle X(x),\star\xi\rangle\diff V(x,\xi) & = - \int_{\Sigma} \langle X\circ f, n \rangle\diff \mu = \int_{\partial M}g_F(X^*,-\nu)\diff \mu.
	\end{align}
	Since $(M,g_F)$ is flat, we have $\divergence_{g_F} X^* = (\divergence X) \circ F$. Hence, by the divergence theorem for Riemannian manifolds (see \cite[Theorem~5.11(2)]{Sakai}) and the area formula,
	\begin{align}
		\int_{\partial M}g_F(X^*,-\nu)\diff \mu = \int_M (\divergence X)\circ F\diff \Omega = \int_{F[M]}(\divergence X)(x)\CalH^0(F^{-1}\{x\})\diff \CalL^3(x)	
	\end{align} 
	which implies the conclusion. 
\end{proof}

Equipped with this tool we can now prove \Cref{thm:LY Alexandrov}.

\begin{proof}[Proof of \Cref{thm:LY Alexandrov}]
	By \Cref{lem:Alexandrov_immersions}, $V, H, E\defeq F[M], \Theta\defeq \CalH^{0}(F^{-1}\{\cdot\})$ satisfy \Cref{hyp:EmuH}. Since $M$ is compact and $F$ is a local diffeomorphism, there exists $k\in \N$ such that
	\begin{align}\label{eq:preimage bound}
		\Theta(x)=\CalH^{0}(F^{-1}\{x\}) \leq k \quad \text{for all }x\in E=F[M],
	\end{align}
	and as a consequence of \Cref{lem:cvol-Alexandrov} and \Cref{rem:cvol-Alexandrov}\eqref{item:exist_vol_exist_cvol} and \eqref{item:m_density_Theta} we find
	\begin{align}\label{eq:cvol Alexandrov}
		\CalV_c(f, x_0) = \int_{F[M]}\frac{\CalH^{0}(F^{-1}\{x\})}{\abs{x-x_0}^2}\diff\CalL^3(x) \quad \text{for all }x_0\in \R^3.
	\end{align}
	The statement then follows from \Cref{cor:LY finite perimeter}.
%	Let $x_0\in \R^3$. We obtain
%	\begin{align}
%		\frac{1}{\sigma^2}\int_{E\cap B_\sigma(x_0)}\Theta\diff\CalL^3\leq \frac{4\pi k}{3}\sigma \to 0 \quad \text{as }\sigma\to 0+,\label{eq:vol decay 0}
%	\end{align}
%	by \eqref{eq:preimage bound}, and similarly 
%	\begin{align}
%		\frac{1}{\rho^2}\int_{E\cap B_\rho(x_0)}\Theta\diff\CalL^3 \leq \frac{k \CalL^3(F[M])}{\rho^2}\to 0 \quad \text{as }\rho\to\infty. 
%	\end{align}
%	Hence, \eqref{eq:volume decay} is satisfied and the statement follows from \Cref{cor:LY finite perimeter}.
\end{proof}

\begin{remark}
%	\begin{enumerate}[(i)]
%		\item Equation \eqref{eq:preimage bound} together with \Cref{lem:Alexandrov_immersions} and \Cref{rem:cvol-Alexandrov} implies
%		\begin{align}\label{eq:cvol Alexandrov}
%			\CalV_c(f, x_0) = \int_{F[M]}\frac{\CalH^{0}(F^{-1}(\{x\}))}{\abs{x-x_0}^2}\diff\CalL^3(x) \quad \text{for all }x_0\in \R^3
%		\end{align}
%		 for all Alexandrov immersions $f=F\vert_{\Sigma}$ with $F\colon M\to\R^3$, $\Sigma=\partial M$ as in \Cref{defi:Alexandrov_immersions}.

		The results of \Cref{thm:LY smooth} and \Cref{thm:LY Alexandrov} are sharp in the sense that equality can be achieved asymptotically for every $c_0\in \R$. Indeed, let $\S^2\subset \R^3$ be the unit sphere, and let $f\colon \S^2\to\R^3, f(x)=rx$ denote the parametrization of the round sphere $\partial B_r(0)\subset \R^3$ with radius $r>0$ and the orientation given by the inner unit normal. This is clearly an Alexandrov immersion (with $M=B_1(0)$, $F(x)=rx$) and hence by \eqref{eq:cvol Alexandrov}, we have
		\begin{align}\label{eq:smiley}
			\CalV_{c}(f,x_0) &=  \int_{B_r(0)}\frac{1}{\abs{x-x_0}^2}\diff\CalL^3(x) = \begin{cases}
				2\pi r & \text{ if }x_0\in \partial B_r(0),\\
				4\pi r & \text{ if }x_0=0.
			\end{cases}
		\end{align}	
			To verify the last equality in \eqref{eq:smiley} we use the original surface integral definition \eqref{eq:def conc vol} for the concentrated volume and that $n=-f/r$. The claim then follows from
		\begin{align}
			\CalV_c(f,x_0)= \frac{1}{r}\int_{\S^2}\frac{\langle f-x_0,f\rangle}{\abs{f-x_0}^2}\diff\mu= \frac{1}{r}\int_{\S^2} \frac{r^2-\langle x_0, f\rangle}{r^2+|x_0|^2-2\langle f, x_0\rangle}\diff \mu.
		\end{align}
	
%		\changed{Then the above integral equals
%		\begin{align}
%			\CalV_c(f,x_0)=2\pi\int_0^1 \int_0^\pi \frac{s^2 \sin \theta}{s^2-2s \cos\theta + 1}\diff\theta\diff s = 2\pi.
%		\end{align}
%		Indeed, 
%		after substituting $u=-\cos\theta$ and using $s\in (0,1)$, the integral equals
%		\begin{align}
%			\int_0^1\int_{-1}^1 \frac{s^2}{s^2+2su +1}\diff u\diff s = \int_0^1 s \log\frac{1+s}{1-s}\diff s = \left.\frac{s^2}{2}\log\frac{1+s}{1-s}+s - \operatorname{arctanh}(s)\right\vert_{0}^1 =1,
%		\end{align}
%		where we integrated by parts and applied partial fraction decomposition.}
		If now $x_0\in \partial B_r(0)$, Inequality \eqref{eq:LY Immersion} reads
		\begin{align}
			1=\CalH^{0}(f^{-1}\{x_0\})\leq \frac{1}{4\pi}\CalH_{c_0}(f) + \frac{c_0}{2\pi}\CalV_c(f,x_0) = \frac{1}{4}\left(c_0r-2\right)^2 + c_0 r \quad \text{ for all }r>0,
		\end{align}
		where the right hand side converges to $1$ as $r\to 0+$. In the case $c_0=0$, equality is achieved by any round sphere.
%	\end{enumerate}
\end{remark}

\subsection{A scale-invariant version}\label{sec:scale invariant}

Clearly, for $x_0=0$ the left hand sides of the Li--Yau inequalities in \Cref{thm:LY smooth} and \Cref{thm:LY Alexandrov} are invariant under rescalings of the immersion, whereas the right hand sides are not. We will now prove a scale-invariant version of the inequality, involving the  \emph{$L^2$-CMC-deficit} of an immersion $f\colon \Sigma\to\R^3$ of an oriented surface $\Sigma$, given by
\begin{align}\label{eq:def cmc deficit}
	\bar{\CalH}(f) = \frac{1}{4}\int_{\Sigma}(H_{\mathrm{sc}}-\bar H_{\mathrm{sc}})^2\diff\mu =\inf_{c_0\in \R}\CalH_{c_0}(f),
\end{align}
cf.\ \eqref{eq:intro:CMC deficit}. Here $\bar H_{\mathrm{sc}}\defeq \CalA(f)^{-1} {\int_{\Sigma} H_{\mathrm{sc}}\diff\mu}$ denotes the \emph{average scalar mean curvature}, provided the latter integral exists. 
Note that $\bar{\CalH}(f) = 0$ if and only if $f$ is an immersion with constant mean curvature, a \emph{CMC-immersion}, justifying the terminology. 
%If $\Sigma=\S^2$ or if $\Sigma$ is closed and $f$ is an embedding, $\bar{\CalH}(f)=0$ is only possible if $f$ parametrizes a round sphere by the classical results of Alexandrov \cite{Aleksandrov} and Hopf \cite[Chapter VI]{Hopf}. This is no longer true for immersed surfaces of higher genus, as the example of the famous Wente torus \cite{Wente} shows.
%Note that by the $L^2$-minimizing property of the mean, $\bar\CalH(f)$ measures the $L^2(\mu)$-distance of the mean curvature to the subspace of constant functions in $L^2$ and thus
%\begin{align}
%	\bar{\CalH}(f) =  \inf_{c_0\in \R}\CalH_{c_0}(f).
%\end{align}
We obtain the following Li--Yau inequality which is invariant under rescaling and also under reversing the orientation on $\Sigma$.
\begin{cor}\label{thm:LY scale inv}
	Let $f\colon \Sigma\to\R^3$ be an immersion of a compact oriented surface $\Sigma$ without boundary. Then for all $x_0\in \R^3$ we have
	\begin{align}
		\CalH^{0}(f^{-1}\{x_0\})\leq \frac{1}{4\pi} \bar{\CalH}(f) + \frac{1}{2\pi} \bar H_{\mathrm{sc}} \CalV_{c}(f,x_0) - \frac{1}{\pi \CalA(f)}\left(\CalV_c(f,x_0)\right)^2.\label{eq:LY scale inv}
	\end{align}
\end{cor}

\begin{proof}%[Proof of \Cref{thm:LY scale inv}]
	By \Cref{prop:cvol existence}, we find that $\CalV_c(f,x_0)$ exists. We may thus use \Cref{thm:LY smooth} for any $c_0\in\R$. Expanding the right hand side of \eqref{eq:LY Immersion}, we obtain a quadratic polynomial in $c_0$. By a direct computation, this polynomial is minimal for $ c_0 =\frac{\int_\Sigma H_{\mathrm{sc}}\diff\mu - 4 \CalV_c(f,x_0)}{\CalA(f)}$ and the minimal value is precisely the right hand side of \eqref{eq:LY scale inv}.
\end{proof}

	\section{Applications}\label{sec:applications}

In this section, we discuss several applications of the Li--Yau inequalities. We first provide a lower bound on the Helfrich energy resulting in nonexistence of minimizers for the penalized Canham--Helfrich model in \Cref{subsec:nonexistence}. In \Cref{subsec:diameter} we prove some important geometric estimates involving the Helfrich energy. We then use these to prove \Cref{thm:regularity_Helfrich_problem}. Lastly, we discuss a criterion for positive total mean curvature in \Cref{subsec:pos tot mean}.

\subsection{Nonexistence of minimizers for the penalized Canham--Helfrich model}\label{subsec:nonexistence}

\begin{lem}\label{thm:Willmore-inequality}
	Suppose $V\in\mathbb V_2^\mathrm{o}(\R^3)$ and $H\in L_\mathrm{loc}^2(\mu_V)$ satisfy \Cref{hyp:generalised_mean_curvature}, $\spt\mu_V$ is compact, $c_0 < 0$, and $x_0\in\R^3$ such that $\theta^{*2}(\mu_V,x_0) \geq 1$ and $\CalV_c(V,x_0)>0$. Then there holds
	\begin{equation}\label{eq:Willmore_inequality}
		\CalH_{c_0}(V) > 4\pi.
	\end{equation}
\end{lem}

\begin{proof}%[Proof of \Cref{thm:Willmore-inequality}]
	This is a consequence of \Cref{cor:LY varif area dec} in combination with \cite[Theorem~3.6]{Scharrer2}.	
\end{proof}

\begin{remark}
	The proof of the above inequality for the Willmore functional (i.e.\ $c_0=0$) \cite[Theorem~7.2.2]{WillmoreRG} also works for the Helfrich functional provided $V$ is given by an Alexandrov immersion $f:\Sigma\to\R^3$ with inner unit normal field $n$. Indeed, denoting with $K^+$ the set of points in $\Sigma$ where both principal curvatures are nonnegative, we find
	\begin{align}
		\CalH_{c_0}(f) & \geq \frac{1}{4}\int_{K^+}|H_f - c_0n|^2\diff \mu \\
		& \geq \frac{1}{4}\int_{K^+}|H_f|^2\diff \mu + \frac{c_0^2}{4}\CalA(f) > \frac{1}{4}\int_{K^+}|H_f|^2\diff\mu \geq \int_{K^+}K\diff \mu
	\end{align} 	
	where $K$ denotes the Gauss curvature. Similarly to \cite[Lemma~7.2.1]{WillmoreRG} we see that if $f$ is an Alexandrov immersion, then
	\begin{equation}
		\int_{K^+}K\diff \mu \geq 4\pi.
	\end{equation}
%	See also the proof of \cite[Proposition 1.1.1]{KSLectureNotes} for more details.
\end{remark}

For all real numbers $c_0,\lambda,p$ we define the energy functional
\begin{equation}
	\CalH^{\lambda,p}_{c_0}(f) \defeq \CalH_{c_0}(f) + \lambda\CalA(f) + p\CalV(f)
\end{equation}
for all smooth immersions $f\colon\Sigma\to\R^3$ of a closed oriented surface $\Sigma$. The constants $\lambda$ and $p$ are referred to as \emph{tensile stress} and \emph{osmotic pressure}. The energy was considered by Zhong-Can and Helfrich~\cite[Equation~(1)]{ZhongHelfrich87} in the study of spherical vesicles. Each minimizer of the constrained minimization \Cref{problem:CH} is a critical point of the functional $\CalH^{\lambda,p}_{c_0}$ for some $\lambda,p$ by the method of Lagrange multipliers. This is one of the reasons why the energy $\CalH^{\lambda,p}_{c_0}$ is subject of numerous works in mathematical physics, biology and mathematics, see for instance \cite{BernardWheelerWheeler} and the references therein.%\cite{Svetina189,ZhongHelfrich89,Steigmann,BernardWheelerWheeler}.

Denote with $\mathcal{S}^\infty$ the set of smoothly embedded spheres in $\R^3$. In view of \eqref{eq:intro:convergence_Helfrich_to_Willmore}, we see that
\begin{equation}\label{eq:upper_bound_infimum}
	\inf_{f\in\mathcal S^\infty}\CalH^{\lambda,p}_{c_0}(f) \leq 4\pi.
\end{equation}
In \cite[Theorem~3]{ScharrerPhD} (see also \cite[Theorem~1.9]{MondinoScharrer1}) the existence of spheres minimizing $\CalH^{\lambda,p}_{c_0}$ was shown, provided $\lambda,c_0>0$ and $p\geq0$. However, in view of~\cite{DeulingHelfrich}, $c_0<0$ is empirically more relevant in the study of red blood cells. \Cref{thm:Willmore-inequality} now reveals that the infimum in~\eqref{eq:upper_bound_infimum} is not attained whenever $c_0<0$ and $\lambda,p \geq 0$. This is actually in accordance with the results on the gradient flow in~\cite{McCoyWheeler,BlattHelfrich}. Notice also the different behaviour of the constrained gradient flow~\cite{RuppIso}.
Examining the scaling behaviour of $\CalH^{\lambda,p}_{c_0}$ evaluated at round spheres, we see that the energy functional is unbounded from below if $p<0$;
in particular, the infimum in~\eqref{eq:upper_bound_infimum} is not attained. Similarly, if $\lambda<0$ and $c_0^2+\lambda<0$, one can use surfaces of degenerating isoperimetric ratio found in \cite[Theorem 1.5]{Scharrer} %,Scharrer,KusnerMcGrath} 
to construct a sequence of embeddings $f_k$ in $\mathcal{S}^\infty$ such that $\CalH^{\lambda,p}_{c_0}(f_k) \to-\infty$ as $k\to\infty$.  

Despite the nonexistence of minimizers explained above, the energy functional $\CalH^{\lambda,p}_{c_0}$ remains an important subject of study, since it is the critical points of $\CalH^{\lambda,p}_{c_0}$ that are of interest.

\subsection{Diameter estimates}\label{subsec:diameter}

In this section, we will show that the Helfrich energy can be used to obtain bounds on the diameter.

\begin{comment}
\begin{defi}\label{def:volume}
	Suppose $V\in\mathbb V_2^\mathrm{o}(\R^3)$ has compact support. Then we define the \emph{algebraic volume of $V$ at $x_0\in\R^3$} by
	\begin{equation}
		\CalV(V,x_0)\defeq -\frac{1}{3}\int_{\mathbb{G}^\mathrm{o}_2(\R^3)}\langle x-x_0,\star\xi\rangle\diff V(x,\xi).
	\end{equation}
	If $V$ is associated to an immersion $f\colon \Sigma\to\R^3$ of a closed surface $\Sigma$ as in \Cref{ex:varifold_of_immersion}, a short computation reveals that $\CalV(V,x_0)=\CalV(f)$ for all $x_0\in \R^3$. 
\end{defi}

\begin{remark}
	In general, the algebraic volume of an oriented varifold depends on the point $x_0$. Indeed, one may consider the varifold associated to the $2$-dimensional unit sphere in $\R^3$ where the upper hemisphere is oppositely oriented to the lower hemisphere.
\end{remark}
\end{comment}

\begin{lem}\label{thm:diameter_bound}
	Suppose $V\in \mathbb{V}^{\mathrm{o}}_2(\R^3)$ and $H\in L^{2}_\mathrm{loc}(\mu_V;\R^3)$ satisfy \Cref{hyp:generalised_mean_curvature}, $\spt\mu_V$ is compact, and $\CalH_{c_0}(V)>0$. Then for all $x_0\in \spt\mu_V$ we have
	\begin{align}\label{eq:diam bound below}
		\frac{\abs{2 \mu_V(\R^3) -3c_0\CalV(V,x_0)}}{2\sqrt{\mu_V(\R^3)\CalH_{c_0}(V)}}\leq \diam\spt\mu_V.
	\end{align}	
\end{lem}

\begin{proof}%[Proof of \Cref{thm:diameter_bound}]
	Using \Cref{hyp:generalised_mean_curvature} for the vector field $X(x)=x-x_0$ (multiplied with a suitable cut-off function away from $\spt\mu_V$), we have
	\begin{align}
		\int_{\R^3}2\diff\mu_V(x)-3c_0\CalV(V,x_0) &= - \int_{\mathbb{G}^{\mathrm{o}}_2(\R^3)}\langle H(x),x-x_0\rangle\diff V(x,\xi)\\
		&\quad +c_0 \int_{\mathbb{G}^{\mathrm{o}}_2(\R^3)}\langle \star\xi,x-x_0\rangle\diff V(x,\xi).
	\end{align}
	Thus, by the Cauchy--Schwarz inequality
	\begin{align}
		\Abs{2\mu_V(\R^3)-3c_0\CalV(V,x_0)}&\leq \int_{\mathbb{G}^{\mathrm{o}}_2(\R^3)}\abs{H(x)-c_0(\star\xi)}\abs{x-x_0}\diff V(x,\xi)\\
		&\leq \sqrt{4\CalH_{c_0}(V)\mu_V(\R^3)}\diam\spt\mu_V. &&\qedhere
	\end{align}
\end{proof}
In the case $c_0=0$, this is just Simon's lower diameter estimate, cf.\ \cite[Lemma 1.1]{SimonWillmore}.
Note that here we did not use the Li--Yau inequality but merely the first variation formula, see \eqref{eq:first_variation} and \eqref{eq:generalised_mean_curvature}. 

%Next, we extend the upper diameter bound in \cite[Lemma 1.1]{SimonWillmore} to our situation.

\begin{lem}\label{thm:upper_diameter_bound}
	Suppose $V,H,E,\Theta$ satisfy \Cref{hyp:EmuH}, $\theta^2(\mu_V,x)\geq 1$ for $\mu_V$-almost all~$x$, $\spt\mu_V$ is connected, and $c_0\leq0$. If $\CalH_{c_0}(V) < \infty$, $\mu_V(\R^3)<\infty$, $\Theta\in L^1(\CalL^3\llcorner E)$, and
	\begin{equation}\label{eq:concentrated_Theta_volume}
		\int_E\frac{\Theta(x)}{|x-x_0|^2}\diff\CalL^3(x) < \infty 
	\end{equation} 
	for $\mu_V$-almost all $x_0$, then $\spt\mu_V$ is compact and
	\begin{equation}\label{eq:upper_diameter_bound}
		\diam \spt \mu_V \leq C \sqrt{\CalH_{c_0}(V)\Bigr(\mu_V(\R^3) + \frac{2}{3}|c_0|\CalV(V)\Bigl)}
	\end{equation}
	where $C = \frac{9}{2\pi}$ and $\CalV(V) = \int_E\Theta\diff\CalL^3$ is the algebraic volume (see \Cref{prop:algebraic_volume}).
\end{lem}

\begin{remark}\label{rem:upper_diameter_bound}
	For $c_0 = 0$, we recover the diameter bound in terms of area and Willmore energy by Simon~\cite[Lemma 1.1]{SimonWillmore}:
	\begin{equation}\label{eq:diam_bound_Willmore}
		\diam \spt \mu_V \leq C \sqrt{\CalW(V)\mu_V(\R^3)}.
	\end{equation}
	This inequality holds true for all $2$-varifolds in $\R^3$ with generalized perpendicular mean curvature, finite Willmore energy, and whose weight measure is finite and has connected support (see \cite[Theorem~1.5]{Scharrer2}). Hence, by~\eqref{eq:bound_Willmore_by_Helfrich} we obtain
	\begin{equation}
		\diam \spt \mu_V \leq C \sqrt{\mu_V(\R^3)\Bigl(2\CalH_{c_0}(V) + \frac{1}{2}c_0^2\mu_V(\R^3)\Bigr)}
	\end{equation}
	for all $V$ satisfying \Cref{hyp:generalised_mean_curvature} with $\CalW(V)<\infty$, $\mu_V(\R^3)<\infty$, and such that $\spt\mu_V$ is connected. % for the same constant $C$ as in~\eqref{eq:diam_bound_Willmore}. 
	In view of \Cref{rem:regularity}, our diameter bound \eqref{eq:upper_diameter_bound} will be particularly useful since it has the Helfrich functional rather than the weight measure as prefactor.
\end{remark}

\begin{proof}[Proof of \Cref{thm:upper_diameter_bound}]
	We will follow the proof of \cite[Lemma~1.1]{SimonWillmore}. Suppose $\spt\mu_V\neq\varnothing$ (otherwise the statement is trivial), let $x_0\in\spt\mu_V$ and define the Radon measure
	\begin{equation}
		\CalH_{c_0}(V,B) \defeq \frac{1}{4}\int_{B\times \mathbb{G}^{\mathrm{o}}(3,2)}|H(x) - c_0(\star\xi)|^2\diff V(x,\xi) \qquad \text{for all Borel sets $B$ in $\R^3$}.
	\end{equation} 
	Using the Cauchy--Schwarz inequality as in~\eqref{eq:gamma 4 to 0} and Young's inequality, we estimate
	\begin{align} 
		\frac{1}{2\rho^2}\int_{B_\rho(x_0)\times\mathbb G^\mathrm{o}(3,2)}\left|\langle x-x_0,H(x) - c_0(\star\xi)\rangle\right| \diff V(x,\xi) %\\ \label{eq:HolderYoung}
		\leq \frac{\mu_V(B_\rho (x_0))}{2\rho^2} + \frac{1}{2}\CalH_{c_0}(V,B_\rho(x_0))
	\end{align}
	for all $\rho>0$. Hence, since $\theta^2(\mu_V,x_0)\geq 1$ by \cite[Theorem~3.6]{Scharrer2} in combination with \Cref{rem:H L2 loc}, we can let $\sigma$ go to zero in \Cref{lem:monotonicity} and use \eqref{eq:gamma 4 to 0}, \eqref{eq:gamma 5 to 0} to infer
	\begin{align}
		\pi &\leq \frac{3}{4}\CalH_{c_0}(V,B_\rho(x_0))+\frac{3\mu_V(B_\rho(x_0))}{2\rho^2} - \frac{c_0}{2}\int_{B_\rho(x_0)\times\mathbb G^\mathrm{o}(3,2)}\frac{\langle x-x_0,\star\xi\rangle}{|x-x_0|^2} \diff V(x,\xi) \\ \label{eq:woGaussGreen}
	 	&\quad + \frac{c_0}{2\rho^2}\int_{B_\rho(x_0)\times\mathbb G^\mathrm{o}(3,2)}\langle x-x_0,\star\xi\rangle \diff V(x,\xi).
	\end{align}

%		Exactly as in \eqref{eq:perimeter proof 1}, for $0<\sigma<\rho$ we may use \eqref{eq:generalized_div_thm} to obtain
		
	Multiplying \eqref{eq:perimeter proof 1} with $\frac{c_0}{2}$ and using $c_0\leq 0$, for any $0<\sigma<\rho$ we have
	\begin{align}
			&- \frac{c_0}{2}\int_{(B_\rho\setminus B_\sigma)(x_0)\times\mathbb G^\mathrm{o}(3,2)}\frac{\langle x-x_0,\star\xi\rangle}{|x-x_0|^2} \diff V(x,\xi) + \frac{c_0}{2\rho^2}\int_{B_\rho(x_0)\times\mathbb G^\mathrm{o}(3,2)}\langle x-x_0,\star\xi\rangle \diff V(x,\xi) \\
		&\quad = \frac{3|c_0|}{2\rho^2}\int_{E\cap B_\rho(x_0)} \Theta\diff\CalL^3 -\frac{\abs{c_0}}{2} \int_{E\cap B_\rho(x_0)\setminus B_\sigma(x_0)}\frac{\Theta(x)}{\abs{x-x_0}^2}\diff\CalL^3(x)\\
		&\qquad - \frac{3\abs{c_0}}{2\sigma^2}\int_{E\cap B_\sigma(x_0)}\Theta\diff\CalL^3 - \frac{\abs{c_0}}{2\sigma^2}\int_{B_\sigma(x_0)\times \mathbb{G}^{\mathrm{o}}(3,2)}\langle x-x_0, \star\xi\rangle\diff V(x,\xi). \label{eq:GGbdr4}
	\end{align}
	Sending $\sigma\to 0+$ and using \eqref{eq:concentrated_Theta_volume}, \Cref{lem:subcrit_integral} and \Cref{rem:subcrit_integral}\eqref{item:subcrit_loc}, we find that
	\begin{align}
		&-\frac{c_0}{2}\int_{B_\rho(x_0)\times\mathbb G^\mathrm{o}(3,2)}\frac{\langle x-x_0,\star\xi\rangle}{|x-x_0|^2} \diff V(x,\xi) + \frac{c_0}{2\rho^2}\int_{B_\rho(x_0)\times\mathbb G^\mathrm{o}(3,2)}\langle x-x_0,\star\xi\rangle \diff V(x,\xi) \\
		&= 
		\frac{|c_0|}{2}\int_{E\cap B_\rho(x_0)}\left(\frac{3}{\rho^2} - \frac{1}{|x-x_0|^2}\right) \Theta(x) \diff \CalL^3(x) \leq \frac{|c_0|}{\rho^2}\int_{E\cap B_\rho(x_0)}\Theta(x)\CalL^3(x). \label{eq:GGbdr5}
	\end{align}
	Combining \eqref{eq:woGaussGreen} and \eqref{eq:GGbdr5}, we thus obtain

	\begin{equation}\label{eq:monotonicity_inequality_rho}
			\pi \leq \frac{3}{4}\CalH_{c_0}(V,B_\rho(x_0))+\frac{3}{2\rho^2} \mu_V(B_\rho(x_0)) + \frac{|c_0|}{\rho^2}\left(\Theta\CalL^3 \llcorner E\right)(B_\rho(x_0)).
	\end{equation} 
	The right hand side of this inequality corresponds to the Radon measure
	\begin{equation}
		\mu_{c_0,V,E} := \frac{3}{4}\CalH_{c_0}(V,\cdot)+\frac{3}{2\rho^2} \mu_V + \frac{|c_0|}{\rho^2}\left(\Theta\CalL^3\llcorner E\right).
	\end{equation}
	The set of $x_0\in\R^3$ that satisfy \eqref{eq:concentrated_Theta_volume} is dense in $\spt\mu_V$. Hence, given any $x_0\in\spt\mu_V$ and any $\varepsilon>0$ we can always find $x_1\in\spt\mu_V$ which satisfies \eqref{eq:monotonicity_inequality_rho} such that $B_\rho(x_1)\subset B_{\rho + \varepsilon}(x_0)$. Thus, letting $\varepsilon \to 0+$, we see that \eqref{eq:monotonicity_inequality_rho} remains valid for all $x_0\in\spt\mu_V$.
	By \Cref{rem:subcrit_integral}\eqref{item:subcrit_loc}, we see $\mu_V(N) = 0$ whenever $N$ is finite and consequently $\mu_{c_0,V,E}(N) = 0$ whenever $N$ is finite. Let $d \defeq \diam \spt \mu_V$ (possibly $d=\infty$), $0<\rho<d$, and $N$ be a positive integer such that $2(N-1)\rho < d$. By the connectedness of $\spt\mu_V$, we can choose points $x_0,\ldots,x_{N-1}\in\spt\mu_V$ such that $x_i \in \partial B_{2i\rho}(x_0)$ for $i=1,\ldots,N-1$. The balls $B_\rho(x_0),\ldots,B_\rho(x_{N-1})$ intersect in at most $N-1$ points. Applying the inequality~\eqref{eq:monotonicity_inequality_rho} for each $x_i$ and summing over $i$ yields
	\begin{equation}\label{eq:diam1}
		N\pi \leq \mu_{c_0,V,E}(\R^3).
	\end{equation}
	Since the right hand side is finite, it follows that the diameter $d$ is finite. Hence, we can choose $N$ such that $2(N-1)\rho<d\leq 2N\rho$. Then~\eqref{eq:diam1} and \Cref{prop:algebraic_volume} imply 
	\begin{equation}\label{eq:diam2}
		d \leq \frac{3}{2\pi}\left(\rho \CalH_{c_0}(V) + \frac{2}{\rho}\mu_V(\R^3) + \frac{4|c_0|}{3\rho}\CalV(V)\right).
	\end{equation} 
	Now, in view of \Cref{thm:Willmore-inequality}, we may take
	\begin{equation}
		\rho = \sqrt{\frac{2\mu_V(\R^3) + \frac{4}{3}|c_0|\CalV(V)}{2\CalH_{c_0}(V)}} = \sqrt{\frac{\mu_V(\R^3) + \frac{2}{3}|c_0|\CalV(V)}{\CalH_{c_0}(V)}}.
	\end{equation}
	%which minimizes the right hand side of \eqref{eq:diam2} with respect to $\rho$.
	Then, by \Cref{thm:diameter_bound}, $\rho < d$ and thus, \eqref{eq:diam2} becomes
	\begin{equation}
		d \leq \frac{9}{2\pi}\sqrt{\CalH_{c_0}(V)\Bigr(\mu_V(\R^3) + \frac{2}{3}|c_0|\CalV(V)\Bigl)}
	\end{equation}
	which concludes the proof.
\end{proof}
%
%\begin{remark}
%	Notice that in \eqref{eq:GGbdr3} the integral
%	\begin{equation}
%	\int_{E\cap B_\rho(x_0)}\left(\frac{3}{\rho^2} - \frac{1}{|x-x_0|^2}\right) \diff \CalL^3(x)
%	\end{equation}
%	is actually zero if $E$ is a half space which can be shown using spherical coordinates. However, if $E$ is a thin annulus around $x_0$ with radii close to $\rho$ and $E$ has a very thin tube connecting from the annulus to the point $x_0$, we see that the inequality~\eqref{eq:GGbdr3} in general cannot be improved.
%\end{remark}

\subsection{Regularity and embeddedness of Canham--Helfrich minimizers}\label{sec:regularity}

We start with a survey of the variational setting in~\cite{RiviereCrelle} %,RiviereNotes} 
(see also \cite{KuwertLiCAG,KellerMondinoRiviere,MondinoScharrer1}). This includes the definition of \emph{Lipschitz immersions}. Then we introduce the space $\mathcal Q_\Sigma$ of \emph{Lipschitz quasi-embeddings} which consists of those Lipschitz immersions whose associated varifolds are varifolds with enclosed volume, cf.\ \Cref{hyp:EmuH}.
%that satisfy a divergence theorem (see \eqref{eq:gauss-green 2}). 
We show that each injective Lipschitz immersion (in particular each smooth embedding) is a Lipschitz quasi-embedding (see \Cref{lem:weak imm finite perim}). Moreover, we prove a weak closure \Cref{lem:weak_closure} which leads to our main regularity \Cref{thm:regularity}.

Let $\Sigma$ be a closed oriented surface %We choose a reference Riemannian metric $g_0$ such that $(\Sigma,g_0)$ has constant sectional curvature and area $4\pi$. 
and let $g_0$ be a reference Riemannian metric on $\Sigma$.
A map $f\colon \Sigma \to \mathbb R^3$ is called \emph{weak branched immersion} if and only if 
\begin{equation}\label{eq:Lip}  
	f \in W^{1, \infty}(\Sigma; \mathbb R^3) \cap W^{2,2}(\Sigma; \mathbb R^3), 
\end{equation} 
there exists a constant $1<C<\infty$ such that
\begin{equation}\label{eq:frame_condition}
	C^{-1}|\mathrm df|^2_{g_0}\leq |\mathrm df\wedge\mathrm df|_{g_0}\leq	C|\mathrm df|^2_{g_0}
\end{equation}
where in local coordinates
\begin{equation}
	\mathrm df\wedge\mathrm df\defeq (\mathrm dx^1\wedge \mathrm dx^2)\partial_{x^1}f\wedge\partial_{x^2}f, %\in ({\textstyle\bigwedge}_2T^*\S^2)\otimes {\textstyle\bigwedge}_2\R^3,
\end{equation}
there exist finitely many so called \emph{branch points} $b_1, \ldots, b_N \in \Sigma$ such that the \emph{conformal factor} satisfies
\begin{equation}\label{eq:conformal_factor}
	\log |\mathrm df|_{g_0} \in L^\infty_\mathrm{loc}(\Sigma \setminus \{b_1, \ldots, b_N\}),
\end{equation}
and the Gauss map $n$ defined as in \eqref{eq:def normal}
%\begin{equation}\label{eq:Gauss-map:regularity}
%	n \defeq \frac{\partial_{x^1} f \times \partial_{x^2} f}{|\partial_{x^1} f \times \partial_{x^2} f|}	
%\end{equation}
%in any positive chart $x$ of $\Sigma$ 
satisfies 
\begin{equation}\label{eq:finie_total_curvature}
	n \in W^{1,2}(\Sigma; \mathbb R^3).
\end{equation}
If in addition 
\begin{equation}\label{eq:conformal}
%\left\{ 
%\begin{split}
|\partial_{x^1} f| = |\partial_{x^2} f| \quad \text{and}\quad 
\langle \partial_{x^1} f, \partial_{x^2} f\rangle = 0
%\end{split} 
%\right. 
\end{equation}
for all conformal charts $x$ of $(\Sigma,g_0)$, then $f$ is called $\emph{conformal}$. A chart $x = (x^1,x^2)$ that satisfies \eqref{eq:conformal} is referred to as \emph{isothermal coordinates}.
Notice that \eqref{eq:conformal} implies \eqref{eq:frame_condition} and, since $\Sigma$ is closed, the conditions~\eqref{eq:Lip}--\eqref{eq:finie_total_curvature} do not depend on the choice of the Riemannian metric $g_0$. The space of weak branched immersions is denoted by $\mathcal F_{\Sigma}$.
The subspace $\CalE_\Sigma$ of \emph{Lipschitz immersions} is defined to consist of all $f\in\CalF_{\Sigma}$ such that there exists a constant $0<C<\infty$ with
\begin{equation} \label{eq:Lipschitz_immersion}
	|\mathrm df\wedge\mathrm df|_{g_0} \geq C.
\end{equation}
Notice that \eqref{eq:Lip} and \eqref{eq:Lipschitz_immersion} imply $\log |\mathrm df|_{g_0}\in L^\infty(\Sigma)$. 
 
Let $f\in\mathcal F_\Sigma$. Analogously to \Cref{ex:varifold_of_immersion}, we infer a (possibly degenerated) $L^\infty$-metric $g\defeq f^*\langle\cdot,\cdot\rangle$, the induced Radon measure $\mu$ over $\Sigma$, the oriented varifold $V\defeq (f,\star n)_\#\mu$, the classical mean curvature $H_f$ of $f$ (in the Sobolev sense), and the induced generalized mean curvature $H$. %(see \eqref{eq:gmc_immersions}).
%By \cite[Theorem~1.4]{RiviereCrelle} there exist local isothermal coordinates for $f$ and 
If $f$ is conformal, we have by \cite[Theorem~3.1]{KuwertLiCAG} that
\begin{equation}\label{eq:density_weka_immersion}
	\mathcal H^0(f^{-1}\{x\}) = \theta^2(\mu_V,x) \qquad\text{for all $x\in\R^3$}.
\end{equation} 
%By \cite[Theorem~1.4]{RiviereCrelle} there exist local isothermal coordinates for $f$. In view of \cite[4.1(1)]{Allard} one can thus show by direct computation (cf. \cite[Lemma~4.1]{MondinoScharrer1}) that
In view of \cite[Equation~(2.11)]{MondinoScharrer1} there holds
\begin{equation}\label{eq:first_variation_weak_immersion}
	\delta V(X) = - \int_{\R^3}\langle X,H\rangle\diff \mu_V 
\end{equation} 
for all $X\in C^1_{c}(\R^3;\R^3)$. %$X\in C^1_{c}(\R^3\setminus\{f(b_1),\ldots ,f(b_N)\};\R^3)$. Since $\Sigma$ is compact and $f$ is a Lipschitz map, there holds $\theta^1_*(\mu_V,x) = 0$ for all $x\in\R^3$. Thus, a cut-off argument shows that \eqref{eq:first_variation_weak_immersion} remains valid for all $X\in C^1_{c}(\R^3;\R^3)$ (cf. \cite[Appendix~A]{KSRemovability}). 
Moreover, by the definition of $H$ and \eqref{eq:finie_total_curvature} we have that
\begin{equation}
	\int_{\R^3}|H|^2\diff\mu_V \leq \int_{\Sigma} |H_f|^2\diff\mu < \infty.
\end{equation}
In particular, $H \in L^2(\mu_V;\R^3)$ and $V,H$ satisfy \Cref{hyp:generalised_mean_curvature}. Now, we can combine \cite[Section~6.1, Theorem~4]{EvansGariepy} and \cite[Theorem~4.1]{SchatzleJDG} to infer
\begin{equation}
	H(f(p)) = H_f(p) \qquad \text{for $\mu$-almost all $p\in\Sigma$}.
\end{equation}
%(cf. also \cite[Theorem~1]{Menne}). 
As in \eqref{eq:H_energy_varifold_vs_immersion}, it follows that for all $c_0 \in \R$ we have
\begin{equation}\label{eq:Helfrich_energies}
	\CalH_{c_0}(V) = \frac{1}{4}\int_{\R^3}|H(x) - c_0(\star\xi)|^2\diff V(x,\xi) = \frac{1}{4}\int_{\Sigma}|H_f - c_0n|^2\diff\mu = \CalH_{c_0}(f).
\end{equation}

The space $\mathcal Q_\Sigma$ is defined to consist of all $f\in\CalE_\Sigma$ such that there exists an $\CalL^3$-measurable set $E$ with
\begin{equation}\label{eq:diameter_condition}
\diam \spt (\CalL^3\llcorner E) \leq \diam f[\Sigma]
\end{equation}
and
\begin{align}\label{eq:gauss-green 2}
	\int_E \divergence X\diff \CalL^3 =- \int_{\Sigma}\langle X\circ f, n\rangle\diff \mu
\end{align}
for any Lipschitz map $X\colon \R^3\to\R^3$ with compact support, i.e.\ the triple $E,V,H$ satisfies \Cref{hyp:EmuH} for $\Theta\equiv1$.
The divergence theorem for sets of finite perimeter~\eqref{eq:div-thm-finite-perimeter}, Equation~\eqref{eq:gauss-green 2}, and the area formula (see \cite[3.2.22(3)]{Federer}) imply
\begin{equation}\label{eq:comparing_normal_vectors}
	n_E(x) = \sum_{p\in f^{-1}\{x\}} n(p) \qquad \text{for $\CalH^2$-almost all $x\in\R^3$.}
\end{equation}
Notice that $x\notin \spt(\CalH^2\llcorner\partial_*E)$ does not imply $f^{-1}\{x\} = \varnothing$. In particular, in view of Figures \ref{fig:pillow} and \ref{fig:current}, the two oriented varifolds associated with $\partial_*E$ and~$f$ do not necessarily coincide. 
%Nevertheless, the triple $E,V,H$ satisfies \Cref{hyp:EmuH} for $\Theta\equiv1$. 
Hence, by \Cref{prop:cvol existence}, \Cref{prop:algebraic_volume}, \Cref{lem:cvol-Alexandrov} and \Cref{rem:cvol-Alexandrov} there holds
\begin{equation}\label{eq:positive_concentrated_volume}
\CalV(V,x_0) = \CalV(f) = \CalL^3(E),\qquad \CalV_c(V,x_0) = \CalV_c(f,x_0) = \int_E\frac{1}{|x-x_0|^2}\diff \CalL^3(x)
\end{equation}
for all $x_0\in\R^3$. If $\CalH^0(f^{-1}\{x\}) \leq 1$ for $\CalH^2$-almost all $x\in\R^3$ then \eqref{eq:comparing_normal_vectors} implies $n_E \circ f = n$, $\partial_*E = f[\Sigma]$ up to a set of $\CalH^2$-measure zero, the two oriented varifolds associated with $\partial_*E$ and $f$ coincide, and since $\spt(\CalH^2\llcorner \partial_*E) \subset \spt(\CalL^3\llcorner E)$, equality holds in \eqref{eq:diameter_condition}. Inspired by the terminology in the smooth case of~\cite{BellettiniWickramasekera}, we refer to $\mathcal Q_\Sigma$ as the space of \emph{Lipschitz quasi-embeddings}. By definition, there holds
\begin{equation}
	\mathcal Q_\Sigma \subset \CalE_\Sigma \subset \CalF_{\Sigma}.
\end{equation}

\begin{lem}\label{lem:weak imm finite perim}
	Let $\Sigma$ be a closed oriented surface and $f\in \mathcal{E}_\Sigma$ be injective. Then, possibly after changing the orientation of $\Sigma$, there exists a connected open bounded set $U\subset \R^3$ of finite perimeter such that $\partial_*U = f[\Sigma]$ up a to a set of $\CalH^2$-measure zero, and
	\begin{equation}\label{eq:gauss-green 3}
		\int_U \divergence X\diff \CalL^3 =- \int_{\Sigma}\langle X\circ f, n\rangle\diff \mu
	\end{equation}
	for any Lipschitz map $X\colon \R^3\to\R^3$.
	In particular,  $f\in \mathcal Q_\Sigma$ and $\mathcal Q_\Sigma$ contains all %$C^{1,1}$-regular 
	smooth embeddings $f\colon \Sigma\to\R^3$ (up to orientation). However, not all $f\in\mathcal Q_\Sigma$ are injective. 
\end{lem}

\begin{remark}\label{rem:weak_imm_finite_perim_pos_vol}
	Notice that changing the orientation on $\Sigma$ is equivalent to changing the sign of the (nonzero) algebraic volume. Hence, if additionally $\CalV(f)>0$,
	%in addition to the assumptions of \Cref{lem:weak imm finite perim}, 
	no change of orientation is necessary in \Cref{lem:weak imm finite perim}.
\end{remark}

\begin{proof}[Proof of \Cref{lem:weak imm finite perim}]
	We may assume that $(\Sigma,g_0)\subset\R^3$ is embedded and $g_0$ is the metric induced by the inclusion map. Since $f$ is injective, we can apply the Jordan--Brouwer separation theorem \cite{Brouwer} to obtain a connected open bounded set $U\subset\R^3$ such that $\partial U = f[\Sigma]$ and $\R^3\setminus \bar U$ is connected. 
	Since $\partial_*U\subset \partial U = f[\Sigma]$ and $\CalH^2(f[\Sigma])<\infty$, Federer's criterion implies that $U$ is a set of finite perimeter. Moreover, for $p\in\Sigma$, one can show that if $f$ is differentiable at $p$, then $f(p)\in\partial_*U$. Hence, by Rademacher's theorem, the sets $\partial_*U$ and $f[\Sigma]$ are $\CalH^2$-almost equal. We still need to show that %(possibly after composing with an orientation reversing diffeomorphism on $\Sigma$)
	\begin{equation}
		\int_{\partial_*U}\langle X,n_U\rangle\diff\CalH^2 = \int_{\Sigma}\langle X\circ f,n\rangle\diff \mu 
	\end{equation}
	for all Lipschitz maps $X\colon\R^3\to\R^3$, where $n_U$ is the measure theoretic inner unit normal of $U$ (see \Cref{sec:finiteperimeter}), and $n$ is the Gauss map of $f$, cf.\ \eqref{eq:def normal}. Let $\nu$ be the unit normal induced by the orientation of $\Sigma\subset \R^3$. We define the $2$-current $T$ on $\R^3$ by
	\begin{equation}
		T(\omega) \defeq  - \int_\Sigma\omega_p(\star\nu(p))\diff \CalH^2(p)
	\end{equation}
	for all differential forms $\omega$ of degree $2$ on $\R^3$. Since $\Sigma$ is closed, we have
	\begin{equation}\label{eq:Stokes}
		\partial T = 0
	\end{equation} 
	(see for instance \cite[4.1.31(1)]{Federer}).
	Given any positive chart $x$ of $\Sigma$, there holds
	\begin{equation}
		\nu = \frac{\partial_{x^1}\times\partial_{x^2}}{|\partial_{x^1}\times\partial_{x^2}|},\qquad ({\textstyle\bigwedge}_2\mathrm df)(\star\nu) =  \frac{|\partial_{x^1}f\wedge\partial_{x^2}f|}{|\partial_{x^1}\times\partial_{x^2}|}(\star n)
	\end{equation}
	where for $\CalH^2$-almost all $p\in\Sigma$, the linear map ${\textstyle\bigwedge}_2\mathrm df_p\colon {\textstyle\bigwedge}_2 T_p\Sigma \to {\textstyle\bigwedge}_2 \mathrm df_p[T_p\Sigma]$ is defined as in \cite[1.3.1]{Federer}. Recalling that in any local chart $x$, the area elements of the immersion $f$ and the inclusion $\Sigma\subset \R^3$ are given by $|\partial_{x^1}f\wedge\partial_{x^2}f|$ and $|\partial_{x^1}\times\partial_{x^2}|$, respectively, we have by \cite[4.1.30]{Federer} that
	\begin{equation}
		R(\omega)\defeq (f_\#T)(\omega) = -\int_{\Sigma}\omega_{f(p)}(\star n(p))\diff \mu(p)
	\end{equation}
	for all differential forms $\omega$ of degree $2$ on $\R^3$. Thus, by \cite[4.1.14]{Federer} and \eqref{eq:Stokes}
	\begin{equation}
		\partial R = \partial (f_\#T) = f_\#(\partial T) = 0.
	\end{equation}
	Therefore, we can combine \cite[4.5.17]{Federer} and \cite[4.5.6]{Federer} to deduce the existence of sets of finite perimeter $E_j\subset E_{j-1}\subset \R^3$, $j\in\Z$ such that 
	\begin{equation}\label{eq:constancy-thm}
		R = \sum_{j\in\Z} R_j,\qquad \mu_V = \sum_{j\in\Z}(\CalH^2\llcorner \partial_*E_j)
	\end{equation}
	where
	\begin{equation}
		R_j(\omega)\defeq - \int_{\partial_*E_j}\omega_x(\star n_{E_j}(x))\diff\CalH^2(x)
	\end{equation}
	are the currents induced by $\partial_*E_j$. Since $U$ is open and connected, we see from \cite[Proposition~2]{ACMM} that $U$ is indecomposable. Given any set of finite perimeter $E\subset \R^3$ with $\partial_*E \subset \partial_*U$ up to a set of $\CalH^2$-measure zero, we see that $\partial_*(U\cap E) \subset \partial_*U$ up to a set of $\CalH^2$-measure zero and thus, by \cite[Proposition~4]{ACMM}, either $\CalL^3(U \cap E) = 0$ or $\CalL^3(U\setminus E)=0$. The same holds true for $U$ replaced by $\R^3\setminus \bar U$. By~\eqref{eq:constancy-thm} we have for all $j\in\Z$ that $\partial_*E_j\subset \spt\mu_V = \partial_*U$ up to a set of $\CalH^2$-measure zero and therefore either $E_j=\R^3$ or $E_j = U$ or $E_j = \R^3\setminus \bar U$ or $E_j = \varnothing$ up to a set of $\CalL^3$-measure zero. Since $f$ is injective, we have that $\theta^2(\mu_V,\cdot)\leq1$. We thus deduce the existence of $j_0\in \Z$ such that (up to a set of $\CalL^3$-measure zero and possibly after changing the orientation on $\Sigma$) %composing with an orientation reversing diffeomorphism on $\Sigma$) 
	\begin{equation}
		E_j = 
		\begin{cases}
			\R^3 &\text{for $j<j_0$},\\
			U &\text{for $j=j_0$}, \\
			\varnothing &\text{for $j>j_0$}.
		\end{cases}
	\end{equation}
	In particular, $R=R_{j_0}$ and~\eqref{eq:gauss-green 3} follows. To see that not all $f\in\mathcal Q_\Sigma$ are injective, one may consider surfaces like in Figures \ref{fig:sausage}, \ref{fig:two_tears}, and \ref{fig:pillow}.
\end{proof}

%\begin{remark}
%	The same proof also works if $f$ is merely Lipschitz and injective rather than $f\in\CalF_{\Sigma}$.
%\end{remark}

In the following, we abbreviate $\CalF\defeq\CalF_{\mathbb S^2}$, 
$\CalE\defeq\CalE_{\mathbb S^2}$, and $\mathcal Q\defeq\mathcal Q_{\mathbb S^2}$.

\begin{lem}\label{lem:weak_closure}
	Suppose $f_k$ is a sequence in $\mathcal Q$, $0\in f_k[\S^2]$ for all $k\in\N$, $c_0\in\R$,
	\begin{equation}\label{eq:lem:uniform_area_bound}
		A_0\defeq\sup_{k\in\N} \CalA(f_k)<\infty, \qquad \inf_{k\in\N}\diam f_k[\S^2] >0,
	\end{equation}	
	and  
	\begin{equation}\label{eq:injectivity_condition}
		\begin{cases}
			\liminf_{k\to\infty}\bigl(\CalH_{c_0}(f_k) +2c_0\inf_{x\in f_k[\S^2]}\CalV_c(f_k,x)\bigr) < 8\pi & \text{if $c_0<0$},\\
			\liminf_{k\to\infty}\bigl(\CalH_{c_0}(f_k) +2c_0\sup_{x\in f_k[\S^2]}\CalV_c(f_k,x)\bigr) < 8\pi & \text{if $c_0\geq0$}.
		\end{cases}
	\end{equation}
	Then, after passing to a subsequence, there exists $f\in\mathcal Q$ injective such that
	\begin{equation} \label{eq:convergence_associated_varifolds}
		V_k\to V\qquad\text{in $\mathbb V_2^\mathrm{o}(\R^3)$ as $k\to\infty$},
	\end{equation}	
	where $V_k,V$ are the oriented $2$-varifolds in $\R^3$ associated with $f_k,f$ (cf.\ \Cref{ex:varifold_of_immersion}) and
	\begin{gather}\label{eq:lem:lsc}
		\CalH_{c_0}(f)\leq \liminf_{k\to\infty}\CalH_{c_0}(f_k). %\\ \label{eq:injectivity_condition_limit}
		%\CalH_{c_0}(f) + 2 c_0\CalV_c(f,x_0) < 8\pi \qquad \text{for all $x_0\in f[\S^2]$}.
	\end{gather}
\end{lem}

\begin{proof}
	Let $g_0$ be the standard metric on $\S^2$. By \cite[Theorem~1.4]{RiviereCrelle}, after reparametrization, we may assume that all $f_k$ are conformal. After passing to a subsequence, we may further assume that for all $k\in\N$,
	\begin{equation}
		\begin{cases}
			\CalH_{c_0}(f_k) +2c_0\inf_{x\in f_k[\S^2]}\CalV_c(f_k,x) < 8\pi & \text{if $c_0<0$},\\
			\CalH_{c_0}(f_k) +2c_0\sup_{x\in f_k[\S^2]}\CalV_c(f_k,x) < 8\pi & \text{if $c_0\geq0$}.
		\end{cases}
	\end{equation}
	%Hence, by \Cref{cor:LY varif area dec}, $\CalH^0(f_k^{-1}\{\cdot\}) = \theta^2(\mu_{V_k},\cdot)\leq 1$.
	Let $E_k$ be the sequence of sets of finite perimeter corresponding to $f_k$ according to \eqref{eq:gauss-green 2}. Using \eqref{eq:positive_concentrated_volume}, for all $x_0\in \R^3$ and $k\in\N$ there holds
	\begin{equation}
		\CalV_c(f_k,x_0) \leq \int_{B_1(x_0)}\frac{1}{|x-x_0|^2}\diff\CalL^3(x) + \CalL^3(E_k) = 4\pi + \CalV(f_k)
	\end{equation}
	and thus, by \eqref{eq:diameter_condition} we can apply the isoperimetric inequality for sets of finite perimeter (see \cite[Corollary~4.5.3]{Federer}) %or \cite[Section 5.6, Theorem 2]{EvansGariepy}) 
	to deduce from the uniform area bound \eqref{eq:lem:uniform_area_bound} that
	\begin{equation}\label{eq:bound_vol_and_cvol}
		V_0\defeq\sup_{k\in\N}\CalV(f_k)<\infty,\qquad \sup_{k\in\N}\sup_{x\in\R^3}\CalV_c(f_k,x)<C(V_0)<\infty.
	\end{equation}
	Hence, by \cite[Equation (2.8)]{MondinoScharrer1} and \eqref{eq:bound_Willmore_by_Helfrich}, there holds
	\begin{align}
		\int_{\S^2}1 + |\mathrm dn_{f_k}|_{g_0}^2\diff \mu_{f_k} & \leq A_0 + 4\CalW(f_k) \leq A_0 + 8\CalH_{c_0}(f_k) + 2c_0^2A_0 \\ \label{eq:uniform_energy_bound} 
		& \leq A_0 + 8\bigl(8\pi + 2|c_0|C(V_0) + c_0^2A_0\bigr)
	\end{align}
	for all $k\in\N$. Therefore, we can apply \cite[Theorem~3.3]{MondinoScharrer1} (see also Theorem~1.5 and Lemma~4.1 in \cite{MondinoRiviereACV}) to infer that after passing to a subsequence, there exist a positive integer $N$ and sequences $\phi_k^1,\ldots,\phi_k^N$ of positive conformal $C^\infty$-diffeomorphisms of $\S^2$ such that for each $i\in\{1,\ldots,N\}$, there exist $f^i\in\mathcal F_{\S^2}$ conformal, $N^i\in \N$, and finitely many points $b^{i,1},\ldots,b^{i,N^i}\in\S^2$ with
	\begin{gather}\label{eq:W22-convergence}
		f^i_k\defeq f_k\circ\phi^i_k\rightharpoonup f^i \qquad \text{weakly in $W^{2,2}_\mathrm{loc}(\S^2\setminus\{b^{i,1},\ldots,b^{i,N^i}\};\R^3)$ as $k\to\infty$},\\ \label{eq:bound_conf_factor}
		\sup_{k\in\N}\|\log |\mathrm df_k^i|_{g_0}\|_{L^\infty_{\mathrm{loc}}(\S^2\setminus\{b^{i,1},\ldots,b^{i,N^i}\})}<\infty.
	\end{gather}
	Moreover, there exist a sequence $\psi_k$ of $C^\infty$-diffeomorphisms of $\S^2$ and $f\in W^{1,\infty}(\S^2;\R^3)$ such that
	\begin{gather} \label{eq:C0-convergence}
		f_k\circ \psi_k \to f \quad\text{in $C^0(\S^2;\R^3)$ as $k\to\infty$},\qquad
		f[\S^2] = \bigcup_{i=1}^Nf^i[\S^2].
	\end{gather}
	Furthermore, there holds
	\begin{equation}\label{eq:lsc}
		\sum_{i=1}^N\mathcal W(f^i) \leq \liminf_{k\to\infty}\mathcal W(f_k),\qquad \sum_{i=1}^N\CalH_{c_0}(f^i)\leq \liminf_{k\to\infty}\CalH_{c_0}(f_k).
	\end{equation}
	Denote with $V^i$ the varifolds associated to $f^i$ and set $V\defeq\sum_{i=1}^NV^i$. 
	In order to show~\eqref{eq:convergence_associated_varifolds}, let $\varphi\colon\R^3\times \mathbb G^\mathrm{o}(3,2)\to\R$ be any continuous function with compact support. Fix an integer $i\in\{1,\ldots,N\}$, choose a conformal chart $x\colon\S^2\setminus\{b^{i,1},\ldots,b^{i,N^i}\}\to\R^2$, and let $K\subset\S^2\setminus\{b^{i,1},\ldots,b^{i,N^i}\}$ be a compact set. Denote by
	\begin{equation}
		\lambda_k^i \defeq \log |\partial_{x^1}f_k^i|,\qquad \lambda^i \defeq \log |\partial_{x^1}f^i|
	\end{equation}
	the conformal factors and recall that the area elements of $f_k^i$ and $f^i$ are given by $e^{2\lambda_k^i}$ and $e^{2\lambda^i}$. Let $n_k^i$, $n_k$, and $n^i$ be the Gauss maps of $f_k^i$, $f_k$, and $f^i$. Following the proof of \cite[Lemma~3.1]{MondinoScharrer1}, we infer that by the weak convergence~\eqref{eq:W22-convergence}, the Rellich--Kondrachov compactness theorem, and the uniform bounds on the conformal factors~\eqref{eq:bound_conf_factor}, after passing to a subsequence,
	\begin{align}\label{eq:Lp-convergence_conf_factor}
		&e^{2\lambda_k^i}\circ x^{-1}\to e^{2\lambda^i}\circ x^{-1} &&\text{in $L^p(x[K])$ as $k\to\infty$ for all $1\leq p<\infty$}, \\
		%&f_k^i\circ x^{-1}\to f^i\circ x^{-1} &&\text{as $k\to\infty$ in $L^p(x[K];\R^3)$ for all $1\leq p<\infty$},\\ 
		\label{eq:pointwise_convergence}
		&f_k^i\circ x^{-1}\to f^i\circ x^{-1} && \text{pointwise almost everywhere on $x[K]$ as $k\to\infty$},\\ \label{eq:pointwise_convergence_Gauss_map}
		&n_k^i\circ x^{-1}\to n^i\circ x^{-1} && \text{pointwise almost everywhere on $x[K]$ as $k\to\infty$}.
	\end{align}
	Hence, since $\varphi$ is continuous, and also the Hodge star operator $\star$ is continuous, 
	\begin{equation}\label{eq:ptwise_convergence}
		\varphi(f_k^i,\star n_k^i)\circ x^{-1} \to \varphi(f^i,\star n^i)\circ x^{-1} \qquad \text{pointwise almost everywhere on $x[K]$}
	\end{equation}
	as $k\to\infty$. Thus, since $\varphi$ is bounded, the dominated convergence theorem and \eqref{eq:Lp-convergence_conf_factor} imply
	\begin{equation}\label{eq:Lp-convergence_on_compacta}
		\left(\varphi(f_k^i,\star n_k^i) e^{2\lambda_k^i}\right)\circ x^{-1}\to \left(\varphi(f^i,\star n^i) e^{2\lambda^i}\right)\circ x^{-1} \qquad \text{in $L^p(x[K])$ as $k\to\infty$}
	\end{equation}
	for any $1\leq p < \infty$. Therefore, inductively passing to a subsequence, we can achieve that for all $k_0\in\N$ and all $k_0\leq k\in\N$, there holds
	\begin{equation}\label{eq:convergence_outside_branch_points_i}
		\int_{x\bigl[\S^2\setminus\bigcup_{j=1}^{N^i}B_{\frac{1}{k_0}}(b^{i,j})\bigr]}\left| \varphi(f_k^i,\star n_k^i) e^{2\lambda_k^i} - \varphi(f^i,\star n^i) e^{2\lambda^i} \right|\circ x^{-1}\diff\CalL^2 \leq \frac{1}{k_0}.
	\end{equation}
	Successively passing to a subsequence, we infer that \eqref{eq:convergence_outside_branch_points_i} holds true simultaneously for all $i\in\{1,\ldots,N\}$. (Notice however that the chart $x$ actually depends on $i$.) Moreover, since $\varphi$ is bounded and by the fact that finite sets have $\mu_{f_i}$-measure zero by \Cref{rem:subcrit_integral}\eqref{item:subcrit_loc},
	there holds
	\begin{equation}
		\lim_{k_0\to\infty}\int_{\bigcup_{j=1}^{N^i}B_{\frac{1}{k_0}}(b^{i,j})}\varphi(f^i,\star n^i)\diff\mu_{f^i} = 0,
	\end{equation}
	for all $i\in \{1, \dots,N\}$.
	Writing $s_k\defeq 1/k$, by \eqref{eq:integral_push_forward} it follows
	\begin{equation}\label{eq:convergence_outside_branch_points}
		\int_{\mathbb G_2^\mathrm{o}(\R^3)}\varphi\diff V^i = \int_{\S^2}\varphi(f^i,\star n^i)\diff\mu_{f^i} = \lim_{k\to\infty}\int_{\S^2\setminus\bigcup_{j=1}^{N^i}B_{s_k}(b^{i,j})}\varphi(f_k^i,\star n_k^i)\diff\mu_{f_k^i}. 
	\end{equation} 
	By the proof of \cite[Theorem~3.3]{MondinoScharrer1} (see also the proof of \cite[Theorem~1.5]{MondinoRiviereACV}), there exist Borel sets $S^{i,j}_k\subset \S^2$ such that (see Equations (3.19) and (3.20) in \cite{MondinoScharrer1})
	\begin{equation}\label{eq:necks_shrink}
		\lim_{k\to\infty}\int_{S^{i,j}_k}1\diff\mu_{f_k^i} = 0 \qquad\text{for all $i\in\{1,\ldots,N\}$ and $j\in\{1,\ldots,N^{i}\}$}
	\end{equation}
	and
	\begin{align}
		\int_{\mathbb G_2^\mathrm{o}(\R^3)}\varphi\diff V_k = \int_{\S^2}\varphi(f_k,\star n_k)\diff\mu_{f_k} &= \sum_{i=1}^N \int_{\S^2\setminus\bigcup_{j=1}^{N^i}B_{s_k}(b^{i,j})}\varphi(f_k^i,\star n_k^i)\diff\mu_{f_k^i} \\ \label{eq:domain_decomposition}
		&\quad + \sum_{i=1}^N\sum_{j=1}^{N^i-1} \int_{S_k^{i,j}}\varphi(f_k^i,\star n_k^i)\diff\mu_{f_k^i}.
	\end{align}
	By \eqref{eq:necks_shrink} and the boundedness of $\varphi$, there holds
	\begin{equation}
		\left|\int_{S_k^{i,j}}\varphi(f_k^i,\star n_k^i)\diff\mu_{f_k^i}\right| \leq \|\varphi\|_{C^0(\mathbb G_2^\mathrm{o}(\R^3))}\int_{S_k^{i,j}}1\diff\mu_{f_k^i}\to 0\qquad\text{as $k\to\infty$}.
	\end{equation} 
	Thus, \eqref{eq:domain_decomposition} and \eqref{eq:convergence_outside_branch_points} imply 
	\begin{equation}
		\lim_{k\to\infty}\int_{\mathbb G_2^\mathrm{o}(\R^3)}\varphi\diff V_k = \sum_{i=1}^N\int_{\mathbb G_2^\mathrm{o}(\R^3)}\varphi\diff V^i = \int_{\mathbb G_2^\mathrm{o}(\R^3)}\varphi\diff V
	\end{equation}
	which proves \eqref{eq:convergence_associated_varifolds}.
	
	By~\eqref{eq:uniform_energy_bound} there holds
	\begin{equation}\label{eq:uniform_energy_bound2}
		D_0\defeq \sup_{k\in\N}\CalW(f_k) < \infty.
	\end{equation}	
	Thus, by \Cref{lem:continuity}, there exists a constant $C(A_0,D_0)$ depending only on the energy bound $D_0$ and the area bound $A_0$ in \eqref{eq:lem:uniform_area_bound} such that
	\begin{equation}\label{eq:uniformly_continuous}
		|\CalV_c(f_k,x) - \CalV_c(f_k,y)|\leq C(A_0,D_0)|x-y|^{1/2}\qquad\text{for all $k\in\N$ and all $x,y\in\R^3$.} 
	\end{equation} 
	Hence, by the varifold convergence~\eqref{eq:convergence_associated_varifolds}, we can apply \Cref{lem:varifold_convergence} and \eqref{eq:C0-convergence} to deduce first
	\begin{equation}\label{eq:non-negative_cvolume}
		\lim_{k\to\infty} \CalV_c(V_k,f_k(p)) = \CalV_c(V,f(p)) \qquad \text{for all $p\in\S^2$}
	\end{equation}
	and secondly, by \eqref{eq:injectivity_condition}, the lower semi-continuity~\eqref{eq:lsc}, and \eqref{eq:Helfrich_energies}
	\begin{equation}\label{eq:proof:injectivity_condition}
			\CalH_{c_0}(V) + 2 c_0\CalV_c(V,x_0) < 8\pi \qquad \text{for all $x_0\in\spt\mu_V$}.
	\end{equation}
	Therefore, we can apply the Li--Yau inequality for general varifolds \Cref{cor:LY varif area dec} to infer $\theta^2(\mu_V,\cdot)< 2$. %$\theta^2(V,x)\in\{0,1\}$ for all $x\in\R^3$. 
	Now, it follows from \eqref{eq:C0-convergence} that $f=f^1\in\CalF$ and $f$ is injective. In particular, \eqref{eq:lem:lsc} follows from \eqref{eq:lsc}. Moreover, by \cite[Theorem~3.1]{KuwertLiCAG}, $f$ has no branch points. That is $\log |\mathrm d f|_{g_0} \in L^\infty(\S^2)$ and thus $f\in\CalE$. 	
	It remains to show that $f\in \mathcal Q$. 
	Recalling that $\{ x \in \R^3 \mid n_{E_k}(x)\neq 0\} = \partial_* E_k$ up to a set of $\CalH^2$-measure zero, we see from \eqref{eq:comparing_normal_vectors} that $\partial_* E_k\subset f_k[\S^2]$ up to a set of $\CalH^2$-measure zero, and thus $\CalH^{2}(\partial_* E_k) \leq \CalA(f_k)$. Hence, the uniform area bound \eqref{eq:lem:uniform_area_bound} and the uniform volume bound \eqref{eq:bound_vol_and_cvol} imply that the sequence $\chi_{E_k}$ is bounded in $BV(\R^3)$.
%	Since $\CalH^{2}(\partial_* E_k) \leq \CalA(f_k)$ \changeFab{by \eqref{eq:comparing_normal_vectors}, the uniform area bound \eqref{eq:lem:uniform_area_bound} and the uniform volume bound \eqref{eq:bound_vol_and_cvol} imply} that the sequence $\chi_{E_k}$ is bounded in $BV(\R^3)$.
	Therefore, by compactness (see \cite[Section 5.2, Theorem 4]{EvansGariepy}), there exists an $\CalL^3$-measurable set $E$ of of finite perimeter such that, after passing to a subsequence, $\chi_{E_k}\to \chi_E$ in $L^1(\R^3)$ and pointwise almost everywhere as $k\to\infty$. In particular, the left hand side in~\eqref{eq:gauss-green 2} converges as $k\to\infty$. Moreover, the right hand side of~\eqref{eq:gauss-green 2} converges by \eqref{eq:integral_push_forward} as a consequence of the varifold convergence~\eqref{eq:convergence_associated_varifolds}. Noting that $\CalL^3\llcorner E_k \to \CalL^3\llcorner E$ as Radon measures for $k\to\infty$, we see that by \eqref{eq:diameter_condition} and the $C^0$-convergence~\eqref{eq:C0-convergence}
	\begin{align}
		\diam \spt (\CalL^3\llcorner E) \leq \liminf_{k\to\infty} \diam f_k[\S^2] \leq \diam f[\S^2].
	\end{align}
	Thus, $f\in\mathcal Q$ and the proof is concluded. 
\end{proof}

\begin{remark}\label{rem:proof:regularity}
	The minimizer in \cite[Theorem~1.7]{MondinoScharrer1} has positive algebraic volume $V_0$. However, in view of \Cref{ex:volume} this is in general not enough to deduce that also the concentrated volume is nonnegative. Thus, we could not apply the Li--Yau inequality \Cref{cor:LY varif area dec} directly to the minimizer in \cite[Theorem~1.7]{MondinoScharrer1}. 
\begin{comment}		
	We also point out that in order to prove \eqref{eq:convergence_associated_varifolds} one could have simply applied \cite[Theorem~3.1]{Hutchinson}. However, lower semi-continuity of the Helfrich functional under the convergence of oriented varifolds is still an open problem. Notice that in \cite{EichmannAGAG} lower semi-continuity is only proven for minimizing sequences of smooth immersion. However, since the limit $f$ of \Cref{lem:weak_closure} is only an element of $\mathcal Q$, in order to show that $f$ satisfies the Euler--Lagrange equation, it is crucial to also have the minimizing sequence in the same class. Thus, we actually did rely on the lower semi-continuity~\eqref{eq:lem:lsc} taken from \cite[Theorem~3.3]{MondinoScharrer1}. Hence, if \eqref{eq:convergence_associated_varifolds} was simply deduced from \cite[Theorem~3.1]{Hutchinson}, the difficulty would then be to prove that the oriented varifold associated with the limit $f$ of \cite[Theorem~3.3]{MondinoScharrer1} coincides with the limit varifold $V$ of \cite[Theorem~3.1]{Hutchinson}.
\end{comment}
\end{remark}

\begin{thm}\label{thm:regularity}
	Suppose $c_0\in\R$, the numbers $A_0,V_0>0$ satisfy the isoperimetric inequality $36\pi V_0^2\leq A_0^3$, and there exists a minimizing sequence $f_k$ of 
	\begin{equation}\label{eq:infimum}
		\bar\eta(c_0,A_0,V_0)\defeq\inf\{\CalH_{c_0}(f)\mid f\in\mathcal Q,\,\CalA(f) = A_0,\,\CalV(f)=V_0\}
	\end{equation}
	such that
	\begin{equation}
		\begin{cases}\label{eq:thm:injectivity_condition_negative}
			\liminf_{k\to\infty}\bigl(\CalH_{c_0}(f_k) +2c_0\inf_{x\in f_k[\S^2]}\CalV_c(f_k,x)\bigr) < 8\pi & \text{if $c_0<0$},\\
			\liminf_{k\to\infty}\bigl(\CalH_{c_0}(f_k) +2c_0\sup_{x\in f_k[\S^2]}\CalV_c(f_k,x)\bigr) < 8\pi & \text{if $c_0\geq0$}.
		\end{cases}
	\end{equation}
	Then the infimum is attained by a smooth embedding $f\colon \S^2\to\R^3$.
\end{thm}

\begin{remark}\label{rem:regularity}
	\begin{enumerate}[(i)]
		\item\label{item:cvol_estimate_c0_neg} In view of \eqref{eq:diameter_condition}, \eqref{eq:positive_concentrated_volume} and \Cref{thm:upper_diameter_bound}, we see that if $c_0\leq0$, then
		\begin{equation}
			\inf_{x_0\in f[\S^2]} \CalV_c(f,x_0) \geq \frac{(2\pi)^2\CalV(f)}{9^2(\CalA(f) + \frac{2}{3}|c_0|\CalV(f))}\frac{1}{\CalH_{c_0}(f)}
		\end{equation}
		for all $f\in\mathcal Q$. Thus, an elementary computation shows that \eqref{eq:thm:injectivity_condition_negative} is satisfied provided
%		\begin{equation}
%			\bar\eta(c_0,A_0,V_0) < 8\pi\left(\frac{1}{2} + \frac{\sqrt{1 + L(c_0,A_0,V_0)}}{2}\right)
%		\end{equation}
	\begin{equation}
		\bar\eta(c_0,A_0,V_0) < 4\pi\left(1+ {\sqrt{1 + L(c_0,A_0,V_0)}}\right)
	\end{equation}
		for 
		\begin{equation}
			L(c_0,A_0,V_0)\defeq \frac{|c_0|V_0}{2\cdot 9^2(A_0 + \frac{2}{3}|c_0|V_0)} > 0.
		\end{equation}
	
		\item\label{item:cvol_estimate_c0_pos} Using \eqref{eq:gauss-green 2}, \eqref{eq:smiley} and \eqref{eq:positive_concentrated_volume}, for all $r>0$ and $f\in \mathcal Q$ we have
		\begin{equation}
			\sup_{x_0\in f[\S^2]}\CalV_c(f,x_0)=\sup_{x_0\in f[\S^2]}\int_{E}\abs{x-x_0}^{-2}\diff \CalL^3(x)\leq 4\pi r + r^{-2}\CalV(f).
		\end{equation}
		Minimizing over $r>0$ yields the estimate $\CalV_c(f,x_0) \leq 3(4\pi^2 \CalV(f))^{\frac{1}{3}}$.
		Thus, \eqref{eq:thm:injectivity_condition_negative} is satisfied for $c_0>0$ provided
		\begin{equation}
			\bar\eta(c_0,A_0,V_0) < 8\pi - 6c_0(4\pi^2 V_0)^{\frac{1}{3}}.
		\end{equation}
	
		\item \label{item:Schygulla} For all $c_0\leq 0$ and $\sigma\geq36\pi$, there exists $\bar A_0, \bar V_0>0$ such that $\bar A_0^3/\bar V_0^2=\sigma$ and $\bar\eta(c_0,A_0,V_0) <8\pi$ for all $0< A_0<\bar A_0$, $0<V_0<\bar V_0$ with $A_0^3/V_0^2=\sigma$. Indeed, in view of \eqref{eq:intro:convergence_Helfrich_to_Willmore}, this is a consequence of \cite[Lemma~1]{Schygulla}.
	\end{enumerate}
\end{remark}

\begin{proof}[Proof of \Cref{thm:regularity}]
	By \eqref{eq:uniform_energy_bound}, we have that
	\begin{equation}
		\sup_{k\in\N} \CalW(f_k) \leq C(c_0,A_0,V_0) < \infty.
	\end{equation}
	Hence, by \Cref{thm:diameter_bound} applied for $c_0 = 0$, there holds $\inf_{k\in\N} \diam f_k[\S^2] > 0$. 
	Moreover, after translations, we may assume $0\in f_k[\S^2]$ for all $k$. Therefore, we can apply \Cref{lem:weak_closure} to obtain $f\in\mathcal Q$ injective such that, after passing to a subsequence,
	\begin{equation}
		V_k \to V \qquad \text{in $\mathbb V_2^\mathrm{o}(\R^3)$ as $k\to\infty$},
	\end{equation}
	where $V_k,V$ are the oriented $2$-varifolds in $\R^3$ associated with $f_k,f$. The varifold convergence implies $\CalA(f) = A_0$ and $\CalV(f) = V_0$. Thus, by~\eqref{eq:lem:lsc}, $f$ attains the infimum~\eqref{eq:infimum}.
	
	Let $\omega\in C^\infty(\S^2,\R^3)$ and define $f_t \defeq f + t\omega$ for $t\in \R$. By \eqref{eq:Lip} and \eqref{eq:Lipschitz_immersion} we have
	\begin{align}
		&f_t \to f && \text{in $W^{1,\infty}(\S^2;\R^3)\cap W^{2,2}(\S^2;\R^3)$ as $t\to0$},\\
		&\mathrm df_t\wedge \mathrm df_t \to \mathrm df \wedge \mathrm df && \text{in $L^\infty(\S^2;({\textstyle\bigwedge}_2T^*\S^2)\otimes {\textstyle\bigwedge}_2\R^3)$ as $t\to0$},\\
		&n_t\to n&&\text{in $L^\infty(\S^2;\R^3)$ as $t\to0$}
	\end{align}
	and the associated varifolds converge in $\mathbb V_2^\mathrm{o}(\R^3)$. Moreover, it follows that $f_t\in\CalE$ for $|t|$ small and $\CalW(f_t) \to \CalW(f)$, $\CalH_{c_0}(f_t) \to \CalH_{c_0}(f)$, and $\CalA(f_t) \to \CalA(f)$ as $t\to0$. %(see \cite{RiviereCrelle}, \cite[Lemma~4.5]{KellerMondinoRiviere}, \cite[Lemma~4.1]{MondinoScharrer1})
	Hence, similarly as in \eqref{eq:uniformly_continuous} and \eqref{eq:non-negative_cvolume}, we can combine \Cref{lem:varifold_convergence} and \Cref{lem:continuity} to deduce that for some $\varepsilon>0$ there holds
	\begin{equation}
		\CalH_{c_0}(f_t) + 2 c_0\CalV_c(f_t,x_0) < 8\pi \qquad \text{for all $|t|<\varepsilon$ and $x_0\in f_t[\S^2]$}.
	\end{equation}
	It follows by \Cref{cor:LY varif area dec} that $f_t$ is injective for $|t|<\varepsilon$ and thus, by \Cref{lem:weak imm finite perim} and \Cref{rem:weak_imm_finite_perim_pos_vol}, $f_t\in\mathcal Q$. Therefore, we can proceed as in \cite{RiviereCrelle} and \cite{MondinoScharrer1} to deduce that $f$ satisfies the Euler--Lagrange equation given in \cite[Lemma~4.1]{MondinoScharrer1}. Now, we can apply \cite[Theorem~4.3]{MondinoScharrer1} to conclude that $f$ is smooth.
\end{proof}

\Cref{thm:regularity_Helfrich_problem} is now a direct consequence.

\begin{proof}[Proof of \Cref{thm:regularity_Helfrich_problem}]
	
	We will prove the theorem in the case where
	\begin{align}
		\Gamma(c_0, A_0, V_0) = \begin{cases}
			4\pi\left(\sqrt{1+L(c_0, A_0, V_0)}-1\right) & \text{ if }c_0< 0, \\
			-6 c_0 \left(4\pi^2 V_0\right)^{\frac{1}{3}} &\text{ if }c_0\geq 0,
		\end{cases}
	\end{align}
	with $L(c_0, A_0, V_0)$ as in \Cref{rem:regularity} (i).
		 %Clearly, $\Gamma_0(c_0, A_0, V_0)>0$ for $c_0\neq 0$.
	Let $f_k \in \mathcal Q$ be a minimizing sequence for \eqref{eq:infimum}. By \Cref{rem:regularity}\eqref{item:cvol_estimate_c0_neg} and \eqref{item:cvol_estimate_c0_pos}, the choice of $\Gamma$, and since $\bar\eta(c_0, A_0, V_0)\leq \eta(c_0, A_0, V_0)$ as a consequence of \Cref{lem:weak imm finite perim}, we find that \eqref{eq:thm:injectivity_condition_negative} is satisfied and hence the infimum \eqref{eq:infimum} is attained by a smooth embedding $f\colon \S^2\to\R^3$, which implies that $f$ is also a minimizer for \eqref{eq:def eta} and thus $\bar\eta(c_0, A_0, V_0)=\eta(c_0, A_0, V_0)$. The last part of \Cref{thm:regularity_Helfrich_problem} follows from \Cref{rem:regularity}\eqref{item:Schygulla}.
\end{proof}

\subsection{Positive total mean curvature}\label{subsec:pos tot mean}

We recall the following inequality due to Minkowski \cite{Minkowski1903}. If $\Omega\subset \R^3$ is a bounded convex open subset with $C^2$-boundary $\partial \Omega$, then
\begin{align}\label{eq:minkowski}
	\frac{1}{2}\int_{\partial\Omega} H_{\mathrm{sc}} \diff \CalH^2\geq \sqrt{4\pi \CalH^2(\partial\Omega)},
\end{align}
with equality if and only if $\Omega$ is a ball. The quantity on the left hand side of \eqref{eq:minkowski} is called \emph{total (scalar) mean curvature}. With the help of \Cref{thm:LY scale inv}, we can generalize \eqref{eq:minkowski} to a class of nonconvex surfaces.

\begin{thm}\label{thm:minkowski}
	Let $f\colon \Sigma\to\R^3$ be an immersion of an oriented closed surface $\Sigma$. If there exists $x_0\in \R^3$ with $\CalV_c(f,x_0)>0$ and $\bar{\CalH}(f) \leq 4\pi \CalH^{0}(f^{-1}\{x_0\})$,
	then we have
	\begin{align}\label{eq:LY minkowski}
		\frac{1}{2}\int H_{\mathrm{sc}}\diff\mu \geq \sqrt{\left(4\pi \CalH^{0}(f^{-1}\{x_0\})-\bar{\CalH}(f)\right)\CalA(f)}.
	\end{align}
\end{thm}
The assumption $\CalV_c(f,x_0)>0$ is especially satisfied if $f$ is an Alexandrov immersion and $x_0\in \R^3$ is arbitrary, see \eqref{eq:cvol Alexandrov}.

We would like to point out that it is possible to deduce \eqref{eq:LY minkowski} with the absolute value on the left hand side from the classical Li--Yau inequality for the Willmore energy. However, the question whether the total mean curvature is positive remains. In general, this has to be answered in the negative; however, under certain convexity or symmetry assumptions on the surface, the total mean curvature can be shown to be positive, cf.\ \cite[Table 1]{DHMT16}. In the case of Alexandrov immersions, \Cref{thm:minkowski} provides a sufficient criterion for positive total mean curvature if the CMC-deficit is not too large, depending on the concentrated volume and the multiplicity at a point.

\begin{proof}[Proof of \Cref{thm:minkowski}]
	Set $\delta\defeq  4\pi \CalH^{0}(f^{-1}\{x_0\})-\bar{\CalH}(f)\geq 0$. By \Cref{thm:LY scale inv}
	we have
	\begin{align}
		\delta\CalA(f) \leq {2\int_{\Sigma} H_{\mathrm{sc}}\diff \mu~ \CalV_{c}(f,x_0)} - 4\CalV_{c}(f, x_0)^2,
	\end{align}
	and therefore, using Young's inequality and $\CalV_c(f,x_0)>0$, we find
	\begin{align}
		\frac{\int H_{\mathrm{sc}}\diff\mu}{2} \geq \frac{\delta \CalA(f)}{4 \CalV_c(f, x_0)} +\CalV_c(f,x_0) \geq \sqrt{\delta\CalA(f)}. &\qedhere
	\end{align}
\end{proof}

	\section*{Acknowledgments}
	F.R. is supported by the Austrian Science Fund (FWF) project/grant P 32788-N.
	C.S. is supported by the Max Planck Institute for Mathematics Bonn. The authors would like to thank Ulrich Menne and Sascha Eichmann for helpful discussions. Moreover, the authors are grateful to the referee
		for their valuable comments on the original manuscript.
	
	\bibliography{Lib}
	\bibliographystyle{abbrv}
\end{document}